\documentclass[a4paper]{amsart}

\usepackage{amsmath, amssymb, amsthm}
\usepackage[margin=1.0in]{geometry}
\usepackage{dsfont}
\usepackage{mathtools}
\usepackage{xcolor}
\usepackage{graphicx}
\usepackage{cleveref}
\usepackage[enableskew]{youngtab}

\newcommand\defeq{\stackrel{\mathclap{\normalfont\mbox{def}}}{=}}

\newcommand{\mone}{\mbox{-1}}
\newcommand{\mtwo}{\mbox{-2}}
\newcommand{\mthree}{\mbox{-3}}

\DeclarePairedDelimiter{\ceil}{\lceil}{\rceil}

\DeclarePairedDelimiter{\floor}{\lfloor}{\rfloor}

\theoremstyle{plain}
\newtheorem{thm}{Theorem}
\numberwithin{thm}{subsection}

\newtheorem{lem}[thm]{Lemma}
\newtheorem{prop}[thm]{Proposition}
\newtheorem{cor}[thm]{Corollary}
\newtheorem{rmk}[thm]{Remark}

\theoremstyle{definition}
\newtheorem{defn}[thm]{Definition}

\newtheorem{ex}[thm]{Example}
\newtheorem{exs}[thm]{Examples}

\usepackage{tikz}
\tikzstyle{V}=[draw, fill =black, circle, inner sep=0pt, minimum size=3pt]

\newcounter{r}
\newcommand\Part[1]{
        \setcounter{r}{1}
	 \foreach \x in {#1}{
 	{\ifnum\value{r}=1
		\draw (0,\value{r}-1)--(\x,\value{r}-1); 
		\fi}
	\draw (0,\value{r}) to (\x,\value{r});
   	\foreach \y in {0, ..., \x} {\draw (\y,\value{r})--(\y,\value{r}-1);}
	\addtocounter{r}{1}
 }}
 \def\PartUNIT{.175}
\newcommand{\PART}[1]{
\begin{matrix}
\begin{tikzpicture}[xscale=\PartUNIT, yscale=-\PartUNIT] 
	\Part{#1}
\end{tikzpicture}
\end{matrix}
}

\begin{document}

\title{The Center of the Partition Algebra}

\author{Samuel Creedon}

\maketitle

\begin{abstract}
In this paper we show that the center of the partition algebra $\mathcal{A}_{2k}(\delta)$, in the semisimple case, is given by the subalgebra of supersymmetric polynomials in the normalised Jucys-Murphy elements. For the non-semisimple case, such a subalgebra is shown to be central, and in particular it is large enough to recognise the block structure of $\mathcal{A}_{2k}(\delta)$. This allows one to give an alternative description for when two simple $\mathcal{A}_{2k}(\delta)$-modules belong to the same block.
\end{abstract}

\section{Introduction}

\subsection*{Background}

Let $\mathbb{C}S_{n}$ be the group algebra of the symmetric group on $n$ letters over the field of complex numbers $\mathbb{C}$. The representation theory of $\mathbb{C}S_{n}$ has been extensively studied and is well-understood. A particularly illuminating approach to the study of the representation theory of $\mathbb{C}S_{n}$ was given by Okounkov and Vershik in \cite{OV96}, where they focused on two main features. The first of such was the chain of algebras
\[ \mathbb{C}=\mathbb{C}S_{0} \subset \mathbb{C}S_{1} \subset \mathbb{C}S_{2} \subset \dots. \]
The idea was to study the representation theory of $\mathbb{C}S_{n}$ for all $n$ simultaneously by understanding the restriction rules in the above chain. Notably, this chain was shown to be multiplicity-free, meaning for any simple $\mathbb{C}S_{n}$-module $M$, the multiplicity of any simple $\mathbb{C}S_{n-1}$-module in the restriction $M\downarrow_{\mathbb{C}S_{n-1}}$ is at most one. The information of the restriction rules in the above chain can be encoded in the form of a directed graph, called the branching graph of the chain. From such a graph one can construct a unique basis (up to scalars) for any simple $\mathbb{C}S_{n}$-module indexed by certain paths in the branching graph, referred to as the \emph{Gelfand-Zetlin basis}, or GZ-basis for short. The branching graph thus gives a useful framework to describing the representation theory for all the symmetric groups, and the GZ-basis allows one to fit any simple module within this framework in a compatible way. The second main feature in the work of Okounkov and Vershik is the role played by the Jucys-Murphy elements, or JM-elements for short. These are a sequence of commuting elements defined for each $k\geq 1$ by
\[ L_{k} = \sum_{1\leq j<k}(j,k),\]
where $(j,k)$ denotes the transposition exchanging $j$ and $k$. These elements are well-adapted to the multiplicity-free chain above. For example, we get a new JM-element for each step in the chain and this element commutes with the smaller algebras, that is $L_{k} \in \mathbb{C}S_{k}$ for all $k\geq 1$ and $L_{k}\pi = \pi L_{k}$ for all $\pi \in \mathbb{C}S_{k-1}$. Also the GZ-basis for any simple $\mathbb{C}S_{n}$-module diagonalises the action of the JM-elements $L_{1},L_{2},\dots,L_{n} \in \mathbb{C}S_{n}$. Hence for each GZ-basis element of any simple $\mathbb{C}S_{n}$-module we have an associated $n$-tuple of eigenvalues corresponding to each JM-element. Giving an explicit description of such tuples would in turn allow one to give an explicit description of the branching graph, from which much of the classical results of the representation theory of $\mathbb{C}S_{n}$ are recovered. This is precisely what was done in \cite{OV96}, which gave an alternative means of showing that the branching graph was isomorphic to Young's lattice, that the paths indexing the GZ-basis correspond to standard Young tableau, and that the eigenvalues of the JM-elements for a given GZ-basis element correspond to the ordered content of boxes in the associated standard Young tableau. This approach thus gave quite a natural explanation of the involvement of Young diagrams and Young tableau within the theory.

A result of particular interest for this paper is the fact that the JM-elements of $\mathbb{C}S_{n}$ can be used to give an alternative description of the center $Z(\mathbb{C}S_{n})$. Since we are working with a group algebra, it is well known that the center can be given in terms of class sums, however it was also shown in \cite{Jucys74} and \cite{Murphy83} that the center $Z(\mathbb{C}S_{n})$ equals the subalgebra of symmetric polynomials in the JM-elements. That is for any $k\geq 1$, let
\[ p_{k}(x_{1},\dots,x_{n}) := x_{1}^{k}+x_{2}^{k}+\dots+x_{n}^{k},\]
be the $k$-th symmetric power-sum polynomial in $\mathbb{C}[x_{1},\dots,x_{n}]$. These generate the algebra of symmetric polynomials in $n$ variables. Then we have that $Z(\mathbb{C}S_{n}) = \langle p_{k}(L_{1},\dots,L_{n}) \hspace{1mm} | \hspace{1mm} k \geq 1 \rangle$. This result on the center was later generalised to the affine degenerate Hecke algebras (see for example \cite{Kleshchev05}).

Many of the features of the representation theory of $\mathbb{C}S_{n}$ described above, including the result on the center, have been generalised to similar algebras. These algebras arise from variations of the classical Schur-Weyl duality. Let $V$ be the defining representation of the general linear group $GL_{n}(\mathbb{C})$, then the $r$-fold tensor product $V^{\otimes r}$ has a natural action of $GL_{n}(\mathbb{C})$ given diagonally, and a natural action of $\mathbb{C}S_{r}$ by permuting the tensor components. The classical Schur-Weyl duality given in \cite{Weyl46} shows that the images of these actions are centralisers of one another. If one replaces $GL_{n}(\mathbb{C})$ with the orthogonal (respectively the symplectic) subgroup, then its centraliser can be described by the Brauer algebra $B_{n}(\delta)$ with parameter $\delta = n$ (respectively $\delta=-n$). If instead one replaces $V^{\otimes r}$ with the mixed tensor $V^{\otimes r}\otimes(V^{*})^{\otimes s}$ of the natural representation and its dual for $GL_{n}(\mathbb{C})$, then the resulting centraliser algebra is given by the walled-Brauer algebra $B_{r,s}(\delta)$ with $\delta=n$ (see for example \cite{BCTLLS94}). Both $B_{r}(\delta)$ and $B_{r,s}(\delta)$ can be defined for arbitrary parameters $\delta \in \mathbb{C}$.  

The Brauer algebras $B_{r}(\delta)$ for $r \geq 0$ fit into a multiplicity-free chain which has a corresponding branching graph, from which a GZ-basis can be established for each simple $B_{r}(\delta)$-module. Also there are analogous elements $L_{1},\dots,L_{r} \in B_{r}(\delta)$ which play the role of the JM-elements of $\mathbb{C}S_{n}$, which are often given the same name, and we use the same notation to denote them here. These JM-elements were first introduced by Nazarov in \cite{Nazarov96}. The GZ-basis diagonalised the action of the JM-elements on the simple modules, and the eigenvalues were given an explicit combinatorial description comparable to the symmetric group case. It was also shown that the center $Z(B_{r}(\delta))$, in the semisimple case, can be described by the symmetric odd-power-sum polynomials in the JM-elements. That is, letting $p_{k}(x_{1},\dots,x_{r})$ as above, we have that $Z(B_{r}(\delta)) = \langle p_{2k-1}(L_{1},\dots,L_{r}) \hspace{1mm} | \hspace{1mm} k \geq 1 \rangle$.  Nazarov also used these JM-elements to define an affine degenerate version of the Brauer algebra, where an analogous result was given regarding the center. This result was later generalised to include the quantised case in \cite{DVR11}.

The walled-Brauer algebras $B_{r,s}(\delta)$ for $r,s \geq 0$ fit into a ``multiplicity-free lattice'' as the $r$ and $s$ vary (Theorem 3.3 of \cite{CDDM08}). As such there is no one natural chain of algebras to consider in this case, one has to chose a particular chain to work with from the lattice. Generalisations of JM-elements have been defined in \cite{BS12} and \cite{JK17}. In the latter, they were denoted by $L_{1},\dots,L_{r},L_{r+1},\dots,L_{r+s} \in B_{r,s}(\delta)$, and a particular chain of algebras was defined which was multiplicity-free, induced a GZ-basis for each simple $B_{r,s}(\delta)$-module, and is compatible with the JM-elements in an analogous manner to the other cases. The center of $B_{r,s}(\delta)$, in the semisimple case, was shown to be generated by supersymmetric polynomials in the JM-elements. That is to say we have the following: let $p_{k}(x_{1},\dots,x_{r})$ and $p_{k}(y_{1},\dots,y_{s})$ be the $k$-th symmetric power-sum polynomials in $\mathbb{C}[x_{1},\dots,x_{r}]$ and $\mathbb{C}[y_{1},\dots,y_{s}]$ respectively. Then
\[ q_{k}(x_{1},\dots,x_{r};y_{1},\dots,y_{s}) := p_{k}(x_{1},\dots,x_{r}) + (-1)^{k+1}p_{k}(y_{1},\dots,y_{s}) \]
are the $k$-th power-sum supersymmetric polynomials. These generate the algebra of supersymmetric polynomials. It was shown in \cite{JK17} that $Z(B_{r,s}(\delta)) = \langle q_{k}(L_{1},\dots,L_{r};L_{r+1},\dots,L_{r+s}) \hspace{1mm} | \hspace{1mm} k \geq 1 \rangle$.

\subsection*{The Partition Algebra}

The partition algebras $\mathcal{A}_{2k}(\delta)$, for $k \in \mathbb{Z}_{\geq 0}$ and $\delta \in \mathbb{C}$, are a family of algebras which were first introduced in the works of Martin \cite{Martin91, Martin96} and Jones \cite{Jones94} regarding connections with Potts models in statistical mechanics. In the classical Schur-Weyl duality, if one restricts the action of $GL_{n}(\mathbb{C})$ on $V^{\otimes k}$ to the subgroup of permutation matrices, which can be identified with the symmetric group $S_{n}$, then the centraliser algebra is isomorphic to the partition algebra $\mathcal{A}_{2k}(n)$ whenever $n\geq 2k$, and is a quotient of the partition algebra otherwise. A basis for the partition algebra is given in terms of diagrams, where multiplication is defined by diagram concatenation, extended linearly across the basis. Such algebras are called diagram algebras. The algebras $\mathbb{C}S_{n}$, $B_{r}(\delta)$, and $B_{r,s}(\delta)$ are all examples of diagram algebras, and in fact each can be viewed as a subalgebra of a partition algebra. The partition algebras are semisimple almost always, that is they are semisimple for all but finitely many choices of the parameter $\delta \in \mathbb{C}$. We have a natural chain of algebras $\mathbb{C}:=\mathcal{A}_{0} \subset \mathcal{A}_{2}(\delta) \subset \mathcal{A}_{4}(\delta) \subset \dots$, but unfortunately this chain is not multiplicity-free. However, incorporating intermediate ``odd'' subalgebras, introduced by Martin in \cite{Martin00}, yields a multiplicity-free chain
\[ \mathbb{C} := \mathcal{A}_{0}(\delta) \subset \mathcal{A}_{1}(\delta) \subset \mathcal{A}_{2}(\delta) \subset \mathcal{A}_{3}(\delta) \subset \dots. \]
As was the case for $\mathbb{C}S_{n}$, we have a corresponding branching graph associated to this multiplicity-free chain (see \emph{\Cref{BranchGrDef}}). From this graph a GZ-basis for any simple $\mathcal{A}_{r}(\delta)$-module can be given, which is indexed by particular paths. We also have a collection of commuting elements $L_{1},\dots, L_{r} \in \mathcal{A}_{r}(\delta)$ which are analogous to the JM-elements of $\mathbb{C}S_{n}$, in particular we use the same notation and refer to them as the JM-elements of the partition algebra. These elements were first introduced by Halverson and Ram in \cite{HR05}, where they gave a diagrammatic description. They also showed that the GZ-basis diagonalises the action of the JM-elements and gave a description of the eigenvalues. The JM-elements were later given an alternative recursive definition by Enyang in \cite{Eny12}.

As discussed above, the centers of the algebras $\mathbb{C}S_{n}$, $B_{r}(\delta)$, and $B_{r,s}(\delta)$ are generated by certain (super)symmetric polynomials in their respected JM-elements. For this paper we seek to give an analogous result for the center of $\mathcal{A}_{2k}(\delta)$. The main result is \emph{\Cref{SSP=Center}} where we show that, in the semisimple case, the supersymmetric polynomials (see \emph{\Cref{SSPDef}}) in the normalised JM-elements generate the center of $\mathcal{A}_{2k}(\delta)$. For the previously discussed algebras, the strategy in showing the result on the center can be broken into two broad steps:
\begin{itemize}
\item[(1)] Show that the desired polynomials in the JM-elements are central.
\item[(2)] When the algebra is semisimple, use the action of the JM-elements to show that the correct number of such polynomials are linearly independent. 
\end{itemize}
For each algebra the details of these steps are different, but the overall approach is the same. The first main result of the paper is \emph{\Cref{SSPCentral}}, where we show step (1), which is independent of whether the partition algebra is semisimple or not. For this result we will make heavy use of the definitions and relations established in the work of Enyang in \cite{Eny12} and \cite{Eny13}. It is worth mentioning, that for the other algebras, step (1) is quite straightforward to show. For the partition algebra the situation is much more complicated, in particular showing that the Coxeter generators commute with such polynomials does not reduce to checking a few relations as was the case of \cite{JK17} (see \emph{\Cref{JK17CoxCommuteRmk}}). In the semisimple case, we use the action of the JM-elements on the GZ-basis described first in \cite{HR05}, along with the elementary supersymmetric polynomials (see \emph{\Cref{EleSSPDef}}), to help us show that the center consists precisely of these polynomials. In the non-semisimple case, we only establish that the subalgebra of such polynomials is central, but do not know whether this constitutes the entire center. However, we can show that this central subalgebra is large enough to recognise the block structure of the partition algebra. This allows us to give an alternative criterion for when two simple modules of $\mathcal{A}_{2k}(\delta)$ belong to the same block. 

The structure of the paper will proceed as follows: Section 2 recalls the definition of the partition algebra and the JM-elements, as described by Enyang in \cite{Eny13}. We also collect some of the relations established in both \cite{Eny12} and \cite{Eny13}, and prove some more relations involving the JM-elements. In Section 3 we begin by recalling the definition of supersymmetric polynomials, and the result of \cite{Stem85} in showing that the elementary supersymmetric polynomails generate all such polynomials. We then use the relations established in Section 2 to show that the supersymmetric polynomials in the normalised JM-elements are central in the partition algebra. In Section 4 we will turn our attention to the semisimple case. We recall some of the representation theory of the partition algebras. We are mainly interested in the branching graph given in \emph{\Cref{BranchGrDef}}, and the GZ-basis for simple modules. Then using the action of the JM-elements on such a basis we will show that the action of the supersymmetric polynomials in the normalised JM-elements can distinguish between the simple modules (\emph{\Cref{SSPDistSimples}}). From this, and a result of linear algebra, we will be able to implicitly produce a basis for the subalgebra of such polynomials in the normalised JM-elements, and then a dimension check will confirm that it is the center (\emph{\Cref{SSP=Center}}). Finally in Section 5 we will recall the block structure of the partition algebra established by Martin in \cite{Martin96}. We will show that the block structure can be recovered from the action of the subalgebra of supersymmetric polynomials in the normalised JM-elements. Knowing this we then conclude with an alternative criterion for when two simple modules of the partition algebra belong to the same block.


\section{Definitions}

Throughout this paper, all algebras will be considered over the field $\mathbb{C}
$ of complex numbers unless otherwise stated. Any algebraically closed field of characteristic 0 would equally do. For this section we begin by recalling the definition of the partition algebras and the Jucys-Murphy elements. We will also collect and prove various relations which will be used in the next section.

\subsection{The Partition Algebra $\mathcal{A}_{r}(\delta)$} \
\vspace{2mm}

Let $X$ be a finite set, then recall that a partition of $X$ is a collection $\pi = \{U_{1},\dots,U_{n}\}$ of subsets of $X$ such that $U_{i}\cap U_{j} = \emptyset$ for all $i \neq j$, and $\cup_{1\leq i \leq n}U_{i} = X$. We refer to any $U_{i} \in \pi$ as a \emph{block} of $\pi$. Let $\Pi(X)$ denote the set of all partitions of $X$. 

For any $k \in \mathbb{N}$ we set $[k]:=\{1,2,\dots,k\}$ and $[k']:=\{1',2',\dots,k'\}$. We view $[k]\cup [k']$ as a formal set of $2k$ elements and let $\Pi_{2k} := \Pi([k]\cup [k'])$. Any partition $\pi \in \Pi_{2k}$ can be represented by a graph consisting of two rows of $k$ vertices, where we label the top row of vertices from $1$ to $k$, and label the bottom row of vertices from $1'$ to $k'$. If $v,w \in [k]\cup[k']$ belong to the same block in the partition $\pi$, then they are connected by a path in the graph, i.e. will be within the same connected component of the graph. For example,
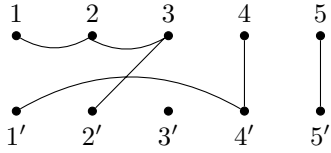
\[ \begin{matrix}\begin{tikzpicture}
	\node[V, label=above:{$1$}] (1) at (0,1){};
	\node[V, label=above:{$2$}] (2) at (1,1){};
	\node[V, label=above:{$3$}] (3) at (2,1){};
	\node[V, label=above:{$4$}] (4) at (3,1){};
	\node[V, label=above:{$5$}] (5) at (4,1){};
	\node[V, label=below:{$1'$}] (1') at (0,0){};
	\node[V, label=below:{$2'$}] (2') at (1,0){};
	\node[V, label=below:{$3'$}] (3') at (2,0){};
	\node[V, label=below:{$4'$}] (4') at (3,0){};
	\node[V, label=below:{$5'$}] (5') at (4,0){};
	\draw (1) to [bend right] (2) to [bend right] (3) to (2') (4) to (4') to [bend right] (1') (5) to (5');
	\end{tikzpicture}\end{matrix} \]
\noindent
represents the partition $\{\{1,2,2',3\},\{3'\},\{1',4,4'\},\{5,5'\}\}$ in $\Pi_{10}$. Many such graphs can represent the same partition $\pi$, since vertices within the same block may be connected in many different ways. We do not distinguish between such representations, that is two finite graphs on $2k$ (labelled) vertices are equivalent if and only if they have the same connected components. Then the equivalence classes of such graphs are in a one-to-one correspondence with $\Pi_{2k}$. We refer to such classes as partition diagrams. In this way we may identify any partition $\pi \in \Pi_{2k}$ with its partition diagram. From this we can define a multiplication of partitions $\pi,\gamma \in \Pi_{2k}$ by concatenation of their respected diagrams as follows: The product $\pi \circ \gamma$ is the partition diagram obtained by placing $\pi$ on top of $\gamma$ and identifying the bottom vertices of $\pi$ with the top vertices of $\gamma$, removing any connected components that live completely within the middle row, and then reading off the resulting connected components produced between the top row of $\pi$ and the bottom row of $\gamma$. This is best understood by example.

\begin{ex} \label{DiaMultEx} Let $\pi, \gamma \in \Pi_{10}$ be the partition diagrams
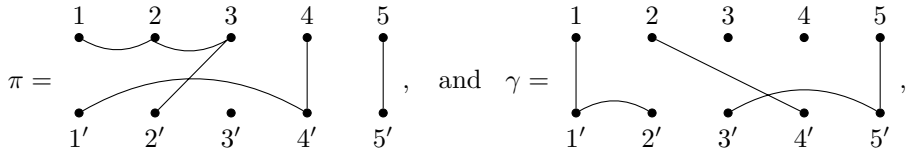
\[ \pi = \begin{matrix}\begin{tikzpicture}
	\node[V, label=above:{$1$}] (1) at (0,1){};
	\node[V, label=above:{$2$}] (2) at (1,1){};
	\node[V, label=above:{$3$}] (3) at (2,1){};
	\node[V, label=above:{$4$}] (4) at (3,1){};
	\node[V, label=above:{$5$}] (5) at (4,1){};
	\node[V, label=below:{$1'$}] (1') at (0,0){};
	\node[V, label=below:{$2'$}] (2') at (1,0){};
	\node[V, label=below:{$3'$}] (3') at (2,0){};
	\node[V, label=below:{$4'$}] (4') at (3,0){};
	\node[V, label=below:{$5'$}] (5') at (4,0){};
	\draw (1) to [bend right] (2) to [bend right] (3) to (2') (4) to (4') to [bend right] (1') (5) to (5');
	\end{tikzpicture}\end{matrix}, \hspace{2mm} \text{ and } \hspace{2mm}  
	\gamma = \begin{matrix}\begin{tikzpicture}
	\node[V, label=above:{$1$}] (1) at (0,1){};
	\node[V, label=above:{$2$}] (2) at (1,1){};
	\node[V, label=above:{$3$}] (3) at (2,1){};
	\node[V, label=above:{$4$}] (4) at (3,1){};
	\node[V, label=above:{$5$}] (5) at (4,1){};
	\node[V, label=below:{$1'$}] (1') at (0,0){};
	\node[V, label=below:{$2'$}] (2') at (1,0){};
	\node[V, label=below:{$3'$}] (3') at (2,0){};
	\node[V, label=below:{$4'$}] (4') at (3,0){};
	\node[V, label=below:{$5'$}] (5') at (4,0){};
	\draw (1) to (1') to [bend left] (2') (2) to (4') (5) to (5') to [bend right] (3');
	\end{tikzpicture}\end{matrix}, \]

then their product is 
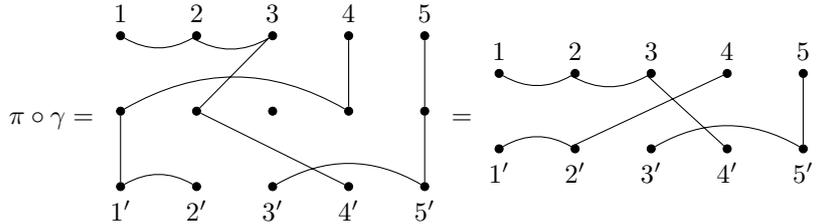
\[ \pi \circ \gamma = \begin{matrix}\begin{tikzpicture}
	\node[V, label=above:{$1$}] (1) at (0,1){};
	\node[V, label=above:{$2$}] (2) at (1,1){};
	\node[V, label=above:{$3$}] (3) at (2,1){};
	\node[V, label=above:{$4$}] (4) at (3,1){};
	\node[V, label=above:{$5$}] (5) at (4,1){};
	\node[V] (1') at (0,0){};
	\node[V] (2') at (1,0){};
	\node[V] (3') at (2,0){};
	\node[V] (4') at (3,0){};
	\node[V] (5') at (4,0){};
	\node[V, label=below:{$1'$}] (1'') at (0,-1){};
	\node[V, label=below:{$2'$}] (2'') at (1,-1){};
	\node[V, label=below:{$3'$}] (3'') at (2,-1){};
	\node[V, label=below:{$4'$}] (4'') at (3,-1){};
	\node[V, label=below:{$5'$}] (5'') at (4,-1){};
	\draw (1) to [bend right] (2) to [bend right] (3) to (2') (4) to (4') to [bend right] (1') (5) to (5') (1') to (1'') to [bend left] (2'') (2') to (4'') (5') to (5'') to [bend right] (3''); \end{tikzpicture}\end{matrix} = 
	 \begin{matrix}\begin{tikzpicture}
	\node[V, label=above:{$1$}] (1) at (0,1){};
	\node[V, label=above:{$2$}] (2) at (1,1){};
	\node[V, label=above:{$3$}] (3) at (2,1){};
	\node[V, label=above:{$4$}] (4) at (3,1){};
	\node[V, label=above:{$5$}] (5) at (4,1){};
	\node[V, label=below:{$1'$}] (1') at (0,0){};
	\node[V, label=below:{$2'$}] (2') at (1,0){};
	\node[V, label=below:{$3'$}] (3') at (2,0){};
	\node[V, label=below:{$4'$}] (4') at (3,0){};
	\node[V, label=below:{$5'$}] (5') at (4,0){};
	\draw (1) to [bend right] (2) to [bend right] (3) to (4') (1') to [bend left] (2') to (4) (5) to (5') to [bend right] (3');
	\end{tikzpicture}\end{matrix}. \]
\end{ex}

This multiplication respects the equivalence relation imposed previously and can be shown to be associative. For $\pi,\gamma \in \Pi_{2k}$ let $n(\pi,\gamma)$ denote the number of connected components removed in the process of obtaining the product $\pi \circ \gamma$. For the above example we have that $n(\pi,\gamma)=1$. Let $\delta \in \mathbb{C}$ and $k \in \mathbb{N}$, then the partition algebra, denoted by $\mathcal{A}_{2k}(\delta)$, is the associative unitial algebra with basis given by $\Pi_{2k}$, and with multiplication $\mathcal{A}_{2k}(\delta) \times \mathcal{A}_{2k}(\delta) \rightarrow \mathcal{A}_{2k}(\delta)$ given by
\[ (\pi,\gamma) \mapsto \delta^{n(\pi,\gamma)}\pi \circ \gamma, \]
for any $\pi,\gamma \in \Pi_{2k}$, extended linearly to all of $\mathcal{A}_{2k}(\delta)$. An arbitrary element of $\mathcal{A}_{2k}(\delta)$ is thus a formal linear combination of partition diagrams in $\Pi_{2k}$. We will represent the multiplication of basis elements $\pi$ and $\gamma$ by the concatenation of symbols $\pi\gamma$. The identity element of $\mathcal{A}_{2k}(\delta)$ is the basis vector given by the partition $\{\{i,i'\} \hspace{1mm} | \hspace{1mm} i \in [k]\}$. For convention we set $\mathcal{A}_{0}(\delta) = \mathbb{C}$.

For $i \in [k-1]$ and $j \in [k]$, we define
\[ s_{i} = \begin{matrix}\begin{tikzpicture}
	
	\node[V, label=above:{$1$}] (1) at (0,1){};
	
	\node[draw=none] (0.5) at (0.5,0.5){$\ldots$};
	
	\node[V, label=below:{$1'$}] (1') at (0,0){};
	\node[V] (2) at (1,1){};
	\node[V, label=above:{$i$}] (i) at (2,1){};
	\node[V, label=above:{$i+1$}] (i+1) at (3,1){};
	\node[V] (3) at (4,1){};
	\node[V] (2') at (1,0){};
	\node[V, label=below:{$i'$}] (i') at (2,0){};
	\node[V, label=below:{$(i+1)'$}] (i+1') at (3,0){};
	\node[V] (3') at (4,0){};
	
	\node[draw=none] (4.5) at (4.5,0.5){$\ldots$};
	
	\node[V, label=above:{$k$}] (k) at (5,1){};
	\node[V, label=below:{$k'$}] (k') at (5,0){};
	
	\draw (1) to (1') (2) to (2') (i) to (i+1') (i+1) to (i') (3) to (3') (k) to (k');
	
	\end{tikzpicture}\end{matrix}, \hspace{7.5mm}	
	e_{2j-1} = \begin{matrix}\begin{tikzpicture}
	
	\node[V, label=above:{$1$}] (1) at (0,1){};
	
	\node[draw=none] (0.5) at (0.5,0.5){$\ldots$};
	
	\node[V, label=below:{$1'$}] (1') at (0,0){};
	\node[V] (2) at (1,1){};
	\node[V, label=above:{$j$}] (j) at (2,1){};
	\node[V] (3) at (3,1){};
	\node[V] (2') at (1,0){};
	\node[V, label=below:{$j'$}] (j') at (2,0){};
	\node[V] (3') at (3,0){};
	
	\node[draw=none] (3.5) at (3.5,0.5){$\ldots$};
	
	\node[V, label=above:{$k$}] (k) at (4,1){};
	\node[V, label=below:{$k'$}] (k') at (4,0){};
	
	\draw (1) to (1') (2) to (2') (3) to (3') (k) to (k');
	
	\end{tikzpicture}\end{matrix}, \]

\[ e_{2i} = \begin{matrix}\begin{tikzpicture}
	
	\node[V, label=above:{$1$}] (1) at (0,1){};
	
	\node[draw=none] (0.5) at (0.5,0.5){$\ldots$};
	
	\node[V, label=below:{$1'$}] (1') at (0,0){};
	\node[V] (2) at (1,1){};
	\node[V, label=above:{$i$}] (i) at (2,1){};
	\node[V, label=above:{$i+1$}] (i+1) at (3,1){};
	\node[V] (3) at (4,1){};
	\node[V] (2') at (1,0){};
	\node[V, label=below:{$i'$}] (i') at (2,0){};
	\node[V, label=below:{$(i+1)'$}] (i+1') at (3,0){};
	\node[V] (3') at (4,0){};
	
	\node[draw=none] (4.5) at (4.5,0.5){$\ldots$};
	
	\node[V, label=above:{$k$}] (k) at (5,1){};
	\node[V, label=below:{$k'$}] (k') at (5,0){};
	
	\draw (1) to (1') (2) to (2') (i) to [bend right] (i+1) (i) to (i') (i') to [bend left] (i+1') (i+1) to (i+1') (3) to (3') (k) to (k');
	
	\end{tikzpicture}\end{matrix}. \]

\noindent
The elements $s_{1},\dots,s_{k-1},e_{1},e_{2},\dots,e_{2k-1}$ generate $\mathcal{A}_{2k}(\delta)$. The subalgebra generated by the elements $s_{i}$ for $i\in [k-1]$ is precisely the symmetric group algebra $\mathbb{C}S_{k}$. There is an algebra anti-automorphism $*:\mathcal{A}_{2k}(\delta) \rightarrow \mathcal{A}_{2k}(\delta)$ defined by flipping a partition diagram up-side-down, and extending linearly over $\mathcal{A}_{2k}(\delta)$. We denote the image of $a \in \mathcal{A}_{2k}(\delta)$ under this anti-automorphism by $a^{*}$. In particular we have that the generators above are invariant under $*$. Furthermore, restricting this anti-automorphism to the subalgebra $\mathbb{C}S_{k}$ yields the usual anti-automorphism of $\mathbb{C}S_{k}$ given by inversion. In \cite[Theorem 1.1]{HR05} a presentation for $\mathcal{A}_{2k}(\delta)$ in terms of the above generators was given. We recall some of the relations in this presentation for later use, which can be easily verified diagrammatically. 

\begin{lem} \label{StRels} \
\begin{itemize}
\item[$(1)$] $s_{i}e_{2i}=e_{2i}s_{i}=e_{2i}$.
\item[$(2)$] $s_{i}e_{2i-1}s_{i}=e_{2i+1}$.
\item[$(3)$] $e_{i}e_{i \pm 1}e_{i} = e_{i}$.
\end{itemize}
\end{lem}

The partition algebra $\mathcal{A}_{2k}(\delta)$ has a subalgebra, denoted by $\mathcal{A}_{2k-1}(\delta)$, which is spanned by the partitions $\pi \in \Pi_{2k}$ such that $k$ and $k'$ belong to the same block of $\pi$. In turn, the algebra $\mathcal{A}_{2k-1}(\delta)$ contains $\mathcal{A}_{2k-2}(\delta)$ as a subalgebra by identifying $\mathcal{A}_{2k-2}(\delta)$ with the span of all partitions $\pi \in \Pi_{2k}$ with $\{k,k'\}$ as a block. From such identifications we obtain a chain of algebras
\[ \mathcal{A}_{0}(\delta) \subset \mathcal{A}_{1}(\delta) \subset \mathcal{A}_{2}(\delta) \subset \mathcal{A}_{3}(\delta) \subset \dots \subset \mathcal{A}_{2k}(\delta). \]
In terms of the above generators $\mathcal{A}_{2k-1}(\delta) = \langle s_{i},e_{j} \hspace{1mm} | \hspace{1mm} i \in [k-2], j \in [2k-2]  \rangle \subset \mathcal{A}_{2k}(\delta)$, that is we have dropped the generators $s_{k-1}$ and $e_{2k-1}$. Then to obtain $\mathcal{A}_{2k-2}(\delta)$, we further drop the generator $e_{2k-2}$. Whenever we are referring to a subalgebra $\mathcal{A}_{r}(\delta) \subset \mathcal{A}_{2k}(\delta)$, we will be using the identification described in the above chain. Furthermore, we will use the variable $r$ in the index to represent any parity, while we will use $2k$ or $2k\pm 1$ otherwise.

\subsection{The Jucys-Murphy Elements} \
\vspace{2mm}

In this section we will give the definition of the Jucys-Murphy elements (JM-elements) of the partition algebra, and describe some properties and relations regarding them. For each inclusion $\mathcal{A}_{r-1}(\delta) \subset \mathcal{A}_{r}(\delta)$ we obtain a new JM-element $L_{r} \in \mathcal{A}_{r}(\delta)$, and this element belongs to the centraliser subalgebra
\[ Z(\mathcal{A}_{r-1}(\delta),\mathcal{A}_{r}(\delta)):=\langle z \in \mathcal{A}_{r}(\delta) \hspace{1mm} | \hspace{1mm} za=az \text{ for all } a \in\mathcal{A}_{r-1}(\delta) \rangle. \] 
These elements were first introduced in \cite{HR05}, where they described them diagrammatically. It proves quite difficult to establish relations between the JM-elements and the generators of $\mathcal{A}_{r}(\delta)$ uses this diagrammatic definition. Fortunately Enyang has given a recursive definition of the JM-elements in \cite{Eny12}, alongside new elements $\sigma_{i}$, which we both recall below.

\begin{rmk}
This subsection relies on the work of Enyang from both \cite{Eny12} and \cite{Eny13}. However there is a change of notation between these two papers. We will be adopting the notation used in \cite{Eny13}, in particular the definition below is from \cite{Eny13}. To ease reference checking, the conversion of notation from \cite{Eny12} to \cite{Eny13} respectively is given by $p_{i} \sim e_{2i-1}$, $p_{i+\frac{1}{2}} \sim e_{2i}$, $\sigma_{i} \sim \sigma_{2i-1}$, $\sigma_{i+\frac{1}{2}} \sim \sigma_{2i}$, $L_{i} \sim L_{2i}$, and $L_{i+\frac{1}{2}} \sim L_{2i+1}$.
\end{rmk}

\begin{defn} \label{EnyJMDef}
Let $L_{1}=0, L_{2}=e_{1}, \sigma_{2}=1$, and $\sigma_{3}=s_{1}$. Then for $i=1,2,\dots,$ define
\[ L_{2i+2} = s_{i}L_{2i}s_{i} - s_{i}L_{2i}e_{2i} - e_{2i}L_{2i}s_{i} + e_{2i}L_{2i}e_{2i+1}e_{2i} + \sigma_{2i+1}, \]
where, for $i=2,3,\dots,$ we have
\begin{align*}
\sigma_{2i+1} = s_{i-1}s_{i}\sigma_{2i-1}s_{i}&s_{i-1}+s_{i}e_{2i-2}\mathbf{L_{2i-2}}s_{i}e_{2i-2}s_{i} + e_{2i-2}L_{2i-2}s_{i}e_{2i-2} \\
&- s_{i}e_{2i-2}L_{2i-2}s_{i-1}e_{2i}e_{2i-1}e_{2i-2} - e_{2i-2}e_{2i-1}e_{2i}s_{i-1}L_{2i-2}e_{2i-2}s_{i}.
\end{align*}
Also for $i=1,2,\dots,$ define
\[ L_{2i+1} = s_{i}L_{2i-1}s_{i} - L_{2i}e_{2i} - e_{2i}L_{2i} + (\delta-L_{2i-1})e_{2i} + \sigma_{2i}, \]
where, for $i=2,3,\dots,$ we have
\begin{align*}
\sigma_{2i} = s_{i-1}s_{i}\sigma_{2i-2}s_{i}&s_{i-1}+e_{2i-2}\mathbf{L_{2i-2}}s_{i}e_{2i-2}s_{i} + s_{i}e_{2i-2}\mathbf{L_{2i-2}}s_{i}e_{2i-2} \\
&-e_{2i-2}\mathbf{L_{2i-2}}s_{i-1}e_{2i}e_{2i-1}e_{2i-2} - s_{i}e_{2i-2}e_{2i-1}e_{2i}s_{i-1}\mathbf{L_{2i-2}}e_{2i-2}s_{i}.
\end{align*}
\end{defn}   

In Section 2.3 of \cite{Eny13} there were a few typos in the definition of the JM-elements due to the change in notation. We have corrected these and made them bold in the above definition. A straightforward proof by induction yields the fact that $L_{i}$ belongs to $\mathcal{A}_{i}(\delta)$ and $\sigma_{i}$ belongs to $\mathcal{A}_{i+1}(\delta)$. Enyang showed that the elements $e_{1},e_{2},\dots,e_{2k-1},\sigma_{2},\sigma_{3},\dots,\sigma_{2k-1}$ generate $\mathcal{A}_{2k}(\delta)$, moreover a presentation of $\mathcal{A}_{2k}(\delta)$ in terms of these generators was given (see Theorem 4.1 of \cite{Eny12}). Remarkably, although the diagrammatic description of the elements $\sigma_{i}$ gets quite complicated, the relations in Enyang's presentation are very simple. We will need a variety of the relations and facts regarding the elements $L_{i}$ and $\sigma_{i}$ established in both \cite{Eny12} and \cite{Eny13} throughout this paper. We collect below such results.   

\begin{lem} \label{EnyRels}
Whenever the indices make sense, we have the following relations:

\begin{itemize}
\item[(1)] \emph{(Sigma Relations)}
\begin{itemize}
\item[(i)] $\sigma_{i}^{*} = \sigma_{i}$
\item[(ii)] $\sigma_{i}^{2} = 1$
\item[(iii)] $\sigma_{2i}\sigma_{2i+1}=\sigma_{2i+1}\sigma_{2i}=s_{i}$
\item[(iv)] $\sigma_{i}$ commutes with $\mathcal{A}_{i-2}(\delta)$
\item[(v)] $\sigma_{2i}e_{2i} = e_{2i}\sigma_{2i} = e_{2i}$
\item[(vi)] $\sigma_{2i+1}e_{2i} = e_{2i}\sigma_{2i+1} = e_{2i}$
\end{itemize}
\item[(2)] \emph{(JM Relations)}
\begin{itemize}
\item[(i)] $L_{i}^{*}=L_{i}$
\item[(ii)] $L_{i}L_{j}=L_{j}L_{i}$
\item[(iii)] $\sum_{i=1}^{r}L_{i}$ is central in $\mathcal{A}_{r}(\delta)$
\item[(iv)] $L_{i}$ commutes with $\mathcal{A}_{i-1}(\delta)$
\end{itemize}
\item[(3)] \emph{(Mixed Relations)}
\begin{itemize}
\item[(i)] $e_{2i+1}\sigma_{2i}e_{2i+1}=(\delta-L_{2i-1})e_{2i+1}$
\item[(ii)] $e_{i}(L_{i}+L_{i+1})=(L_{i}+L_{i+1})e_{i} = \delta e_{i}$
\item[(iii)] $\sigma_{2i}e_{2i-1}e_{2i}=L_{2i}e_{2i}$, and $e_{2i}e_{2i-1}\sigma_{2i}=e_{2i}L_{2i}$
\item[(iv)] $\sigma_{2i+1}e_{2i+1}e_{2i}=L_{2i}e_{2i}$, and $e_{2i}e_{2i+1}\sigma_{2i+1}=e_{2i}L_{2i}$
\end{itemize}
\end{itemize}

\end{lem}

\begin{proof}
All of these relations can be found in \cite{Eny12}, we simply give details of where to find them. $(1)(i)$ for the even indices is \cite[Proposition 3.3 (1)]{Eny12}, and for the odd indices it then follows from \cite[Proposition 3.4]{Eny12}. $(1)(ii)$ is given in \cite[Proposition 4.2]{Eny12}. $(1)(iii)$ is given in \cite[Theorem 4.1]{Eny12}. $(1)(iv)$ is given by \cite[Theorem 3.8]{Eny12}. Both $(1)(v)$ and $(1)(vi)$ are shown in \cite[Theorem 4.1]{Eny12}. $(2)(i)$ follows by definition for even indices (noting that $L_{2i}$ commutes with $e_{2i+1}$), and odd indices by \cite[Proposition 3.3 (2)]{Eny12}. $(2)(ii)$ follows from $(2)(iv)$ which is given by \cite[Proposition 3.8]{Eny12}. $(2)(iii)$ is precisely \cite[Proposition 3.10]{Eny12}. $(3)(i)$ is \cite[Proposition 4.3 (2)]{Eny12}. $(3)(ii)$ is \cite[Proposition 3.9 (1) and (2)]{Eny12}. The first equality of $(3)(iv)$ is \cite[Proposition 3.2 (4)]{Eny12}, and the second follows by acting on that equality by $*$. The first equality of $(3)(iii)$ is obtained from \cite[Proposition 3.2 (3)]{Eny12} by left multiplying by $s_{i}$, and again the second follows via $*$.

\end{proof}

We will mainly be interested in working with the normalised JM-elements, defined to be $N_{i} := L_{i} - \delta/2$. These elements are better suited to describe the center of $\mathcal{A}_{2k}(\delta)$. We will now establish some relations regarding the elements $N_{i}$ which will be used extensively in the next section. We first show which generators commute with the normalised JM-elements.

\begin{lem} \label{CommutingRels}
We have the following commuting relations: 

\begin{itemize}
\item[(1)] $e_{i}N_{j} = N_{j}e_{i}$ for all $j\neq i,i+1$.
\item[(2)] $s_{i}N_{j} = N_{j}s_{i}$ for all $j\neq 2i-1,2i,2i+1,2i+2$.
\item[(3)] $\sigma_{2i+1}N_{j} = N_{j}\sigma_{2i+1}$ for all $j\neq 2i,2i+1,2i+2$.
\item[(4)] $\sigma_{2i}N_{j} = N_{j}\sigma_{2i}$ for all $j\neq 2i-1,2i,2i+1$.
\end{itemize}
\end{lem} 

\begin{proof}
(1): We have $N_{j} \in \mathcal{A}_{j}(\delta)$, so diagrammatically one can see that $e_{2i}$ commutes with $N_{1},N_{2},\dots,N_{2i-1}$, and that $e_{2i-1}$ commutes with $N_{1},N_{2},\dots,N_{2i-2}$. Also $N_{j}$ commutes with $\mathcal{A}_{j-1}(\delta)$ by \emph{\Cref{EnyRels}} (2)(iv), and $e_{i} \in \mathcal{A}_{i+1}(\delta)$, hence $e_{2i}$ commutes with $N_{2i+2},N_{2i+3},\dots,N_{2k}$, and $e_{2i-1}$ commutes with $N_{2i+1},N_{2i+2},\dots,N_{2k}$. Thus collectively $e_{2i}$ commutes with all $N_{j}$ for $j \neq 2i,2i+1$, and $e_{2i-1}$ commutes with all $N_{j}$ for $j \neq 2i-1,2i$, which gives (1). Arguing in the same manner gives (2).

(3): By \emph{\Cref{EnyRels}} (1)(iv), $\sigma_{2i+1}$ commutes with $\mathcal{A}_{2i-1}(\delta)$, and so it commutes with $N_{1},N_{2},\dots,N_{2i-1}$. Furthermore, $N_{j}$ commutes with $\mathcal{A}_{j-1}(\delta)$ by \emph{\Cref{EnyRels}} (2)(iv), and $\sigma_{2i+1} \in \mathcal{A}_{2i+2}(\delta)$, so we see that $\sigma_{2i+1}$ commutes with $N_{2k+3},N_{2k+4},\dots,N_{2k}$. This completes (3). A similar argument shows (4).

\end{proof}

We would like to know what interaction occurs between the Coxeter generators and the normalised JM-elements when they do not commute. The lemma below can be used to give us relations of the form $s_{i}N_{2i}=N_{2i+2}s_{i}+C$ and $s_{i}N_{2i-1}=N_{2i+1}s_{i}+D$, where $C$ and $D$ are some linear combination of diagrams, which will come in handy in the next section. 

\begin{lem} \label{CoxSkeinRels}
We have the following relations:

\begin{itemize}
\item[(1)] $N_{2i+1} = s_{i}N_{2i-1}s_{i} - \sigma_{2i}e_{2i-1}e_{2i} - e_{2i}e_{2i-1}\sigma_{2i} + e_{2i}e_{2i-1}\sigma_{2i}e_{2i-1}e_{2i} + \sigma_{2i}$. 
\item[(2)] $N_{2i+2} = s_{i}N_{2i}s_{i} - \sigma_{2i+1}e_{2i-1}e_{2i} - e_{2i}e_{2i-1}\sigma_{2i+1} + e_{2i}e_{2i+1}\sigma_{2i+1}e_{2i+1}e_{2i} + \sigma_{2i+1}$.
\end{itemize}
\end{lem}

\begin{proof}
(1): Firstly, by \emph{\Cref{EnyRels}} (1)(ii) and (1)(iii) we have that
\[ s_{i}\sigma_{2i}s_{i} = s_{i}\sigma_{2i}\sigma_{2i}\sigma_{2i+1} = s_{i}\sigma_{2i+1} = \sigma_{2i}\sigma_{2i+1}\sigma_{2i+1} = \sigma_{2i}. \]
Now we show that $(\delta - L_{2i-1})e_{2i} = e_{2i}e_{2i-1}\sigma_{2i}e_{2i-1}e_{2i}$:
\begin{align*}
(\delta-L_{2i-1})e_{2i} &= (\delta-L_{2i-1})e_{2i}e_{2i+1}e_{2i} \hspace{6mm} (\text{by } \Cref{StRels} (3)) \\
&= e_{2i}(\delta-L_{2i-1})e_{2i+1}e_{2i} \hspace{6mm} (\text{by } \Cref{CommutingRels} (1)) \\
&= e_{2i}e_{2i+1}\sigma_{2i}e_{2i+1}e_{2i} \hspace{11mm} (\text{by } \Cref{EnyRels} (3)(i)) \\
&= e_{2i}e_{2i+1}s_{i}\sigma_{2i}s_{i}e_{2i+1}e_{2i} \hspace{5mm} (\text{Since } s_{i}\sigma_{2i}s_{i}=\sigma_{2i}) \\
&= e_{2i}s_{i}e_{2i-1}\sigma_{2i}e_{2i-1}s_{i}e_{2i} \hspace{5mm} (\text{by } \Cref{StRels} (2)) \\
&= e_{2i}e_{2i-1}\sigma_{2i}e_{2i-1}e_{2i}. \hspace{10mm} (\text{by } \Cref{StRels} (1))
\end{align*}
From this we obtain
\begin{align*}
L_{2i+1} &\defeq s_{i}L_{2i-1}s_{i}-L_{2i}e_{2i}-e_{2i}L_{2i} + (\delta-L_{2i-1})e_{2i} + \sigma_{2i}, \\
&= s_{i}L_{2i-1}s_{i}-L_{2i}e_{2i}-e_{2i}L_{2i} + e_{2i}e_{2i-1}\sigma_{2i}e_{2i-1}e_{2i} + \sigma_{2i}, \\
&= s_{i}L_{2i-1}s_{i}-\sigma_{2i}e_{2i-1}e_{2i}-e_{2i}e_{2i-1}\sigma_{2i} + e_{2i}e_{2i-1}\sigma_{2i}e_{2i-1}e_{2i} + \sigma_{2i},
\end{align*}
where the second equality followed from the substitution $(\delta - L_{2i-1})e_{2i} = e_{2i}e_{2i-1}\sigma_{2i}e_{2i-1}e_{2i}$, and the last follows from \emph{\Cref{EnyRels}} (3)(iii). Since $N_{i}:=L_{i}-\delta/2$, it is clear that replacing $L_{2i+1}$ and $L_{2i-1}$ with $N_{2i+1}$ and $N_{2i-1}$ respectively will still yield a valid equality.

(2): First we have
\begin{align*}
s_{i}L_{2i}e_{2i} &= s_{i}\sigma_{2i}e_{2i-1}e_{2i} \hspace{6mm} (\text{by } \Cref{EnyRels} (3)(iii)) \\
&= \sigma_{2i+1}e_{2i-1}e_{2i} \hspace{6mm} (\text{Since } s_{i}\sigma_{2i} = \sigma_{2i+1}).
\end{align*}
Using the anti-automorphism $*$ we obtain $e_{2i}L_{2i}s_{i} = e_{2i}e_{2i-1}\sigma_{2i+1}$. We also have that
\begin{align*}
e_{2i}L_{2i}e_{2i+1}e_{2i} &= e_{2i}e_{2i-1}\sigma_{2i}e_{2i+1}e_{2i} \hspace{12mm} (\text{by } \Cref{EnyRels} (3)(iii)) \\
&= e_{2i}e_{2i-1}s_{i}\sigma_{2i+1}e_{2i+1}e_{2i} \hspace{6mm} (\text{by } \Cref{EnyRels} (1)(iii)) \\
&= e_{2i}s_{i}e_{2i+1}\sigma_{2i+1}e_{2i+1}e_{2i} \hspace{6mm} (\text{by } \Cref{StRels} (2)) \\
&= e_{2i}e_{2i+1}\sigma_{2i+1}e_{2i+1}e_{2i}. \hspace{8mm} (\text{by } \Cref{StRels} (1))
\end{align*}
Using these we get
\begin{align*}
L_{2i+2} &\defeq s_{i}L_{2i}s_{i} - s_{i}L_{2i}e_{2i} - e_{2i}L_{2i}s_{i} + e_{2i}L_{2i}e_{2i+1}e_{2i} + \sigma_{2i+1} \\
&= s_{i}L_{2i}s_{i} - \sigma_{2i+1}e_{2i-1}e_{2i} - e_{2i}e_{2i-1}\sigma_{2i+1} + e_{2i}e_{2i+1}\sigma_{2i+1}e_{2i+1}e_{2i} + \sigma_{2i+1}.
\end{align*}
Once again replacing $L_{2i+2}$ with $N_{2i+2}$ and $L_{2i}$ with $N_{2i}$ gives a valid equality.

\end{proof}

The next lemma tells us how the normalised JM-elements interact with the odd indexed generators $\sigma_{i}$, when they do not commute.

\begin{lem} \label{OddSigSkeinRels}
The following relations hold:
\begin{itemize}
\item[(1)] $N_{2i+2} = \sigma_{2i+1}N_{2i}\sigma_{2i+1} - e_{2i+1}e_{2i} - e_{2i}e_{2i+1} + e_{2i}e_{2i+1}\sigma_{2i+1}e_{2i+1}e_{2i} + \sigma_{2i+1}$.
\item[(2)] $N_{2i+1} = \sigma_{2i+1}N_{2i+1}\sigma_{2i+1} + e_{2i+1}e_{2i} + e_{2i}e_{2i+1} - \sigma_{2i+1}e_{2i+1}e_{2i} -e_{2i}e_{2i+1}\sigma_{2i+1}$.
\end{itemize}
\end{lem}

\begin{proof}
(1): From \emph{\Cref{CoxSkeinRels}} (2) we  have
\[ N_{2i+2} = s_{i}N_{2i}s_{i} - \sigma_{2i+1}e_{2i-1}e_{2i} - e_{2i}e_{2i-1}\sigma_{2i+1} + e_{2i}e_{2i+1}\sigma_{2i+1}e_{2i+1}e_{2i} + \sigma_{2i+1}. \]
Multiplying this equality on the left and right by $\sigma_{2i}$ gives
\begin{align*} 
\sigma_{2i}N_{2i+2}\sigma_{2i} = \sigma_{2i}s_{i}N_{2i}s_{i}\sigma_{2i} - \sigma_{2i}\sigma_{2i+1}&e_{2i-1}e_{2i}\sigma_{2i} - \sigma_{2i}e_{2i}e_{2i-1}\sigma_{2i+1}\sigma_{2i} \\ 
&+ \sigma_{2i}e_{2i}e_{2i+1}\sigma_{2i+1}e_{2i+1}e_{2i}\sigma_{2i} + \sigma_{2i}\sigma_{2i+1}\sigma_{2i}.
\end{align*}
By \emph{\Cref{CommutingRels}} (4) we know that $\sigma_{2i}$ commute with $N_{2i+2}$, and by \emph{\Cref{EnyRels}} (1)(ii) $\sigma_{2i}^{2}=1$, hence the left hand side of above becomes $\sigma_{2i}N_{2i+2}\sigma_{2i} = N_{2i+2}$. \emph{\Cref{EnyRels}} (1)(iii) and (1)(v) tell us that $s_{i}=\sigma_{2i}\sigma_{2i+1}=\sigma_{2i+1}\sigma_{2i}$ and $\sigma_{2i}e_{2i} = e_{2i}\sigma_{2i}=e_{2i}$. From this we see that the above reduces to
\[ N_{2i+2} = \sigma_{2i+1}N_{2i}\sigma_{2i+1} - s_{i}e_{2i-1}e_{2i} - e_{2i}e_{2i-1}s_{i} + e_{2i}e_{2i+1}\sigma_{2i+1}e_{2i+1}e_{2i} + \sigma_{2i+1}. \]
Lastly, by \emph{\Cref{StRels}} we have that $s_{i}e_{2i-1}e_{2i} = e_{2i+1}s_{i}e_{2i} = e_{2i+1}e_{2i}$, and similarly $e_{2i}e_{2i-1}s_{i} = e_{2i}e_{2i+1}$. Hence the above equation reduces to the desired one.

(2): The sum $\sum_{i=1}^{r}N_{i}$ is central in $\mathcal{A}_{r}(\delta)$ by \emph{\Cref{EnyRels}} (2)(iii), and by \emph{\Cref{CommutingRels}} (3) we know that $\sigma_{2i+1}N_{j} = N_{j}\sigma_{2i+1}$ for all $j \neq 2i, 2i+1, 2i+1$. From these two facts we have that
\[ \sigma_{2i+1}(N_{2i} + N_{2i+1} + N_{2i+2})\sigma_{2i+1} = N_{2i} + N_{2i+1} + N_{2i+2}. \]
Rearranging we have
\[ N_{2i+1} = \sigma_{2i+1}N_{2i+1}\sigma_{2i+1} + (\sigma_{2i+1}N_{2i}\sigma_{2i+1} - N_{2i+2}) + (\sigma_{2i+1}N_{2i+2}\sigma_{2i+1} - N_{2i}). \tag{Eq1} \]
We will now focus on the bracketed terms in (Eq1). For the first bracketed term, (1) tells us that
\begin{align*}
\sigma_{2i+1}N_{2i}\sigma_{2i+1} - N_{2i+2} &= (N_{2i+2} + e_{2i+1}e_{2i} + e_{2i}e_{2i+1} - e_{2i}e_{2i+1}\sigma_{2i+1}e_{2i+1}e_{2i} - \sigma_{2i+1}) - N_{2i+2} \\
&= e_{2i+1}e_{2i} + e_{2i}e_{2i+1} - e_{2i}e_{2i+1}\sigma_{2i+1}e_{2i+1}e_{2i} - \sigma_{2i+1}
\end{align*}
For the second bracketed term, multiplying the equality from (1) by $\sigma_{2i+1}$ on both the left and right hand sides, and then rearranging gives
\begin{align*}
\sigma_{2i+1}N_{2i+2}\sigma_{2i+1} - N_{2i} &= (N_{2i} - \sigma_{2i+1}e_{2i+1}e_{2i} -e_{2i}e_{2i+1}\sigma_{2i+1} + e_{2i}e_{2i+1}\sigma_{2i+1}e_{2i+1}e_{2i} + \sigma_{2i+1} ) - N_{2i} \\
&= - \sigma_{2i+1}e_{2i+1}e_{2i} -e_{2i}e_{2i+1}\sigma_{2i+1} + e_{2i}e_{2i+1}\sigma_{2i+1}e_{2i+1}e_{2i} + \sigma_{2i+1}.
\end{align*}
Therefore the sum of the two bracketed terms is
\begin{align*}
(\sigma_{2i+1}N_{2i}\sigma_{2i+1} - N_{2i+2}) + (\sigma_{2i+1}&N_{2i+2}\sigma_{2i+1} - N_{2i}) = \\
&e_{2i+1}e_{2i} + e_{2i}e_{2i+1} - \sigma_{2i+1}e_{2i+1}e_{2i} -e_{2i}e_{2i+1}\sigma_{2i+1}.
\end{align*}
Hence (Eq1) becomes
\[ N_{2i+1} = \sigma_{2i+1}N_{2i+1}\sigma_{2i+1} + e_{2i+1}e_{2i} + e_{2i}e_{2i+1} - \sigma_{2i+1}e_{2i+1}e_{2i} -e_{2i}e_{2i+1}\sigma_{2i+1}. \]

\end{proof}

For our proposes, we do not need an analogous result regarding the even indexed generators $\sigma_{i}$. The next lemma tells us how the generators $e_{i}$ interact with the normalised JM-elements whenever they do not commute.

\begin{lem} \label{AntiSymmRels}
We have the following relations:
\begin{itemize}
\item[(1)] $e_{i}N_{i} = -e_{i}N_{i+1}$
\item[(2)] $N_{i}e_{i} = -N_{i+1}e_{i}$
\end{itemize}
\end{lem}

\begin{proof}
From \emph{\Cref{EnyRels}} (3)(ii), and the definition that $N_{i}=L_{i}-\delta/2$, we have
\[ e_{i}(N_{i} + N_{i+1}) = -\delta e_{i} + e_{i}(L_{i}+L_{i+1}) = -\delta e_{i} + \delta e_{i} = 0. \]
Hence $e_{i}N_{i} = -e_{i}N_{i+1}$ giving (1). Relation (2) follows by applying the anti-automorphism $*$ to (1).

\end{proof}


\section{Central Subalgebra}

For this section we will show that the subalgebra of supersymmetric polynomials in the normalised JM-elements belongs to the center of the partition algebra. We begin by recalling the definition of the algebra of supersymmetric polynomials and some of its properties. Then we will use the relations given in the previous section to show that the generators of the partition algebra commute with any such polynomial in the normalised JM-elements.

\subsection{Supersymmetric Polynomials} \
\vspace{2mm}

What is covered here can be found in \cite{Stem85} and \cite{Moens07}. We remodel the definitions a little to better align with our situation. Let $r$ be a non-negative integer, then we will denote by $r_{e}$ the largest even integer such that $r_{e}\leq r$, and similarly we denote by $r_{o}$ the largest odd integer such that $r_{o}\leq r$. We also define the sets $E(r) = \{2,4,\dots,r_{e}\}$ and $O(r) = \{1,3,\dots,r_{o}\}$. We set $\mathcal{P}_{r} := \mathbb{C}[x_{1},\dots,x_{r}]$, the algebra of polynomials in $r$ commuting variables. For convention, when $r=0$, we set $r_{e}=r_{o}=0$, $E(r)=O(r)=\emptyset$, and $\mathcal{P}_{0} := \mathbb{C}$.

\begin{defn} \label{SSPDef}
Let $r$ be a non-negative integer and $p \in \mathcal{P}_{r}$. We say that $p$ is supersymmetric if
\begin{itemize}
\item[(1)] $p$ is \emph{parity symmetric}: $p$ is symmetric in $x_{1},x_{3},\dots,x_{r_{o}}$, and symmetric in $x_{2},x_{4},\dots,x_{r_{e}}$. 
\item[(2)] $p$ satisfies the \emph{cancellation property}: substituting $x_{1} = -x_{2} = y$ yields a polynomial in $x_{3},x_{4},\dots,x_{r}$ which is independent of $y$.
\end{itemize}
We will denote by $SS_{r}[x]$ the set of all supersymmetric polynomials in $r$ commuting variables.
\end{defn}

When $r$ is even, then the number of odd index variables agrees with the number of even index variables, while when $r$ is odd, there is one more odd index variable than even. We will often suppress the arguments of a polynomial $p$, but when we want to be clearer about the number of variables in play we will write $p(x_{1},x_{2},\dots,x_{r})$. It is not difficult to see that the two properties in the above definition respect addition and multiplication of polynomials, and so $SS_{r}[x]$ is in fact a subalgebra of $\mathcal{P}_{r}$. For convention we set $SS_{0}[x] = \mathbb{C}$, and we have that $SS_{1}[x] = \mathbb{C}[x]$, the polynomial algebra in one variable $x$.

\begin{rmk} \label{JK17SSPConvRmk}
In \cite{JK17}, they consider supersymmetric polynomials in the two sets of variables $X=\{x_{1},\dots,x_{r}\}$ and $Y=\{y_{1},\dots,y_{s}\}$. That is a supersymmetric polynomial is one which is symmetric in the $X$ variables, symmetric in the $Y$ variables, and satisfies an analogous cancellation property to (2) above. We are working with the specialisation $X = \{x_{1},x_{3},\dots,x_{r_{o}}\}$ and $Y = \{x_{2},x_{4},\dots,x_{r_{e}}\}$.
\end{rmk}

\begin{exs} \label{SSPExs} \

\begin{itemize}
\item[(1)] Let $n,r \in \mathbb{Z}_{\geq 0}$. Then the $n$-th power sum supersymmetric polynomials $q_{n}$ in $r$ commuting variables are given by
\[ q_{n} = x_{1}^{n}+x_{3}^{n}+\dots +x_{r_{o}}^{n} + (-1)^{n+1}\left( x_{2}^{n}+x_{4}^{n}+\dots +x_{r_{e}}^{n} \right).\]
It is immediate that any permutation among the odd indexed variables leaves $q_{n}$ invariant, and similarly for the even indexed variables. The sign $(-1)^{n+1}$ which appears also means that the cancellation property of \emph{\Cref{SSPDef}} is upheld, hence $q_{n} \in SS_{r}[x]$. It was shown in \cite{Stem85} that the algebra of supersymmetric polynomials is generated by all the supersymmetric power-sum polynomials, that is $SS_{r}[x]=\langle q_{n} \hspace{1mm} | \hspace{1mm} n\geq 0 \rangle$. 
\item[(2)] Let $r=4$. Then consider the polynomial
\[ l_{2}(x_{1},x_{2},x_{3},x_{4}) = x_{2}^{2}+x_{2}x_{4}+x_{4}^{2} + (x_{1}+x_{3})(x_{2}+x_{4}) + x_{1}x_{3}. \]
One can see that permuting the variables $x_{2}$ and $x_{4}$ around, or permuting $x_{1}$ and $x_{3}$ around leaves $l_{2}$ unchanged, hence it is parity symmetric. Furthermore, setting $x_{1}=-x_{2}=y$ gives
\begin{align*}
l_{2}(y,-y,x_{3},x_{4}) &= (-y)^{2}-yx_{4}+x_{4}^{2} + (y+x_{3})(-y+x_{4}) + yx_{3} \\
&= y^{2}-yx_{4}+x_{4}^{2}-y^{2}+yx_{4}-yx_{3}+x_{3}x_{4}+yx_{3} \\
&= x_{4}^{2}+x_{3}x_{4}
\end{align*}
which is independent of $y$. Thus $l_{2}$ also satisfies the cancellation property, and so $l_{2} \in SS_{4}[x]$.
\end{itemize}
\end{exs}

We will be particularly interested in what are called the elementary supersymmetric polynomials, of which $l_{2}$ from above is an example. To define these elements we will work in the algebra $\mathcal{P}_{r}[[t]]$ of formal power series in the commuting variable $t$ with coefficients in $\mathcal{P}_{r}$.

\begin{defn} \label{EleSSPDef}
Let $n,r\geq 0$ be non-negative integers. The elementary supersymmetric polynomials $l_{n}$ in $r$ commuting variables are defined to be the coefficients in the generating function
\[ \sum_{n=0}^{\infty}l_{n}t^{n} = \frac{\prod_{i \in O(r)}(1+x_{i}t)}{\prod_{j \in E(r)}(1-x_{j}t)} = \frac{(1+x_{1}t)(1+x_{3}t)\dots(1+x_{r_{o}}t)}{(1-x_{2}t)(1-x_{4}t)\dots(1-x_{r_{e}}t)}. \]
\end{defn}

The properties of supersymmetry are immediately seen to be upheld from the definition of $l_{n}$, showing that $l_{n} \in SS_{r}[x]$. Noting that $|E(r)| = \floor{r/2}$ and $|O(r)| = \ceil{r/2}$, where $\floor{-}$ and $\ceil{-}$ are the floor and ceiling functions respectively, then we may alternatively write
\[ \sum_{n=0}^{\infty}l_{n}t^{n} = \frac{\prod_{i=1}^{\ceil{r/2}}(1+x_{2i-1}t)}{\prod_{j=1}^{\floor{r/2}}(1-x_{2j}t)}. \]

\begin{ex} \label{EleSSPEx}
We will give a general expression for all $l_{n} \in SS_{4}[x]$. We have that
\begin{align*}
\sum_{n=0}^{\infty}l_{n}t^{n} &= \frac{(1+x_{1}t)(1+x_{3}t)}{(1-x_{2}t)(1-x_{4}t)} = (1+(x_{1}+x_{3})t+x_{1}x_{3}t^{2})\left(\frac{1}{1-x_{2}t}\right)\left(\frac{1}{1-x_{4}t}\right) \\
&= (1+(x_{1}+x_{3})t+x_{1}x_{3}t^{2})\left( \sum_{n=0}^{\infty}x_{2}^{n}t^{n} \right)\left( \sum_{n=0}^{\infty}x_{4}^{n}t^{n} \right) \\
&= (1+(x_{1}+x_{3})t+x_{1}x_{3}t^{2})\sum_{n=0}^{\infty}\left( \sum\limits_{\substack{a+b=n \\ a,b \geq 0}} x_{2}^{a}x_{4}^{b} \right)t^{n} \\
&= \sum_{n=0}^{\infty} \left[ \sum\limits_{\substack{a+b=n \\ a,b \geq 0}} x_{2}^{a}x_{4}^{b} +\left( (x_{1}+x_{3})\sum\limits_{\substack{a+b=n-1 \\ a,b \geq 0}} x_{2}^{a}x_{4}^{b} \right) + \left( x_{1}x_{3}\sum\limits_{\substack{a+b=n-2 \\ a,b \geq 0}} x_{2}^{a}x_{4}^{b} \right) \right]t^{n}.
\end{align*}
Therefore we have
\[ l_{n} = \sum\limits_{\substack{a+b=n \\ a,b \geq 0}} x_{2}^{a}x_{4}^{b} + \left( (x_{1}+x_{3})\sum\limits_{\substack{a+b=n-1 \\ a,b \geq 0}} x_{2}^{a}x_{4}^{b} \right) + \left( x_{1}x_{3}\sum\limits_{\substack{a+b=n-2 \\ a,b \geq 0}} x_{2}^{a}x_{4}^{b} \right). \]
The first four cases are $l_{0}=1$ and
\begin{align*}
l_{1} &= x_{1}+x_{2}+x_{3}+x_{4}, \\
l_{2} &= (x_{2}^{2} + x_{2}x_{4} + x_{4}^{2}) + (x_{1}+x_{3})(x_{2}+x_{4}) + x_{1}x_{3}, \\
l_{3} &= (x_{2}^{3} + x_{2}^{2}x_{4} + x_{2}x_{4}^{2} + x_{4}^{3}) + (x_{1}+x_{3})(x_{2}^{2} + x_{2}x_{4} + x_{4}^{2}) + x_{1}x_{3}(x_{2}+x_{4}).
\end{align*} 
In particular, the polynomial $l_{2}$ is precisely the one given in \emph{\Cref{SSPExs}} (2).
\end{ex}

The polynomials $l_{n}$ are the supersymmetric counterparts to the regular elementary symmetric polynomials. From the above example, one can observe that in $r$ commuting variables $l_{1} = x_{1}+x_{2}+\dots+x_{r}$, and so by \emph{\Cref{EnyRels}} (2) (iii), the element $l_{1}(N_{1},\dots,N_{r})$ is central in $\mathcal{A}_{r}(\delta)$. Unlike the regular elementary symmetric polynomials, when $r\geq 2$, we have that $l_{n} \neq 0$ for all $n>0$. In \cite{Stem85}, the following was proven:

\begin{thm}[\emph{Theorem 2; Corollary of \cite{Stem85}}] \label{StemSSPGens} \
The elementary supersymmetric polynomials in $r$ commuting variables generate $SS_{r}[x]$. That is 
\[SS_{r}[x] = \langle l_{n} \hspace{1mm} | \hspace{1mm} n\geq 0 \rangle. \]
\end{thm}

\subsection{ The elements $l_{n}(N_{1},\dots,N_{r})$ are central in $\mathcal{A}_{r}(\delta)$} \ 
\vspace{2mm}

Let $N_{1},N_{2},\dots,N_{r}$ be the normalised JM-elements of $\mathcal{A}_{r}(\delta)$ as described in \emph{Section 2.2}. We know that $N_{1},N_{2},\dots,N_{r}$ pairwise commute, which means there exists an algebra homomorphism $\mathcal{P}_{r} \rightarrow \mathcal{A}_{r}(\delta)$ given by $x_{i} \mapsto N_{i}$ for each $1 \leq i \leq r$. Restricting to the subalgebra $SS_{r}[x]$ yields an algebra homomorphism $SS_{r}[x] \rightarrow \mathcal{A}_{r}(\delta)$. We denote the image of this homomorphism by $SS_{r}[N_{1},\dots,N_{r}]$. In this section we will show that the subalgebra $SS_{r}[N_{1},\dots,N_{r}]$ belongs to the center of $\mathcal{A}_{r}(\delta)$. By definition,  the elements $l_{n}(N_{1},\dots,N_{r})$ in $SS_{r}[N_{1},\dots,N_{r}]$ are given by the generating function
\[ \sum_{n=0}^{\infty}l_{n}(N_{1},\dots,N_{r})t^{n} = \frac{\prod_{i=1}^{\ceil{r/2}}(1+N_{2i-1}t)}{\prod_{j=1}^{\floor{r/2}}(1-N_{2j}t)} \in \mathcal{A}_{r}(\delta)[[t]], \]
and they generate all of $SS_{r}[N_{1},\dots,N_{r}]$ by \emph{\Cref{StemSSPGens}}. Also recall that
\begin{align*}
\mathcal{A}_{r}(\delta) &= \langle e_{1},e_{2},\dots,e_{r-1},\sigma_{2},\sigma_{3},\dots,\sigma_{r-1} \rangle \\
&= \langle e_{1},e_{2},\dots,e_{r-1},s_{1},s_{2},\dots,s_{\floor{r/2}-1} \rangle.
\end{align*}
We will prove that $SS_{r}[N_{1},\dots,N_{r}] \subseteq Z(\mathcal{A}_{r}(\delta))$ directly by showing that each generator of $\mathcal{A}_{r}(\delta)$ commutes with $l_{n}(N_{1},\dots,N_{r})$, for any $n \geq 0$. This will be done by using the above generating function and the various relations in $\mathcal{A}_{r}(\delta)$ that we established previously. Hence we will be working within $\mathcal{A}_{r}(\delta)[[t]]$, the algebra of formal power series in the commuting variable $t$ with coefficients in $\mathcal{A}_{r}(\delta)$. We begin with the generators $e_{i}$.

\begin{lem} \label{EiCommute}
In $\mathcal{A}_{r}(\delta)$, we have that
\[ e_{i}l_{n}(N_{1},\dots,N_{r})=l_{n}(N_{1},\dots,N_{r})e_{i} \]
for all $1\leq i \leq r-1$ and $n\geq 0$.
\end{lem}

\begin{proof}
We prove this for $e_{2i}$. The case for $e_{2i-1}$ follows in the same manner. From \emph{\Cref{CommutingRels}} (1) we know that $e_{2i}N_{j}=N_{j}e_{2i}$ for all $j \neq 2i,2i+1$. Hence it suffices to show that
\[ e_{2i}\frac{1+N_{2i+1}t}{1-N_{2i}t} = \frac{1+N_{2i+1}t}{1-N_{2i}t}e_{2i}. \]
By \emph{\Cref{AntiSymmRels}} we know that $e_{2i}N_{2i}=-e_{2i}N_{2i+1}$ and $N_{2i}e_{2i}=-N_{2i+1}e_{2i}$. Therefore we have that $e_{2i}(1+N_{2i+1}t)=e_{2i}(1-N_{2i}t)$ and $(1+N_{2i+1}t)e_{2i}=(1-N_{2i}t)e_{2i}$. Thus the above equality holds since both sides are equal to $e_{2i}$.

\end{proof}

We now wish to show that $s_{i}l_{n}(N_{1},\dots,N_{r})=l_{n}(N_{1},\dots,N_{r})s_{i}$. However, it turns out that we will need to first establishing that $l_{n}(N_{1},\dots,N_{r})$ commutes with $\sigma_{2i+1}$.

\begin{prop} \label{OddSigCommute}
In $\mathcal{A}_{r}(\delta)$ we have that
\[ \sigma_{2i+1}l_{n}(N_{1},\dots,N_{r}) = l_{n}(N_{1},\dots,N_{r})\sigma_{2i+1}, \]
for any $n\geq 0$ and $1\leq 2i+1 \leq r-1$. 
\end{prop}

\begin{proof}
By \emph{\Cref{CommutingRels}} (3) we have that $\sigma_{2i+1}N_{j}=N_{j}\sigma_{2i+1}$ for all $j\neq 2i,2i+1,2i+2$. Hence to prove this proposition it suffices to show
\[ \sigma_{2i+1}\frac{1+N_{2i+1}t}{(1-N_{2i}t)(1-N_{2i+2}t)} = \frac{1+N_{2i+1}t}{(1-N_{2i}t)(1-N_{2i+2}t)}\sigma_{2i+1}. \]
To do this we will start with the left hand side and use the relations of \emph{\Cref{OddSigSkeinRels}} to pull $\sigma_{2i+1}$ to the right. In doing so we pick up many additional terms, and showing that such terms all cancel out will complete the proof. To begin, we have that
\begin{align*}
\sigma_{2i+1}(1+N_{2i+1}t) &= \sigma_{2i+1} + (\sigma_{2i+1}N_{2i+1})t \\
&= \sigma_{2i+1} + (N_{2i+1}\sigma_{2i+1} + \sigma_{2i+1}e_{2i+1}e_{2i} + e_{2i}e_{2i+1} - e_{2i+1}e_{2i} - e_{2i}e_{2i+1}\sigma_{2i+1})t \\
&= (1+N_{2i+1}t)\sigma_{2i+1} + (\sigma_{2i+1}e_{2i+1}e_{2i} + e_{2i}e_{2i+1} - e_{2i+1}e_{2i} - e_{2i}e_{2i+1}\sigma_{2i+1})t
\end{align*}
where the second equality follows from \emph{\Cref{OddSigSkeinRels}} (2) and the relations $\sigma_{2i+1}e_{2i}=e_{2i}$ and $\sigma_{2i+1}^{2}=1$. From this we have
\begin{align*}
\sigma_{2i+1}\frac{1+N_{2i+1}t}{(1-N_{2i}t)(1-N_{2i+2}t)} &= \sigma_{2i+1}(1+N_{2i+1}t)\frac{1}{(1-N_{2i}t)(1-N_{2i+2}t)} \\
&= (1+N_{2i+1}t)\sigma_{2i+1}\frac{1}{(1-N_{2i}t)(1-N_{2i+2}t)} + C
\end{align*}
where $C = C_{1}+C_{2}+C_{3}+C_{4}$ with
\begin{align*}
C_{1} &:= \sigma_{2i+1}e_{2i+1}e_{2i}\frac{t}{(1-N_{2i}t)(1-N_{2i+2}t)}, \\
C_{2} &:= e_{2i}e_{2i+1}\frac{t}{(1-N_{2i}t)(1-N_{2i+2}t)}, \\
C_{3} &:= -e_{2i+1}e_{2i}\frac{t}{(1-N_{2i}t)(1-N_{2i+2}t)}, \\
C_{4} &:= -e_{2i}e_{2i+1}\sigma_{2i+1}\frac{t}{(1-N_{2i}t)(1-N_{2i+2}t)}.
\end{align*}
Next we have that
\begin{align*}
\sigma_{2i+1}(1-N_{2i+2}t) &= \sigma_{2i+1} - (\sigma_{2i+1}N_{2i+2})t \\
&= \sigma_{2i+1} - (N_{2i}\sigma_{2i+1} - \sigma_{2i+1}e_{2i+1}e_{2i} - e_{2i}e_{2i+1} + e_{2i}e_{2i+1}\sigma_{2i+1}e_{2i+1}e_{2i} + 1)t \\
&= (1-N_{2i}t)\sigma_{2i+1} + (\sigma_{2i+1}e_{2i+1}e_{2i} + e_{2i}e_{2i+1} - e_{2i}e_{2i+1}\sigma_{2i+1}e_{2i+1}e_{2i} - 1)t
\end{align*}
where the second equality follows from \emph{\Cref{OddSigSkeinRels}} (1) and the relations $\sigma_{2i+1}e_{2i}=e_{2i}$ and $\sigma_{2i+1}^{2}=1$. Then multiplying on the left by $(1-N_{2i}t)^{-1}$ and on the right by $(1-N_{2i+2}t)^{-1}$, then rearranging yields
\begin{align*}
\sigma_{2i+1}\frac{1}{1-N_{2i+2}t} = \frac{1}{1-N_{2i}t}\sigma_{2i+1} &+ \frac{1}{1-N_{2i}t}\Big(-\sigma_{2i+1}e_{2i+1}e_{2i} \\
&- e_{2i}e_{2i+1} + e_{2i}e_{2i+1}\sigma_{2i+1}e_{2i+1}e_{2i} + 1\Big)\frac{t}{1-N_{2i+2}t}
\end{align*}
From this we have that
\begin{align*}
(1+N_{2i+1}t)\sigma_{2i+1}\frac{1}{(1-N_{2i}t)(1-N_{2i+2}t)} + C &= (1+N_{2i+1}t)\left(\sigma_{2i+1}\frac{1}{1-N_{2i+2}t}\right)\frac{1}{1-N_{2i}t} + C \\
&= \frac{1+N_{2i+1}t}{1-N_{2i}t}\sigma_{2i+1}\frac{1}{1-N_{2i}t} + D + C,
\end{align*}
where $D=D_{1}+D_{2}+D_{3}+D_{4}$ with
\begin{align*}
D_{1} &:= -\frac{1+N_{2i+1}t}{1-N_{2i}t}\sigma_{2i+1}e_{2i+1}e_{2i}\frac{t}{(1-N_{2i}t)(1-N_{2i+2}t)},  \\
D_{2} &:= -\frac{1+N_{2i+1}t}{1-N_{2i}t}e_{2i}e_{2i+1}\frac{t}{(1-N_{2i}t)(1-N_{2i+2}t)}, \\
D_{3} &:= \frac{1+N_{2i+1}t}{1-N_{2i}t}e_{2i}e_{2i+1}\sigma_{2i+1}e_{2i+1}e_{2i}\frac{t}{(1-N_{2i}t)(1-N_{2i+2}t)}, \\
D_{4} &:= \frac{1+N_{2i+1}t}{1-N_{2i}t}\frac{t}{(1-N_{2i}t)(1-N_{2i+2}t)}.
\end{align*}
Lastly we have that
\begin{align*}
\sigma_{2i+1}(1-N_{2i}t) &= \sigma_{2i+1} - (\sigma_{2i+1}N_{2i})t \\
&= \sigma_{2i+1} - (N_{2i+2}\sigma_{2i+1} + e_{2i+1}e_{2i} + e_{2i}e_{2i+1}\sigma_{2i+1} - e_{2i}e_{2i+1}\sigma_{2i+1}e_{2i+1}e_{2i} - 1)t \\
&= (1-N_{2i+2}t)\sigma_{2i+1} + (-e_{2i+1}e_{2i} - e_{2i}e_{2i+1}\sigma_{2i+1} + e_{2i}e_{2i+1}\sigma_{2i+1}e_{2i+1}e_{2i} + 1)t
\end{align*}
where the second equality follows from \emph{\Cref{OddSigSkeinRels}} (1) and the relations $e_{2i}\sigma_{2i+1}=e_{2i}$ and $\sigma_{2i+1}^{2}=1$. Then multiplying on the left by $(1-N_{2i+2}t)^{-1}$ and on the right by $(1-N_{2i}t)^{-1}$, then rearranging yields
\begin{align*}
\sigma_{2i+1}\frac{1}{1-N_{2i}t} = \frac{1}{1-N_{2i+2}t}\sigma_{2i+1} &+ \frac{1}{1-N_{2i+2}t}\Big(e_{2i+1}e_{2i} \\
&+ e_{2i}e_{2i+1}\sigma_{2i+1} - e_{2i}e_{2i+1}\sigma_{2i+1}e_{2i+1}e_{2i} - 1\Big)\frac{t}{1-N_{2i}t}
\end{align*}
From this we have that
\[ \frac{1+N_{2i+1}t}{1-N_{2i}t}\sigma_{2i+1}\frac{1}{1-N_{2i}t} + D + C = \frac{1+N_{2i+1}t}{(1-N_{2i}t)(1-N_{2i+2}t)}\sigma_{2i+1} + E + D + C, \]
where $E=E_{1}+E_{2}+E_{3}+E_{4}$ with 
\begin{align*}
E_{1} &:= \frac{1+N_{2i+1}t}{(1-N_{2i}t)(1-N_{2i+2}t)}e_{2i+1}e_{2i}\frac{t}{1-N_{2i}t},  \\
E_{2} &:= \frac{1+N_{2i+1}t}{(1-N_{2i}t)(1-N_{2i+2}t)}e_{2i}e_{2i+1}\sigma_{2i+1}\frac{t}{1-N_{2i}t},  \\
E_{3} &:= -\frac{1+N_{2i+1}t}{(1-N_{2i}t)(1-N_{2i+2}t)}e_{2i}e_{2i+1}\sigma_{2i+1}e_{2i+1}e_{2i}\frac{t}{1-N_{2i}t},  \\
E_{4} &:= -\frac{1+N_{2i+1}t}{(1-N_{2i}t)(1-N_{2i+2}t)}\frac{t}{1-N_{2i}t},  \\
\end{align*}
Thus collectively we have
\[ \sigma_{2i+1}\frac{1+N_{2i+1}t}{(1-N_{2i}t)(1-N_{2i+2}t)} = \frac{1+N_{2i+1}t}{(1-N_{2i}t)(1-N_{2i+2}t)}\sigma_{2i+1} + E + D + C, \]  
and so the result will follow if we can show that $E+D+C=0$. We do this by showing that the terms $C_{i},D_{i},E_{i}$ pair up into additive complements:

$(C_{1}=-D_{1})$:
\begin{align*}
D_{1} &= -\frac{1+N_{2i+1}t}{1-N_{2i}t}\sigma_{2i+1}e_{2i+1}e_{2i}\frac{t}{(1-N_{2i}t)(1-N_{2i+2}t)} \\
&= -\frac{1+N_{2i+1}t}{1-N_{2i}t}L_{2i}e_{2i}\frac{t}{(1-N_{2i}t)(1-N_{2i+2}t)} \hspace{18mm} (\text{by } \Cref{EnyRels} (3)(iv)) \\
&= -L_{2i}\frac{1+N_{2i+1}t}{1-N_{2i}t}e_{2i}\frac{t}{(1-N_{2i}t)(1-N_{2i+2}t)} \hspace{18mm} (\text{JM-elements commute}) \\
&= -L_{2i}e_{2i}\frac{t}{(1-N_{2i}t)(1-N_{2i+2}t)} \hspace{35mm} (\text{by } \Cref{AntiSymmRels}) \\
&= -\sigma_{2i+1}e_{2i+1}e_{2i}\frac{t}{(1-N_{2i}t)(1-N_{2i+2}t)} \hspace{24mm} (\text{by } \Cref{EnyRels} (3)(iv)) \\
&= -C_{1}.
\end{align*} 

$(C_{2}=-D_{2})$:
\begin{align*}
D_{2} &= -\frac{1+N_{2i+1}t}{1-N_{2i}t}e_{2i}e_{2i+1}\frac{t}{(1-N_{2i}t)(1-N_{2i+2}t)} \\
&= -e_{2i}e_{2i+1}\frac{t}{(1-N_{2i}t)(1-N_{2i+2}t)} \hspace{25mm} (\text{by } \Cref{AntiSymmRels}) \\
&= -C_{2}.
\end{align*}

$(C_{3}=-E_{1})$: Firstly we have that
\begin{align*}
e_{2i+1}e_{2i}N_{2i+2} &= e_{2i+1}N_{2i+2}e_{2i} \hspace{9mm} (\text{by } \Cref{CommutingRels} (1)) \\
&= - e_{2i+1}N_{2i+1}e_{2i} \hspace{6mm} (\text{by } \Cref{AntiSymmRels}) \\
&= e_{2i+1}N_{2i}e_{2i} \hspace{12mm} (\text{by } \Cref{AntiSymmRels}) \\
&= N_{2i}e_{2i+1}e_{2i} \hspace{12mm} (\text{by } \Cref{CommutingRels} (1)).
\end{align*}
As such $e_{2i+1}e_{2i}(1-N_{2i+2}t) = (1-N_{2i}t)e_{2i+1}e_{2i}$ and so
\[ \frac{1}{1-N_{2i}t}e_{2i+1}e_{2i} = e_{2i+1}e_{2i}\frac{1}{1-N_{2i+2}t}. \tag{Eq2} \]
Using this we see that
\begin{align*}
E_{1} &= \frac{1+N_{2i+1}t}{(1-N_{2i}t)(1-N_{2i+2}t)}e_{2i+1}e_{2i}\frac{t}{1-N_{2i}t} \\
&= \frac{1}{1-N_{2i}t}e_{2i+1}e_{2i}\frac{t}{1-N_{2i}t} \hspace{32mm} (\text{by } \Cref{AntiSymmRels}) \\
&= e_{2i+1}e_{2i}\frac{t}{(1-N_{2i}t)(1-N_{2i+2}t)} \hspace{24mm} (From \text{ Eq2}) \\
&= -C_{3}.
\end{align*}

$(C_{4}=-E_{2})$:
\begin{align*}
E_{2} &= \frac{1+N_{2i+1}t}{(1-N_{2i}t)(1-x
N_{2i+2}t)}e_{2i}e_{2i+1}\sigma_{2i+1}\frac{t}{1-N_{2i}t}  \\
&= \frac{1+N_{2i+1}t}{(1-N_{2i}t)(1-N_{2i+2}t)}e_{2i}L_{2i}\frac{t}{1-N_{2i}t} \hspace{19.5mm} (\text{by } \Cref{EnyRels} (3)(iv)) \\
&= \frac{1}{1-N_{2i+2}t}e_{2i}L_{2i}\frac{t}{1-N_{2i}t} \hspace{37mm} (\text{by } \Cref{AntiSymmRels}) \\
&= e_{2i}L_{2i}\frac{t}{(1-N_{2i}t)(1-N_{2i+2}t)} \hspace{32mm} (\text{$N_{2i+2}$ commutes with $e_{2i}$ and $L_{2i}$}) \\
&= e_{2i}e_{2i+1}\sigma_{2i+1}\frac{t}{(1-N_{2i}t)(1-N_{2i+2}t)} \hspace{21mm} (\text{by } \Cref{EnyRels} (3)(iv)) \\
&= -C_{4}.
\end{align*}

$(E_{3}=-D_{3})$: First of all we have that
\begin{align*}
N_{2i+2}e_{2i}e_{2i+1}\sigma_{2i+1}e_{2i+1}e_{2i} &= N_{2i+2}e_{2i}L_{2i}e_{2i+1}e_{2i} \hspace{15mm} (\text{by } \Cref{EnyRels} (3)(iv)) \\
&= e_{2i}L_{2i}N_{2i+2}e_{2i+1}e_{2i} \hspace{15mm} (\text{$N_{2i+2}$ commutes with $e_{2i}$ and $L_{2i}$}) \\
&= -e_{2i}L_{2i}N_{2i+1}e_{2i+1}e_{2i} \hspace{12.5mm} (\text{by } \Cref{AntiSymmRels}) \\
&= e_{2i}N_{2i}L_{2i}e_{2i+1}e_{2i} \hspace{19mm} (\text{by } \Cref{AntiSymmRels}) \\
&= e_{2i}L_{2i}e_{2i+1}N_{2i}e_{2i} \hspace{19mm} (\text{$N_{2i}$ commutes with $e_{2i+1}$ and $L_{2i}$}) \\
&= e_{2i}L_{2i}e_{2i+1}N_{2i+2}e_{2i} \hspace{15.5mm} (\text{by } 2 \times \Cref{AntiSymmRels}) \\
&= e_{2i}L_{2i}e_{2i+1}e_{2i}N_{2i+2} \hspace{15.5mm} (\text{by } \Cref{CommutingRels} (1)) \\
&= e_{2i}e_{2i+1}\sigma_{2i+1}e_{2i+1}e_{2i}N_{2i+2} \hspace{4.5mm} (\text{by } \Cref{EnyRels} (3)(iv)).
\end{align*}
Hence $(1-N_{2i+2}t)e_{2i}e_{2i+1}\sigma_{2i+1}e_{2i+1}e_{2i} = e_{2i}e_{2i+1}\sigma_{2i+1}e_{2i+1}e_{2i}(1-N_{2i+2}t)$. Therefore
\[ \frac{1}{1-N_{2i+2}t}e_{2i}e_{2i+1}\sigma_{2i+1}e_{2i+1}e_{2i} = e_{2i}e_{2i+1}\sigma_{2i+1}e_{2i+1}e_{2i}\frac{1}{1-N_{2i+2}t}. \]
From this we have that
\begin{align*}
E_{3} &= -\frac{1+N_{2i+1}t}{(1-N_{2i}t)(1-N_{2i+2}t)}e_{2i}e_{2i+1}\sigma_{2i+1}e_{2i+1}e_{2i}\frac{t}{1-N_{2i}t} \\
&= -\frac{1+N_{2i+1}t}{1-N_{2i}t}e_{2i}e_{2i+1}\sigma_{2i+1}e_{2i+1}e_{2i}\frac{t}{(1-N_{2i}t)(1-N_{2i+2}t)}  \\
&= -D_{3}.
\end{align*}

$(E_{4}=-D_{4})$: This is immediate.

\vspace{2mm}
\noindent
Altogether we have that
\begin{align*}
C_{1} &= -D_{1}, \hspace{3mm} C_{3} = -E_{1}, \hspace{3mm} E_{3} = -D_{3}, \\
C_{2} &= -D_{2}, \hspace{3mm} C_{4} = -E_{2}, \hspace{3mm} E_{4} = -D_{4}.
\end{align*}
Hence $E+D+C=0$ as required.

\end{proof}

Since $\mathcal{A}_{r}(\delta) = \langle e_{1},e_{2},\dots,e_{r-1},\sigma_{2},\sigma_{3},\dots,\sigma_{r-1} \rangle$, we could attempt to show that $\sigma_{2i}$ commutes with any $l_{n}(N_{1},\dots,N_{r})$ in the same manner done above. However, it turns out to be very difficult to show that the additional terms which appear cancel out. So instead we prove that $s_{i}$ and $l_{n}(N_{1},\dots,N_{r})$ commute. We will show this is the exact same manner as done above with $\sigma_{2i+1}$, and we will use the above proposition only once. This may seem excessive, but it appears necessary if one wants to prove it in this manner (see \emph{\Cref{CoxCommuteRmk}}).

\begin{rmk} \label{JK17CoxCommuteRmk}
In \cite{JK17} they showed that the Coxeter generators commute with supersymmetric polynomials in the JM-elements by showing that they commute with symmetric polynomials in $L_{1},\dots,L_{r}$, and that they commute with symmetric polynomials in $L_{r+1},\dots,L_{r+s}$. This was done by checking that $s_{i}$ commutes with both $L_{i}+L_{i+1}$ and $L_{i}L_{i+1}$ for all $i\neq r$. Since the supersymmetric power-sum polynomials are a sum of symmetric polynomials in the two sets of variables, it then follows that $s_{i}$ commutes with them. So in fact the Coxeter generators commuted with the larger algebra of polynomials symmetric in both sets of variables. For our situation, the analogous approach would be to show that symmetric polynomials in the odd indexed normalised JM-elements commute with the Coxeter generators, and similarly for the even index normalised JM-elements. However this is not true. From \emph{\Cref{CoxSkeinRels}} one can show that the commutators $[s_{i},L_{2i-1}+L_{2i+1}]$ and $[s_{i},L_{2i}+L_{2i+1}]$ are non-zero, and in fact we have
\[ [s_{i},L_{2i-1}+L_{2i+1}] = -[s_{i},L_{2i}+L_{2i+1}]. \]
In our setting, the cancellation property is needed for the Coxeter generators to commute.
\end{rmk}

\begin{prop} \label{CoxCommute}
In $\mathcal{A}_{r}(\delta)$ we have that
\[ s_{i}l_{n}(N_{1},\dots,N_{r}) = l_{n}(N_{1},\dots,N_{r})s_{i}, \]
for any $n\geq 0$ and $1\leq i \leq \floor{r/2}-1$. 
\end{prop}

\begin{proof}
From \emph{\Cref{CommutingRels}} (2) we know that $s_{i}N_{j}=N_{j}s_{i}$ for all $j \neq 2i-1,2i,2i+1,2i+2$, hence to prove the result it suffices to show that
\[ s_{i}\frac{(1+N_{2i-1}t)(1+N_{2i+1}t)}{(1-N_{2i}t)(1-N_{2i+2}t)} = \frac{(1+N_{2i-1}t)(1+N_{2i+1}t)}{(1-N_{2i}t)(1-N_{2i+2}t)}s_{i}. \]
We start from the right hand side, pull $s_{i}$ through using the relations of \emph{\Cref{CoxSkeinRels}}, and then show that the additional terms which appear all cancel out. Firstly we have that
\begin{align*}
s_{i}(1-N_{2i}t) &= s_{i} - (s_{i}N_{2i})t \\
&= s_{i} - (N_{2i+2}s_{i} + \sigma_{2i+1}e_{2i-1}e_{2i} + e_{2i}e_{2i-1}\sigma_{2i} - e_{2i}e_{2i+1}\sigma_{2i+1}e_{2i+1}e_{2i} - \sigma_{2i})t \\
&= (1-N_{2i+2}t)s_{i} + (-\sigma_{2i+1}e_{2i-1}e_{2i} - e_{2i}e_{2i-1}\sigma_{2i} + e_{2i}e_{2i+1}\sigma_{2i+1}e_{2i+1}e_{2i} + \sigma_{2i})t
\end{align*}
where the second equality follows from \emph{\Cref{CoxSkeinRels}} (2). From this we have that
\begin{align*}
\frac{1}{1-N_{2i+2}t}s_{i} = s_{i}\frac{1}{1-N_{2i}t} &+ \frac{1}{1-N_{2i+2}t}\Big(-\sigma_{2i+1}e_{2i-1}e_{2i} \\
&- e_{2i}e_{2i-1}\sigma_{2i} + e_{2i}e_{2i+1}\sigma_{2i+1}e_{2i+1}e_{2i} + \sigma_{2i}\Big)\frac{t}{1-N_{2i}t}.
\end{align*}
Hence we obtain
\begin{align*}
\frac{(1+N_{2i-1}t)(1+N_{2i+1}t)}{(1-N_{2i}t)(1-N_{2i+2}t)}s_{i} &= \frac{(1+N_{2i-1}t)(1+N_{2i+1}t)}{1-N_{2i}t} \left(\frac{1}{1-N_{2i+2}t}s_{i} \right) \\
&= \frac{(1+N_{2i-1}t)(1+N_{2i+1}t)}{1-N_{2i}t}s_{i}\frac{1}{1-N_{2i}t} + C,
\end{align*}
where $C = C_{1}+C_{2}+C_{3}+C_{4}$ with
\begin{align*}
C_{1} &:= -\frac{(1+N_{2i-1}t)(1+N_{2i+1}t)}{(1-N_{2i}t)(1-N_{2i+2}t)}\sigma_{2i+1}e_{2i-1}e_{2i}\frac{t}{1-N_{2i}t}, \\
\\
C_{2} &:= -\frac{(1+N_{2i-1}t)(1+N_{2i+1}t)}{(1-N_{2i}t)(1-N_{2i+2}t)}e_{2i}e_{2i-1}\sigma_{2i}\frac{t}{1-N_{2i}t}, \\
\\
C_{3} &:= \frac{(1+N_{2i-1}t)(1+N_{2i+1}t)}{(1-N_{2i}t)(1-N_{2i+2}t)}e_{2i}e_{2i+1}\sigma_{2i+1}e_{2i+1}e_{2i}\frac{t}{1-N_{2i}t}, \\
\\
C_{4} &:= \frac{(1+N_{2i-1}t)(1+N_{2i+1}t)}{(1-N_{2i}t)(1-N_{2i+2}t)}\sigma_{2i}\frac{t}{1-N_{2i}t}.
\end{align*}
Next we have that
\begin{align*}
s_{i}(1-N_{2i+2}t) &= s_{i} - (s_{i}N_{2i+2})t \\
&= s_{i} - (N_{2i}s_{i} - \sigma_{2i}e_{2i-1}e_{2i} - e_{2i}e_{2i-1}\sigma_{2i+1} + e_{2i}e_{2i+1}\sigma_{2i+1}e_{2i+1}e_{2i} + \sigma_{2i})t \\
&= (1-N_{2i}t)s_{i} + (\sigma_{2i}e_{2i-1}e_{2i} + e_{2i}e_{2i-1}\sigma_{2i+1} - e_{2i}e_{2i+1}\sigma_{2i+1}e_{2i+1}e_{2i} - \sigma_{2i})t
\end{align*}
where the second equality follows from \emph{\Cref{CoxSkeinRels}} (2). From this we have that
\begin{align*}
\frac{1}{1-N_{2i}t}s_{i} = s_{i}\frac{1}{1-N_{2i+2}t} &+ \frac{1}{1-N_{2i}t}\Big(\sigma_{2i}e_{2i-1}e_{2i} \\
&+ e_{2i}e_{2i-1}\sigma_{2i+1} - e_{2i}e_{2i+1}\sigma_{2i+1}e_{2i+1}e_{2i} - \sigma_{2i}\Big)\frac{t}{1-N_{2i+2}t}.
\end{align*}
With this we obtain
\begin{align*}
(1+N_{2i-1}t)(1+N_{2i+1}t)&\left(\frac{1}{1-N_{2i}t}s_{i}\right)\frac{1}{1-N_{2i}t} + C \\
\\
&= (1+N_{2i-1}t)(1+N_{2i+1}t)s_{i}\frac{1}{(1-N_{2i}t)(1-N_{2i+2}t)} + D + C,
\end{align*}
where $D = D_{1}+D_{2}+D_{3}+D_{4}$ with
\begin{align*}
D_{1} &:= \frac{(1+N_{2i-1}t)(1+N_{2i+1}t)}{1-N_{2i}t}\sigma_{2i}e_{2i-1}e_{2i}\frac{t}{(1-N_{2i}t)(1-N_{2i+2}t)}, \\
\\
D_{2} &:= \frac{(1+N_{2i-1}t)(1+N_{2i+1}t)}{1-N_{2i}t}e_{2i}e_{2i-1}\sigma_{2i+1}\frac{t}{(1-N_{2i}t)(1-N_{2i+2}t)}, \\
\\
D_{3} &:= -\frac{(1+N_{2i-1}t)(1+N_{2i+1}t)}{1-N_{2i}t}e_{2i}e_{2i+1}\sigma_{2i+1}e_{2i+1}e_{2i}\frac{t}{(1-N_{2i}t)(1-N_{2i+2}t)}, \\
\\
D_{4} &:= -\frac{(1+N_{2i-1}t)(1+N_{2i+1}t)}{1-N_{2i}t}\sigma_{2i}\frac{t}{(1-N_{2i}t)(1-N_{2i+2}t)}.
\end{align*}
Next we have that
\begin{align*}
(1+N_{2i+1}t)s_{i} &= s_{i} + (N_{2i+1}s_{i})t \\
&= s_{i} + (s_{i}N_{2i-1} - \sigma_{2i}e_{2i-1}e_{2i} - e_{2i}e_{2i-1}\sigma_{2i+1} + e_{2i}e_{2i-1}\sigma_{2i}e_{2i-1}e_{2i} + \sigma_{2i+1})t \\
&= s_{i}(1+N_{2i-1}t) + (-\sigma_{2i}e_{2i-1}e_{2i} - e_{2i}e_{2i-1}\sigma_{2i+1} + e_{2i}e_{2i-1}\sigma_{2i}e_{2i-1}e_{2i} + \sigma_{2i+1})t
\end{align*}
where the second equality follows from \emph{\Cref{CoxSkeinRels}} (1). With this we have that
\begin{align*}
(1+N_{2i-1}t)(1+N_{2i+1}t)s_{i}&\frac{1}{(1-N_{2i}t)(1-N_{2i+2}t)} + D + C \\
\\
&= (1+N_{2i-1}t)s_{i}(1+N_{2i-1}t)\frac{1}{(1-N_{2i}t)(1-N_{2i+2}t)}  + E + D + C,
\end{align*}
where $E=E_{1}+E_{2}+E_{3}+E_{4}$ with
\begin{align*}
E_{1} &:= -(1+N_{2i-1}t)\sigma_{2i}e_{2i-1}e_{2i}\frac{t}{(1-N_{2i}t)(1-N_{2i+2}t)}, \\
\\
E_{2} &:= -(1+N_{2i-1}t)e_{2i}e_{2i-1}\sigma_{2i+1}\frac{t}{(1-N_{2i}t)(1-N_{2i+2}t)}, \\
\\
E_{3} &:= (1+N_{2i-1}t)e_{2i}e_{2i-1}\sigma_{2i}e_{2i-1}e_{2i}\frac{t}{(1-N_{2i}t)(1-N_{2i+2}t)}, \\
\\
E_{4} &:= (1+N_{2i-1}t)\sigma_{2i+1}\frac{t}{(1-N_{2i}t)(1-N_{2i+2}t)}.
\end{align*}
Lastly we have that
\begin{align*}
(1+N_{2i-1}t)s_{i} &= s_{i} + (N_{2i-1}s_{i})t \\
&= s_{i} + (s_{i}N_{2i+1} + \sigma_{2i+1}e_{2i-1}e_{2i} + e_{2i}e_{2i-1}\sigma_{2i} - e_{2i}e_{2i-1}\sigma_{2i}e_{2i-1}e_{2i} - \sigma_{2i+1})t \\
&= s_{i}(1+N_{2i+1}t) + (\sigma_{2i+1}e_{2i-1}e_{2i} + e_{2i}e_{2i-1}\sigma_{2i} - e_{2i}e_{2i-1}\sigma_{2i}e_{2i-1}e_{2i} - \sigma_{2i+1})t
\end{align*}
where the second equality follows from \emph{\Cref{CoxSkeinRels}} (1). With this we have that
\begin{align*}
(1+N_{2i-1}t)s_{i}(1+N_{2i-1}t)&\frac{1}{(1-N_{2i}t)(1-N_{2i+2}t)}  + E + D + C \\
\\
&= s_{i}\frac{(1+N_{2i-1}t)(1+N_{2i+1}t)}{(1-N_{2i}t)(1-N_{2i+2}t)} + F + E + D + C
\end{align*}
where $F=F_{1}+F_{2}+F_{3}+F_{4}$ with
\begin{align*}
F_{1} &:= \sigma_{2i+1}e_{2i-1}e_{2i}\frac{(1+N_{2i-1}t)t}{(1-N_{2i}t)(1-N_{2i+2}t)}, \\
\\
F_{2} &:= e_{2i}e_{2i-1}\sigma_{2i}\frac{(1+N_{2i-1}t)t}{(1-N_{2i}t)(1-N_{2i+2}t)}, \\
\\
F_{3} &:= -e_{2i}e_{2i-1}\sigma_{2i}e_{2i-1}e_{2i}\frac{(1+N_{2i-1}t)t}{(1-N_{2i}t)(1-N_{2i+2}t)}, \\
\\
F_{4} &:= -\sigma_{2i+1}\frac{(1+N_{2i-1}t)t}{(1-N_{2i}t)(1-N_{2i+2}t)}.
\end{align*}
Thus collectively we have
\[ \frac{(1+N_{2i-1}t)(1+N_{2i+1}t)}{(1-N_{2i}t)(1-N_{2i+2}t)}s_{i} = s_{i}\frac{(1+N_{2i-1}t)(1+N_{2i+1}t)}{(1-N_{2i}t)(1-N_{2i+2}t)} + F + E + D + C, \]
and so the result will follow if we can show that $F+E+D+C=0$. We do this by showing that the terms $C_{i},D_{i},E_{i},F_{i}$ pair up into additive complements:

$(C_{1} = -F_{1})$:
\begin{align*}
C_{1} &= -\frac{(1+N_{2i-1}t)(1+N_{2i+1}t)}{(1-N_{2i}t)(1-N_{2i+2}t)}\sigma_{2i+1}e_{2i-1}e_{2i}\frac{t}{1-N_{2i}t}, \\
&= -\frac{1+N_{2i+1}t}{(1-N_{2i}t)(1-N_{2i+2}t)}\sigma_{2i+1}(1+N_{2i-1}t)e_{2i-1}e_{2i}\frac{t}{1-N_{2i}t}, \hspace{4mm} (\text{by } \Cref{CommutingRels} (3)) \\
&= -\sigma_{2i+1}\frac{(1+N_{2i-1}t)(1+N_{2i+1}t)}{(1-N_{2i}t)(1-N_{2i+2}t)}e_{2i-1}e_{2i}\frac{t}{1-N_{2i}t}, \hspace{19mm} (\text{by } \Cref{OddSigCommute}) \\
&= -\sigma_{2i+1}\frac{1+N_{2i+1}t}{1-N_{2i+2}t}e_{2i-1}e_{2i}\frac{t}{1-N_{2i}t}, \hspace{40mm} (\text{by } \Cref{AntiSymmRels}) \\
&= -\sigma_{2i+1}e_{2i-1}(1+N_{2i+1}t)e_{2i}\frac{t}{(1-N_{2i}t)(1-N_{2i+2}t)}, \hspace{17.5mm} (\text{by } \Cref{CommutingRels} (1)) \\
&= -\sigma_{2i+1}e_{2i-1}(1+N_{2i-1}t)e_{2i}\frac{t}{(1-N_{2i}t)(1-N_{2i+2}t)}, \hspace{17.5mm} (\text{by } 2 \times \Cref{AntiSymmRels}) \\
&= -\sigma_{2i+1}e_{2i-1}e_{2i}\frac{(1+N_{2i-1}t)t}{(1-N_{2i}t)(1-N_{2i+2}t)}, \hspace{35.5mm} (\text{by } \Cref{CommutingRels} (1)) \\
&= -F_{1}.
\end{align*}

$(C_{2} = -F_{2})$:
\begin{align*}
C_{2} &= -\frac{(1+N_{2i-1}t)(1+N_{2i+1}t)}{(1-N_{2i}t)(1-N_{2i+2}t)}e_{2i}e_{2i-1}\sigma_{2i}\frac{t}{1-N_{2i}t} \\
&= -\frac{(1+N_{2i-1}t)(1+N_{2i+1}t)}{(1-N_{2i}t)(1-N_{2i+2}t)}e_{2i}L_{2i}\frac{t}{1-N_{2i}t} \hspace{12mm} (\text{by } \Cref{EnyRels} (3)(iii)) \\
&= -\frac{1+N_{2i+1}t}{1-N_{2i}t}e_{2i}L_{2i}\frac{(1+N_{2i-1}t))t}{(1-N_{2i}t)(1-N_{2i+2}t)} \hspace{12mm} (\text{by } \Cref{CommutingRels} (1)) \\
&= -e_{2i}L_{2i}\frac{(1+N_{2i-1}t))t}{(1-N_{2i}t)(1-N_{2i+2}t)} \hspace{28.5mm} (\text{by } \Cref{AntiSymmRels}) \\
&= -e_{2i}e_{2i-1}\sigma_{2i}\frac{(1+N_{2i-1}t))t}{(1-N_{2i}t)(1-N_{2i+2}t)} \hspace{21mm} (\text{by } \Cref{EnyRels} (3)(iii)) \\
&= -F_{2}.
\end{align*}

$(C_{3} = -D_{3})$: Firstly we have that
\begin{align*}
N_{2i+2}e_{2i}e_{2i+1}\sigma_{2i+1}e_{2i+1}e_{2i} &= N_{2i+2}e_{2i}e_{2i+1}L_{2i}e_{2i} \hspace{16mm} (\text{by } \Cref{EnyRels} (3)(iv)) \\
&= e_{2i}N_{2i+2}e_{2i+1}L_{2i}e_{2i} \hspace{16mm} (\text{by } \Cref{CommutingRels} (1)) \\
&= e_{2i}N_{2i}e_{2i+1}L_{2i}e_{2i} \hspace{19.5mm} (\text{by } 2 \times \Cref{AntiSymmRels}) \\
&= e_{2i}e_{2i+1}L_{2i}N_{2i}e_{2i} \hspace{19.5mm} (\text{by } \Cref{CommutingRels} (1)) \\
&= -e_{2i}e_{2i+1}L_{2i}N_{2i+1}e_{2i} \hspace{13mm} (\text{by } \Cref{AntiSymmRels}) \\
&= e_{2i}e_{2i+1}N_{2i+2}L_{2i}e_{2i} \hspace{16mm} (\text{by } \Cref{AntiSymmRels}) \\
&= e_{2i}e_{2i+1}L_{2i}e_{2i}N_{2i+2} \hspace{16mm} (\text{by } \Cref{CommutingRels} (1)) \\
&= e_{2i}e_{2i+1}\sigma_{2i+1}e_{2i+1}e_{2i}N_{2i+2} \hspace{5mm} (\text{by } \Cref{EnyRels} (3)(iv))
\end{align*}
Hence we have that
\[ \frac{1}{1-N_{2i+2}t}e_{2i}e_{2i+1}\sigma_{2i+1}e_{2i+1}e_{2i}  = e_{2i}e_{2i+1}\sigma_{2i+1}e_{2i+1}e_{2i} \frac{1}{1-N_{2i+2}t}. \]
From this we have
\begin{align*}
C_{3} &= \frac{(1+N_{2i-1}t)(1+N_{2i+1}t)}{(1-N_{2i}t)(1-N_{2i+2}t)}e_{2i}e_{2i+1}\sigma_{2i+1}e_{2i+1}e_{2i}\frac{t}{1-N_{2i}t} \\
&= \frac{(1+N_{2i-1}t)(1+N_{2i+1}t)}{1-N_{2i}t}e_{2i}e_{2i+1}\sigma_{2i+1}e_{2i+1}e_{2i}\frac{t}{(1-N_{2i}t)(1-N_{2i+2}t)} \\
&= -D_{3}.
\end{align*}

$(C_{4} = -D_{4})$: By \emph{\Cref{CommutingRels}} (4), $N_{2i+2}$ commutes with $\sigma_{2i}$. Therefore we have
\begin{align*}
C_{4} &= \frac{(1+N_{2i-1}t)(1+N_{2i+1}t)}{(1-N_{2i}t)(1-N_{2i+2}t)}\sigma_{2i}\frac{t}{1-N_{2i}t} \\
&= \frac{(1+N_{2i-1}t)(1+N_{2i+1}t)}{1-N_{2i}t}\sigma_{2i}\frac{t}{(1-N_{2i}t)(1-N_{2i+2}t)} \\
&= -D_{4}.
\end{align*}

$(D_{1} = -E_{1})$:
\begin{align*}
D_{1} &= \frac{(1+N_{2i-1}t)(1+N_{2i+1}t)}{1-N_{2i}t}\sigma_{2i}e_{2i-1}e_{2i}\frac{t}{(1-N_{2i}t)(1-N_{2i+2}t)} \\
&= \frac{(1+N_{2i-1}t)(1+N_{2i+1}t)}{1-N_{2i}t}L_{2i}e_{2i}\frac{t}{(1-N_{2i}t)(1-N_{2i+2}t)} \hspace{12mm} (\text{by } \Cref{EnyRels} (3)(iii)) \\
&= (1-N_{2i-1}t)L_{2i}\frac{1+N_{2i+1}t}{1-N_{2i}t}e_{2i}\frac{t}{(1-N_{2i}t)(1-N_{2i+2}t)} \hspace{15mm} (\text{Since JM-elements commute}) \\
&= (1-N_{2i-1}t)L_{2i}e_{2i}\frac{t}{(1-N_{2i}t)(1-N_{2i+2}t)} \hspace{31mm} (\text{by } \Cref{AntiSymmRels}) \\
&= (1-N_{2i-1}t)\sigma_{2i}e_{2i-1}e_{2i}\frac{t}{(1-N_{2i}t)(1-N_{2i+2}t)} \hspace{24mm} (\text{by } \Cref{EnyRels} (3)(iii)) \\
&= -E_{1}.
\end{align*}

$(D_{2} = -E_{2})$:
\begin{align*}
D_{2} &= \frac{(1+N_{2i-1}t)(1+N_{2i+1}t)}{1-N_{2i}t}e_{2i}e_{2i-1}\sigma_{2i+1}\frac{t}{(1-N_{2i}t)(1-N_{2i+2}t)} \\
&= (1+N_{2i-1}t)e_{2i}e_{2i-1}\sigma_{2i+1}\frac{t}{(1-N_{2i}t)(1-N_{2i+2}t)} \hspace{22mm} (\text{by } \Cref{AntiSymmRels}) \\
&= -E_{2}.
\end{align*}

$(E_{3} = -F_{3})$: We have that
\begin{align*}
N_{2i-1}e_{2i}e_{2i-1}\sigma_{2i}e_{2i-1}e_{2i} &= N_{2i-1}e_{2i}e_{2i-1}L_{2i}e_{2i} \hspace{16mm} (\text{by } \Cref{EnyRels} (3)(iii)) \\
&= e_{2i}N_{2i-1}e_{2i-1}L_{2i}e_{2i} \hspace{16mm} (\text{by } \Cref{CommutingRels} (1)) \\
&= e_{2i}N_{2i+1}e_{2i-1}L_{2i}e_{2i} \hspace{16mm} (\text{by } 2 \times \Cref{AntiSymmRels}) \\
&= e_{2i}e_{2i-1}L_{2i}N_{2i+1}e_{2i} \hspace{16mm} (\text{by } \Cref{CommutingRels} (1)) \\
&= -e_{2i}e_{2i-1}L_{2i}N_{2i}e_{2i} \hspace{17mm} (\text{by } \Cref{AntiSymmRels}) \\
&= e_{2i}e_{2i-1}N_{2i-1}L_{2i}e_{2i} \hspace{16mm} (\text{by } \Cref{AntiSymmRels}) \\
&= e_{2i}e_{2i-1}L_{2i}e_{2i}N_{2i-1} \hspace{16mm} (\text{by } \Cref{CommutingRels} (1)) \\
&= e_{2i}e_{2i-1}\sigma_{2i}e_{2i-1}e_{2i}N_{2i-1}. \hspace{8mm} (\text{by } \Cref{EnyRels} (3)(iii))
\end{align*}
Hence we have that
\[ (1+N_{2i-1}t)e_{2i}e_{2i-1}\sigma_{2i}e_{2i-1}e_{2i}  = e_{2i}e_{2i-1}\sigma_{2i}e_{2i-1}e_{2i}(1+N_{2i-1}t). \]
From this we have
\begin{align*}
E_{3} &= (1+N_{2i-1}t)e_{2i}e_{2i-1}\sigma_{2i}e_{2i-1}e_{2i}\frac{t}{(1-N_{2i}t)(1-N_{2i+2}t)} \\
&= e_{2i}e_{2i-1}\sigma_{2i}e_{2i-1}e_{2i}\frac{(1+N_{2i-1}t)t}{(1-N_{2i}t)(1-N_{2i+2}t)} \\
&= -F_{3}.
\end{align*}

$(E_{3} = -F_{3})$: We have that $E_{3} = -F_{3}$ since $N_{2i-1}$ commutes with $\sigma_{2i+1}$ by \emph{\Cref{CommutingRels}} (3).

\vspace{2mm}
\noindent
Altogether we have that
\begin{align*}
C_{1} &= -F_{1}, \hspace{3mm} D_{1} = -E_{1}, \\
C_{2} &= -F_{2}, \hspace{3mm} D_{2} = -E_{2}, \\
C_{3} &= -D_{3}, \hspace{3mm} E_{3} = -F_{3}, \\
C_{4} &= -D_{4}, \hspace{3mm} E_{4} = -F_{4}.
\end{align*}
Hence $F+E+D+C=0$ as required.

\end{proof}

\begin{rmk} \label{CoxCommuteRmk}
We only used the relation $\sigma_{2i+1}l_{n}(N_{1},\dots,N_{r}) = l_{n}(N_{1},\dots,N_{r})\sigma_{2i+1}$ when showing that $C_{1}=-F_{1}$. The issue here was the fact that the term $\sigma_{2i+1}e_{2i-1}e_{2i}=s_{i}L_{2i}e_{2i}$ was present, and so moving the generating function from left to right would require us to pass through either $\sigma_{2i+1}$ or $s_{i}$. Since the proposition itself is about $s_{i}$, it proved easier to check first that $\sigma_{2i+1}$ and $l_{n}(N_{1},\dots,N_{r})$ commute.
\end{rmk}

\begin{thm} \label{SSPCentral}
All supersymmetric polynomials in the normalised Jucys-Murphy elements are central in the partition algebra. That is $SS_{r}[N_{1},\dots,N_{r}] \subseteq Z(\mathcal{A}_{r}(\delta))$.
\end{thm}

\begin{proof}
The generators $e_{1},\dots,e_{r-1},s_{1},\dots,s_{\floor{r/2}-1}$ of $\mathcal{A}_{r}(\delta)$ commute with every $l_{n}(N_{1},\dots,N_{r})$ by \emph{\Cref{EiCommute}} and \emph{\Cref{CoxCommute}}, and such polynomials generate $SS_{r}[N_{1},\dots,N_{r}]$ by \emph{\Cref{StemSSPGens}}.

\end{proof}


\section{Semisimple Case}

For this section we will show that in the semisimple case the center of the partition algebra $\mathcal{A}_{2k}(\delta)$ is precisely $SS_{2k}[N_{1},\dots,N_{2k}]$. To do this we will need to utilise some of the representation theory of $\mathcal{A}_{2k}(\delta)$. We begin by setting up some notation and definitions in order to describe the branching graph $\hat{\mathcal{A}}$. This graph is to the partition algebra what Young's lattice is to the symmetric group algebra $\mathbb{C}S_{k}$. In particular, this graph gives us a means to construct a unique basis for any simple $\mathcal{A}_{2k}(\delta)$-module indexed by certain paths in $\hat{\mathcal{A}}$, called the \emph{Gelfand-Zetlin basis}. The action of the JM-elements on such a basis was first described in \cite{HR05}, and is given by the \emph{contents} of paths. Knowing this tells us how the subalgebra $SS_{2k}[N_{1},\dots,N_{2k}]$ acts on any simple module. Using this action, specifically the action of the elementary supersymmetric polynomials, and a result from linear algebra, we will be able to show the existence of a basis for $SS_{2k}[N_{1},\dots,N_{2k}]$, from which a dimension check concludes that it agrees with the center.

\subsection{Representation theory} \
\vspace{2mm}

It was shown in \cite{MS93} (see also \cite{Martin96}) that the partition algebra $\mathcal{A}_{2k}(\delta)$ is semisimple if and only if $\delta \not\in \{0,1,\dots,2k-2\}$ . Until stated otherwise, we will assume that $\mathcal{A}_{2k}(\delta)$ is semisimple. In this situation the chain of algebras
\[ \mathbb{C} = \mathcal{A}_{0}(\delta) \subset \mathcal{A}_{1}(\delta) \subset \mathcal{A}_{2}(\delta) \subset \dots \subset \mathcal{A}_{2k}(\delta), \tag{MFC} \]
is multiplicity free (see \cite[Proposition 7]{Martin00}). That is, let $M$ be a simple $\mathcal{A}_{r}(\delta)$-module for some $1\leq r\leq 2k$. Then the multiplicity of any simple $\mathcal{A}_{r-1}(\delta)$-module as a summand of Res$_{r-1}^{r}(M)$ is at most one, where Res$_{r-1}^{r}$ is the restriction functor from $\mathcal{A}_{r}(\delta)$-mod to $\mathcal{A}_{r-1}(\delta)$-mod. It is this property which allows one to construct a unique (up to scalars) basis for each simple module. We now set up the required definitions to describe the branching graph $\hat{\mathcal{A}}$.

\begin{defn} \label{PartDef}
A \emph{partition} of a positive integer $k$ is a $n$-tuple $\lambda = (\lambda_{1},\lambda_{2},\dots,\lambda_{n}) \in \mathbb{N}^{n}$, for some $n\geq 1$, such that $\lambda_{1} \geq \lambda_{2} \geq \dots \geq \lambda_{n}$ and $\sum_{i=1}^{n}\lambda_{i} = k$. We write $\lambda \vdash k$ and $|\lambda|=k$. Given such a partition we define its associated \emph{Young diagram} as the set
\[ [\lambda] := \{ (i,j) \in \mathbb{N} \times \mathbb{N} \hspace{1mm} | \hspace{1mm} \text{For each } 1\leq i \leq n, \hspace{1mm} \text{$j$ runs over }  1\leq j \leq \lambda_{i} \}. \]
We view this set as a collection of left-justified boxes, with $n$ rows and $\lambda_{i}$ boxes in the $i$-th row. We will often not distinguish between a partition $\lambda$ and its associated Young diagram $[\lambda]$. In particular, we will call an element $a=(i,j)$ of $[\lambda]$ a box, and write $a\in \lambda$. For a box $a=(i,j) \in \lambda$, we say that $i$ is the row index of $a$ and $j$ is the column index of $a$.
\end{defn}

Given a partition $\lambda = (\lambda_{1},\dots,\lambda_{n}) \vdash k$, we say that the \emph{height} of $\lambda$, written $h(\lambda)$, is the number of rows of $\lambda$ minus 1. Similarly we say that the \emph{width} of $\lambda$, written $w(\lambda)$, is the number of columns of $\lambda$ minus 1. Given a box $a = (i,j) \in \lambda$, we say the \emph{content} of $a$ is $c(a)=j-i$. We say that a box $a \in \lambda$ is \emph{removable} if $[\lambda]\backslash\{a\}$ is a Young diagram for a partition of $|\lambda|-1$. Similarly, we say a box $a \in \mathbb{N}\times \mathbb{N}$ is an addable box of $\lambda$ if $[\lambda]\cup\{a\}$ is a Young diagram of a partition of $|\lambda|+1$.

\begin{ex} \label{PartEx}
Consider the partition $\lambda = (7,5,4,3) \vdash 19$. The Young diagram $[\lambda]$, where each box $a \in \lambda$ has its contents inscribed within it, is given by
\[ \young(0123456,\mone0123,\mtwo\mone01,\mthree\mtwo\mone) \]
We have that $h(\lambda) = 3$ and $w(\lambda)=6$. The last (right-most) box in each row of $\lambda$ is a removable box, and these are all the removable boxes of $\lambda$. An example of an addable box is $(2,6)$.
\end{ex}

By convention the empty set $\emptyset$ is a partition of 0, and we set $h(\emptyset)=w(\emptyset)=0$. Now in general two boxes $a,b \in \lambda$ belong to the same diagonal (top-left to bottom-right) in the Young diagram if and only if their contents agree, that is $c(a)=c(b)$. Thus we can index the diagonals of $[\lambda]$ by content. Note that the multi-set of contents of $\lambda$ determines the partition $\lambda$. We will denote the multi-set of contents of $\lambda$ by $MC(\lambda)$. We set $MC(\emptyset)=\emptyset$. We will use superscripts to denote the multiplicity of an element in a multi-set. For example let $\lambda = (7,5,4,3)$ as in \emph{\Cref{PartEx}} above, then
\[ MC(\lambda) = \{-3,-2^{2},-1^{3},0^{3},1^{3},2^{2},3^{2},4,5,6\}.  \]
It is clear to see that given a box $a \in \lambda$, then $-h(\lambda)\leq c(a) \leq w(\lambda)$. Moreover, for any integer value $c$ where $-h(\lambda)\leq c \leq w(\lambda)$, there exists a box $a \in \lambda$ such that $c(a)=c$. Hence the set $\{-h(\lambda),\dots,w(\lambda)\}$ is precisely the indexing set for the diagonals of $[\lambda]$ via content. It will be helpful to express the information of the multi-set $MC(\lambda)$ in the following manner.

\begin{defn} \label{DiaDatumDef}
Let $k$ be a positive integer and $\lambda = (\lambda_{1},\dots,\lambda_{n}) \vdash k$. We define the \emph{diagonal datum} of $\lambda$ to be the pair $(D(\lambda),m_{\lambda})$ where $D(\lambda)=\{-h(\lambda),\dots,w(\lambda)\}$ and $m_{\lambda}:D(\lambda)\rightarrow \mathbb{N}$, given by $m_{\lambda}(x) = |\{ a \in \lambda \hspace{1mm} | \hspace{1mm} c(a) = x \}|$. We set $D(\emptyset)=\emptyset$, and $m_{\emptyset}$ the empty function.
\end{defn}
\noindent
We may now define the branching graph, where we adopt the notation used within \cite{Eny13}. Let $\Lambda_{k}$ denote the set of all partition $\lambda \vdash k$. Then we set
\[ \hat{\mathcal{A}}_{2k} = \hat{\mathcal{A}}_{2k+1} = \{ (\lambda,l) \hspace{2mm} | \hspace{2mm} \lambda \in \Lambda_{k-l} \text{ for some } l = 0,1,\dots,k \}. \]
That is, given $(\lambda,l) \in \hat{\mathcal{A}}_{2k}$, then $\lambda$ is a partition of $k-l$ for some $0\leq l \leq k$. Knowing $k$ and $\lambda$ one could recover $l$ by $l=k-|\lambda|$. With this in mind, when it is clear we are working in $\hat{\mathcal{A}}_{2k}= \hat{\mathcal{A}}_{2k+1}$, we will occasionally suppress the $l$ in $(\lambda,l)$ and just write $\lambda$.

\begin{defn} \label{BranchGrDef}
The branching graph $\hat{\mathcal{A}}$ is the directed graph with
\begin{itemize}
\item[(1)] vertices on level $i$ are given by the set $\hat{\mathcal{A}_{i}}$ for each $i \geq 0$,
\item[(2)] if $i$ is even, then an edge $(\lambda,l) \rightarrow (\mu,m)$ exists for $(\lambda,l) \in \hat{\mathcal{A}}_{i}$ and $(\mu,m) \in \hat{\mathcal{A}}_{i+1}$ if either
\begin{itemize}
\item[(a)] $(\lambda,l)=(\mu,m)$,
\item[(b)] $m=l+1$ and $\mu$ is obtained from $\lambda$ by removing a box,
\end{itemize}
\item[(3)] if $i$ is odd, then an edge $(\lambda,l) \rightarrow (\mu,m)$ exists for $(\lambda,l) \in \hat{\mathcal{A}}_{i}$ and $(\mu,m) \in \hat{\mathcal{A}}_{i+1}$ if either
\begin{itemize}
\item[(a)] $(\lambda,l+1)=(\mu,m)$,
\item[(b)] $m=l$ and $\mu$ is obtained from $\lambda$ by adding a box.
\end{itemize}
\end{itemize}
\end{defn}

The first 7 levels of $\hat{\mathcal{A}}$ are given in Figure 1 below.

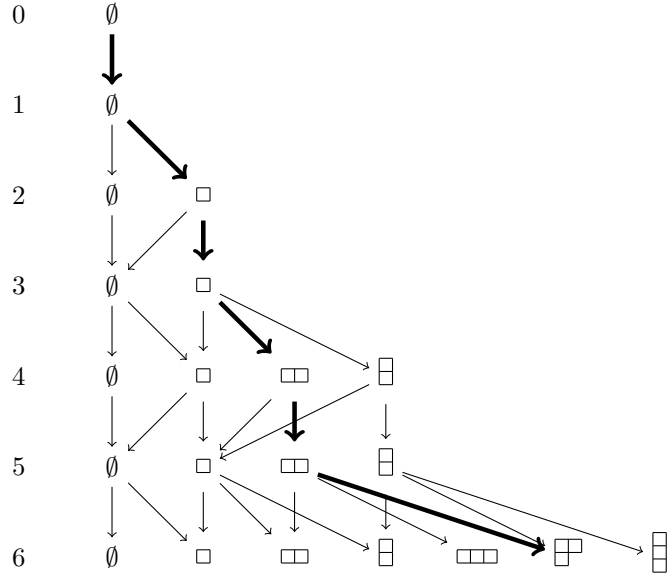
\begin{figure}[h]
\[ \begin{matrix}\begin{tikzpicture}[scale=1.2]
	\node[label=left:{$0 \hspace{8mm}$}] (01) at (0,1){$\emptyset$};
	
    \node[label=left:{$1 \hspace{8mm}$}] (11) at (0,0){$\emptyset$};	
	
	\node[label=left:{$2 \hspace{8mm}$}] (21) at (0,-1){$\emptyset$};
	\node (22) at (1,-1){$\PART{1}$};
	
	\node[label=left:{$3 \hspace{8mm}$}] (31) at (0,-2){$\emptyset$};
	\node (32) at (1,-2){$\PART{1}$};
	
	\node[label=left:{$4 \hspace{8mm}$}] (41) at (0,-3){$\emptyset$};
	\node (42) at (1,-3){$\PART{1}$};
	\node (43) at (2,-3){$\PART{2}$};
	\node (44) at (3,-3){$\PART{1,1}$};
	
	\node[label=left:{$5 \hspace{8mm}$}] (51) at (0,-4){$\emptyset$};
	\node (52) at (1,-4){$\PART{1}$};
	\node (53) at (2,-4){$\PART{2}$};
	\node (54) at (3,-4){$\PART{1,1}$};
	
	\node[label=left:{$6 \hspace{8mm}$}] (61) at (0,-5){$\emptyset$};
	\node (62) at (1,-5){$\PART{1}$};
	\node (63) at (2,-5){$\PART{2}$};
	\node (64) at (3,-5){$\PART{1,1}$};
	\node (65) at (4,-5){$\PART{3}$};
	\node (66) at (5,-5){$\PART{2,1}$};
	\node (67) at (6,-5){$\PART{1,1,1}$};
	
	\draw[->, line width=0.6mm] (01)--(11);
	
	\draw[->] (11)--(21); \draw[->, line width=0.6mm] (11)--(22);
	
	\draw[->] (21)--(31); \draw[->] (22)--(31); \draw[->, line width=0.6mm] (22)--(32);
	
	\draw[->] (31)--(41); \draw[->] (31)--(42); \draw[->] (32)--(42); \draw[->, line width=0.6mm] (32)--(43); \draw[->] (32)--(44);
	
	\draw[->] (41)--(51); \draw[->] (42)--(51); \draw[->] (42)--(52); \draw[->] (43)--(52); \draw[->, line width=0.6mm] (43)--(53); \draw[->] (44)--(52); \draw[->] (44)--(54);
	
	\draw[->] (51)--(61); \draw[->] (51)--(62); \draw[->] (52)--(62); \draw[->] (52)--(63); \draw[->] (52)--(64); \draw[->] (53)--(63); \draw[->] (53)--(65); \draw[->, line width=0.6mm] (53)--(66); \draw[->] (54)--(64); \draw[->] (54)--(66); \draw[->] (54)--(67);
	
	\end{tikzpicture}\end{matrix} \]
	\caption{ The branching graph $\hat{\mathcal{A}}$ truncated at level 6. The path $T^{((2,1),0)}$ has been expressed in bold.}
\end{figure} 

Hence going from an even level to an odd level, one can either ``do nothing'' or remove a box, while going from an odd level to and even level, one can do nothing or add a box. A path in $\hat{\mathcal{A}}$ is defined in the natural way one would for a directed graph. We will be interested in paths which start at level 0 and reach a given vertex $(\lambda,l)$.

\begin{defn} \label{PathDef}
Let $(\lambda,l) \in \hat{\mathcal{A}}_{r}$, then a path $T$ in $\hat{\mathcal{A}}$ from $\emptyset \in \hat{\mathcal{A}}_{0}$ to $(\lambda,l) \in \hat{\mathcal{A}}_{r}$ is a sequence
\[ T = (\lambda^{(0)},l_{0})\rightarrow(\lambda^{(1)},l_{1})\rightarrow \dots \rightarrow (\lambda^{(r-1)},l_{r-1}) \rightarrow (\lambda^{(r)},l_{r}), \]
where $(\lambda^{(i)},l_{i}) \in \hat{\mathcal{A}}_{i}$, $(\lambda^{(i)},l_{i}) \rightarrow (\lambda^{(i)},l_{i+1})$ is an edge in $\hat{\mathcal{A}}$, and $(\lambda^{(r)},l_{r})=(\lambda,l)$. We write Sh$(T)=(\lambda,l)$, and will alternatively write $T=((\lambda^{(i)},l_{i}))_{i=0}^{r}$. Also we will denote by $\mathsf{Path}(\lambda,l)$ the set of paths $T$ for which Sh$(T)=(\lambda,l)$.
\end{defn}

Truncating the graph $\hat{\mathcal{A}}$ at level $2k$ yields the branching graph associated with the multiplicity free chain (MFC) described above. We summarise what this means in the following theorem (see \cite[Theorem 2.24]{HR05} and \cite[Proposition 7]{Martin00}):

\begin{thm} \label{BranchingThm}
Assume that $\mathcal{A}_{2k}(\delta)$ is semisimple. Then for all $r\leq 2k$:
\begin{itemize}
\item[(1)] The vertices $\hat{\mathcal{A}}_{r}$ on the $r$-th level of $\hat{\mathcal{A}}$ give an indexing set for the simple $\mathcal{A}_{r}(\delta)$-modules. We will denote by $\Delta_{r}^{(\lambda,l)}$ the simple module corresponding to $(\lambda,l) \in \hat{\mathcal{A}}_{r}$.
\item[(2)] For each $(\lambda,l) \in \hat{\mathcal{A}}_{r}$, we have
\[ \emph{Res}_{r-1}^{r}(\Delta_{r}^{(\lambda,l)}) = \bigoplus_{(\mu,m) \rightarrow (\lambda,l)}\Delta_{r-1}^{(\mu,m)}, \]
where the sum runs over all edges $(\mu,m) \rightarrow (\lambda,l)$ in $\hat{\mathcal{A}}$.
\item[(3)] For each $(\lambda,l) \in \hat{\mathcal{A}}_{r}$ we have \emph{dim}$(\Delta_{r}^{(\lambda,l)}) = |\mathsf{Path}(\lambda,l)|$.
\end{itemize}
\end{thm}

It will be helpful to pick out a particular path in each $\mathsf{Path}(\lambda,l)$ to work with in later sections. The definition below is precisely the path described in \cite[Lemma 3.9]{Eny13}.

\begin{defn} \label{StPathDef}
Let $(\lambda,l) \in \hat{\mathcal{A}}_{r}$. We define the \emph{standard path} $T^{(\lambda,l)} = ((\lambda^{(i)},l_{i}))_{i=0}^{r} \in \mathsf{Path}(\lambda,l)$ to be the path described as follows:
\begin{itemize}
\item[(i)] $(\lambda^{(2i)},l_{2i}) = (\lambda^{(2i+1)},l_{2i+1})$ for all $0\leq i \leq \ceil{r/2}-1$.
\item[(ii)] $(\lambda^{(2i)},l_{2i}) = (\emptyset,i)$ for all $0 \leq i \leq l$.
\item[(iii)] $(\lambda^{(2i+2)},l_{2i+2})=(\lambda^{(2i)}\cup\{a\},l_{2i})$ for all $l\leq i \leq \floor{r/2}-1$, where $a$ is an addable box of $\lambda^{(2i)}$ with minimal row index.
\end{itemize}
\end{defn} 

Intuitively the path $T^{(\lambda,l)}$ is the one which never removes any boxes, only adds a box when it is forced to, and it priorities adding boxes in the highest row, i.e. lowest row index. In \emph{Figure 1} the path $T^{(\lambda,l)}$ for $(\lambda,l)=((2,1),0) \in \hat{\mathcal{A}}_{6}$ has been given in bold. In general the path $T^{(\lambda,l)}$ is maximal with respect to a particular ordering on the paths $\mathsf{Path}(\lambda,l)$ (see \cite[Definition 3.8]{Eny13}), however we will not employ such a fact here. For us, this path will be useful since it is easy to calculate its \emph{contents}, defined as follows.

\begin{defn} \label{PathContDef}
Let $(\lambda,l) \in \hat{\mathcal{A}_{r}}$ and $T=((\lambda^{(i)},l_{i}))_{i=0}^{r} \in \mathsf{Path}(\lambda,l)$. Then the \emph{contents} of $T$ is the $r$-tuple $(\text{cont}(T,i))_{i=1}^{r}$ defined as follows: For $i$ even we have
\[ \text{cont}(T,i) = \begin{cases}
                \frac{\delta}{2} - |\lambda^{(i)}| & if \lambda^{(i)} = \lambda^{(i-1)} \\
                c(a) - \frac{\delta}{2} & if \lambda^{(i)}=\lambda^{(i-1)}\cup\{a\}
            \end{cases}.  \]
For $i$ odd we have
\[ \text{cont}(T,i) = \begin{cases}
                |\lambda^{(i)}| - \frac{\delta}{2} & if \lambda^{(i)} = \lambda^{(i-1)} \\
                \frac{\delta}{2} - c(a) & if \lambda^{(i)}=\lambda^{(i-1)}\backslash \{a\}
            \end{cases}.  \]
\end{defn}

The next lemma follows from the definition of the standard path. 

\begin{lem} \label{StPathCont}
Let $(\lambda,l) \in \hat{\mathcal{A}}_{2k}$ and $T^{(\lambda,l)}=((\lambda^{(i)},l_{i}))_{i=0}^{2k}$ be the standard path in $\mathsf{Path}(\lambda,l)$. Then as multisets,
\[ \{\emph{cont}(T^{(\lambda,l)},2i+1)\}_{i=0}^{k-1} = \left\{\left(-\frac{\delta}{2}\right)^{l},i-\frac{\delta}{2}\right\}_{i=0}^{k-l-1}, \]
\[ \{\emph{cont}(T^{(\lambda,l)},2i)\}_{i=1}^{k} = \left\{ \left(\frac{\delta}{2}\right)^{l}, c(a) - \frac{\delta}{2} \right\}_{a \in \lambda} = \left\{ \left(\frac{\delta}{2}\right)^{l}, \left(j-\frac{\delta}{2}\right)^{m_{\lambda}(j)}\right\}_{j \in D(\lambda)}, \]
where the superscript denote multiplicity, and $(D(\lambda),m_{\lambda})$ is the diagonal datum of $\lambda$.
\end{lem}

We now define the \emph{Gelfand-Zetlin basis} for any simple $\mathcal{A}_{r}(\delta)$-module $\Delta_{r}^{(\lambda,l)}$ with $(\lambda,l) \in \hat{\mathcal{A}}_{r}$. From \emph{\Cref{BranchingThm}} (2) we have the canonical decomposition
\[ \text{Res}_{r-1}^{r}(\Delta_{r}^{(\lambda,l)}) = \bigoplus_{(\mu,m) \rightarrow (\lambda,l)}\Delta_{r-1}^{(\mu,m)}, \]
where the sum runs over all edges $(\mu,m) \rightarrow (\lambda,l)$ in $\hat{\mathcal{A}}$. If we iterate this process on the above summands, and continue as such down to $\mathcal{A}_{0}(\delta)$, then we obtain a unique decomposition of the simple $\mathcal{A}_{r}(\delta)$-module $\Delta_{r}^{(\lambda,l)}$ into simple $\mathcal{A}_{0}$-modules (i.e. 1-dimensional $\mathbb{C}$-vector spaces) indexed by paths in $\hat{\mathcal{A}}$ from $\emptyset \in \hat{\mathcal{A}}_{0}$ to $(\lambda,l) \in \hat{\mathcal{A}}_{k}$. That is
\[ \text{Res}_{0}^{r}(\Delta_{r}^{(\lambda,l)}) = \bigoplus_{T \in \mathsf{Path}(\lambda,l)}V_{T}, \]
Where Res$_{0}^{r} = \text{Res}_{0}^{1} \circ \text{Res}_{1}^{2} \circ \dots \circ \text{Res}_{r-1}^{r}$. Picking a vector $v_{T} \in V_{T}$ for each $T \in \mathsf{Path}(\lambda,l)$ gives a unique (up to scalars) basis $\{v_{T} \hspace{1mm} | \hspace{1mm} T \in \mathsf{Path}(\lambda,l)\}$ for the simple $\mathcal{A}_{r}(\delta)$-module $\Delta_{r}^{(\lambda,l)}$. We refer to such a basis as the \emph{Gelfand-Zetlin basis}, or GZ-basis for short. Lastly, Halverson and Ram shown that the GZ-basis of a given simple module $\Delta_{r}^{(\lambda,l)}$ diagonalises the action of the Jucys-Murphy elements, and they gave a description of this action.

\begin{thm}[Theorem 3.37 (b) of \cite{HR05}] \label{HRJMAction}
Assume $\mathcal{A}_{r}(\delta)$ is semisimple. Let $\{v_{T} \hspace{1mm} | \hspace{1mm} T \in \mathsf{Path}(\lambda,l)\}$ be a GZ-basis for the simple $\mathcal{A}_{r}(\delta)$-module $\Delta_{r}^{(\lambda,l)}$ indexed by $(\lambda,l) \in \hat{\mathcal{A}}_{r}$. Then
\[ N_{i}v_{T} = \emph{cont}(T,i)v_{T}. \]
\end{thm}

\subsection{The Center} \
\vspace{2mm}

We will now prove that whenever $\mathcal{A}_{2k}(\delta)$ is semisimple we have that $SS_{2k}[N_{1},\dots,N_{2k}]=Z(\mathcal{A}_{2k}(\delta))$. This will be done by employing a result from linear algebra (\emph{\Cref{JKLinIndPolys}} below) from which one can implicitly give a collection of linearly independent elements of $SS_{2k}[N_{1},\dots,N_{2k}]$. The number of such elements will be $|\hat{\mathcal{A}}_{2k}|$, the number of isomorphism classes of simple $\mathcal{A}_{2k}(\delta)$-modules, which in the semisimple case agrees with the dimension of the center, hence showing equivalence. To begin we establish some small results.

\begin{lem} \label{EvalPathEquiv}
Assume $\mathcal{A}_{2k}(\delta)$ is semisimple. Let $(\lambda,l) \in \hat{\mathcal{A}}_{2k}$ and $T,S \in \mathsf{Path}(\lambda,l)$. Then for any supersymmetric polynomial $p \in SS_{2k}[x]$, we have that
\[ p\left(\emph{cont}(T,1),\dots,\emph{cont}(T,2k)\right) = p\left(\emph{cont}(S,1),\dots,\emph{cont}(S,2k)\right). \]
\end{lem}

\begin{proof}
Let $\Delta_{2k}^{(\lambda,l)}$ be the simple $\mathcal{A}_{2k}(\delta)$-module indexed by $(\lambda,l)$, and let $\{v_{T} \hspace{1mm} | \hspace{1mm} T \in \mathsf{Path}(\lambda,l)\}$ be a GZ-basis for $\Delta_{2k}^{(\lambda,l)}$. From \emph{\Cref{HRJMAction}} we have that
\begin{align*}
p(N_{1},\dots,N_{2k})v_{T} &= p(\text{cont}(T,1),\dots,\text{cont}(T,2k))v_{T}, \\
p(N_{1},\dots,N_{2k})v_{S} &= p(\text{cont}(S,1),\dots,\text{cont}(S,2k))v_{S}.
\end{align*}
Now by Schur's lemma we know that
\[ \text{End}_{\mathcal{A}_{2k}}\left(\Delta_{2k}^{(\lambda,l)}\right) \cong \mathbb{C} \text{id}_{(\lambda,l)}, \]
where $\text{id}_{(\lambda,l)}$ is the identity $\mathcal{A}_{2k}(\delta)$-endomorphism of $\Delta_{2k}^{(\lambda,l)}$. Let $z \in Z\left(\mathcal{A}_{2k}(\delta)\right)$, then $zv = \lambda_{z}v$ for some $\lambda_{z} \in \mathbb{C}$ and all $v \in \Delta_{2k}^{(\lambda,l)}$. By \emph{\Cref{SSPCentral}} we know that $p(N_{1},\dots,N_{2k}) \in Z\left(\mathcal{A}_{2k}(\delta)\right)$, and so $p(N_{1},\dots,N_{2k})$ acts on $\Delta_{2k}^{(\lambda,l)}$ by some constant. In particular, it acts on both $v_{T}$ and $v_{S}$ by the same constant, and so
\[ p\left(\text{cont}(T,1),\dots,\text{cont}(T,2k)\right) = p\left(\text{cont}(S,1),\dots,\text{cont}(S,2k)\right). \]
\end{proof}  

The above lemma tells us that to calculate the action of $p(N_{1},\dots,N_{2k})$ on any simple module $\Delta_{2k}^{(\lambda,l)}$, we can evaluate $p$ at the contents of any path $T \in \mathsf{Path}(\lambda,l)$. We choose $T^{(\lambda,l)} \in \mathsf{Path}(\lambda,l)$ given in \emph{\Cref{StPathDef}}, and set
\[ c_{\lambda}(i) := \text{cont}(T^{(\lambda,l)},i). \] 
For any $p \in SS_{2k}[x]$ we set
\[ p(\lambda) := p(c_{\lambda}(1),\dots,c_{\lambda}(2k)). \]
The result we seek to use is the following one. A proof can be found in \cite[Lemma 3.4]{JK17}.

\begin{lem} \label{JKLinIndPolys}
Let $A$ be a $\mathbb{C}$-subalgebra of $\mathbb{C}[x_{1},\dots,x_{n}]$ and let
\[ (c_{11},\dots,c_{1n}),\dots,(c_{m1},\dots,c_{mn}) \]
be $m$ $n$-tuples in $\mathbb{C}^{n}$ for some $m \in \mathbb{Z}_{>0}$. Suppose that for each $1\leq i \neq j \leq m$, there exists a polynomial $p \in A$ such that $p(i) \neq p(j)$, where $p(i) := p(c_{i1},\dots,c_{in})$. Then there exists a family of polynomials $p_{1},\dots,p_{m} \in A$ such that
\[ \begin{vmatrix}
p_{1}(1) & p_{1}(2) & \dots & p_{1}(m) \\
p_{2}(1) & p_{2}(2) & \dots & p_{2}(m) \\
\dots & \dots & \dots & \dots \\
p_{m}(1) & p_{m}(2) & \dots & p_{m}(m) \\
\end{vmatrix} \neq 0. \]
\end{lem} 

We wish to employ the above for $A=SS_{2k}[x]$ and the tuples
\[ (c_{\lambda}(1),\dots,c_{\lambda}(2k)), \]
where $\lambda$ runs over all of $\hat{\mathcal{A}}_{2k}$. That is $m=|\hat{\mathcal{A}}_{2k}|$ and $n=2k$. Thus we need to show that for any $(\lambda,l),(\rho,r) \in \hat{\mathcal{A}}_{2k}$ with $(\lambda,l) \neq (\rho,r)$, there exists a supersymmetric polynomial $p \in SS_{2k}[x]$ such that $p(\lambda)\neq p(\rho)$. We will prove this by showing the contrapositive, namely if $p(\lambda)=p(\rho)$ for all $p \in SS_{2k}[x]$, then we want to show that $(\lambda,l)=(\rho,r)$. Recall that the elementary supersymmetric polynomials $l_{n}$ (given in \emph{\Cref{EleSSPDef}}) generate $SS_{2k}[x]$, so this is equivalent to showing that if 
\[ \frac{\prod_{i=1}^{k}(1+c_{\lambda}(2i-1)t)}{\prod_{i=1}^{k}(1-c_{\lambda}(2i)t)} = \frac{\prod_{i=1}^{k}(1+c_{\rho}(2i-1)t)}{\prod_{i=1}^{k}(1-c_{\rho}(2i)t)}, \]
then $(\lambda,l)=(\rho,r)$. \emph{\Cref{StPathCont}} tells us that for any $(\lambda,l) \in \hat{\mathcal{A}}_{2k}$,
\[ \frac{\prod_{i=1}^{k}(1+c_{\lambda}(2i-1)t)}{\prod_{i=1}^{k}(1-c_{\lambda}(2i)t)} = \frac{\left(1-\frac{\delta}{2}t\right)^{l}\prod_{i=0}^{k-l-1}\left(1+\left(i-\frac{\delta}{2}\right)t\right)}{\left(1-\frac{\delta}{2}t \right)^{l}\prod_{j \in D(\lambda)}\left(1+\left(\frac{\delta}{2}-j\right)t\right)^{m_{\lambda}(j)}} = \frac{\prod_{i=0}^{k-l-1}(1+(i-\frac{\delta}{2})t)}{\prod_{j \in D(\lambda)}(1+(\frac{\delta}{2}-j)t)^{m_{\lambda}(j)}}. \]
Thus collectively, it suffices to show that if \[ \frac{\prod_{i=0}^{k-l-1}(1+(i-\frac{\delta}{2})t)}{\prod_{j \in D(\lambda)}(1+(\frac{\delta}{2}-j)t)^{m_{\lambda}(j)}} = \frac{\prod_{i=0}^{k-r-1}(1+(i-\frac{\delta}{2})t)}{\prod_{j \in D(\rho)}(1+(\frac{\delta}{2}-j)t)^{m_{\rho}(j)}}, \tag{Eq3} \]
then $(\lambda,l)=(\rho,r)$. In order to show this, it will be helpful to understand when the rational polynomials in $t$ involved in (Eq3) are reduced, that is when the numerator and denominator share no common irreducible factors. Also, when they are not reduced, we would like to know what factors cancel. This is all described in the following lemma. Recall the diagonal datum $(D(\lambda),m_{\lambda})$ of a partition $\lambda$ from \emph{\Cref{DiaDatumDef}}, then for any integer $\delta$ we let $D(\lambda)_{\leq \delta} = D(\lambda)\cap \mathbb{Z}_{\leq \delta}$ and $D(\lambda)_{> \delta} = D(\lambda)\cap \mathbb{Z}_{>\delta}$.

\begin{lem} \label{EleSSPRedFrac}
Let $(\lambda,l) \in \hat{\mathcal{A}}_{2k}$, then we have the following two cases:
\begin{itemize}
\item[(1)] Suppose $\delta \not\in \{-h(\lambda),\dots,2k-2\}$. Then
\[ \frac{\prod_{i=0}^{k-l-1}(1+(i-\frac{\delta}{2})t)}{\prod_{j \in D(\lambda)}(1+(\frac{\delta}{2}-j)t)^{m_{\lambda}(j)}}, \]
is reduced.
\item[(2)] Suppose $h(\lambda)\geq 1$ and $\delta \in \{-h(\lambda),\dots,-1\}$. Then
\[ \frac{\prod_{i=0}^{k-l-1}(1+(i-\frac{\delta}{2})t)}{\prod_{j \in D(\lambda)}(1+(\frac{\delta}{2}-j)t)^{m_{\lambda}(j)}} = \frac{\prod_{i=\delta+h(\lambda)+1}^{k-l-1}(1+(i-\frac{\delta}{2})t)}{\prod_{j \in D(\lambda)_{\leq \delta}}(1+(\frac{\delta}{2}-j)t)^{m_{\lambda}(j)-1}\prod_{j \in D(\lambda)_{> \delta}}(1+(\frac{\delta}{2}-j)t)^{m_{\lambda}(j)}}, \]
where the rational polynomial in $t$ on the right is reduced.
\end{itemize}
\end{lem}

\begin{proof}
$(1)$: We seek to show that the polynomials $\prod_{i=0}^{k-l-1}(1+(i-\frac{\delta}{2})t)$ and $\prod_{j \in D(\lambda)}(1+(\frac{\delta}{2}-j)t)^{m_{\lambda}(j)}$ share no common factors. Assume, for a contradiction, that this is not the case. Then $i-\delta/2 = \delta/2 -j$, and so $\delta = i+j$, for some $0\leq i \leq k-l-1$ and $j \in D(\lambda) = \{-h(\lambda),\dots,w(\lambda)\}$. Thus immediately we see that if $\delta \not\in \mathbb{Z}$ then the fraction is reduced. Furthermore we have that
\[ -h(\lambda) \leq \delta \leq w(\lambda)+k-l-1 \leq 2(k-l-1) \leq 2k-2, \]
which contradicts the assumption that $\delta \not\in \{-h(\lambda),\dots,2k-2\}$.

$(2)$: We now seek to understand what factors of 
\[ \frac{\prod_{i=0}^{k-l-1}(1+(i-\frac{\delta}{2})t)}{\prod_{j \in D(\lambda)}(1+(\frac{\delta}{2}-j)t)^{m_{\lambda}(j)}}, \]
cancel out when $h(\lambda)\geq 1$ and $\delta \in \{-h(\lambda),\dots,-1\}$. As above the numerator and denominator share a common factor if $\delta = i+j$ for some $0\leq i \leq k-l-1$ and $j \in D(\lambda)$. Let
\[ P(\delta) = \{(i,j) \hspace{1mm} | \hspace{1mm} \delta = i+j, \hspace{1mm} 0\leq i \leq k-l-1, \hspace{1mm} j \in D(\lambda)\}. \]
So $|P(\delta)|$ is the number of common factors between $\prod_{i=0}^{k-l-1}(1+(i-\frac{\delta}{2})t)$ and $\prod_{j \in D(\lambda)}(1+(\frac{\delta}{2}-j)t)^{m_{\lambda}(j)}$. One can see that 
\[ P(\delta) = \{(0,\delta),(1,\delta-1),\dots,(h(\lambda)+\delta,-h(\lambda))\}, \] 
in particular $|P(\delta)| = h(\lambda)+\delta+1$. Therefore the factors of $\prod_{i=0}^{k-l-1}(1+(i-\frac{\delta}{2})t)$ corresponding to $i=0,1,\dots,h(\lambda)+\delta$ each cancel with one of the factors of $\prod_{j \in D(\lambda)}(1+(\frac{\delta}{2}-j)t)^{m_{\lambda}(j)}$ corresponding to $j=\delta,\delta-1,\dots,-h(\lambda)$ respectively. Hence we obtain
\[ \frac{\prod_{i=0}^{k-l-1}(1+(i-\frac{\delta}{2})t)}{\prod_{j \in D(\lambda)}(1+(\frac{\delta}{2}-j)t)^{m_{\lambda}(j)}} = \frac{\prod_{i=\delta+h(\lambda)+1}^{k-l-1}(1+(i-\frac{\delta}{2})t)}{\prod_{j \in D(\lambda)_{\leq \delta}}(1+(\frac{\delta}{2}-j)t)^{m_{\lambda}(j)-1}\prod_{j \in D(\lambda)_{> \delta}}(1+(\frac{\delta}{2}-j)t)^{m_{\lambda}(j)}}, \]
where the rational polynomial in $t$ on the right is reduced.

\end{proof}

Note that the above lemma has given a description of the generating function of $l_{n}(\lambda)$ in reduced form for all $(\lambda,l) \in \hat{\mathcal{A}}_{2k}$ whenever $\mathcal{A}_{2k}(\delta)$ is semisimple, that is whenever $\delta \not\in \{0,1,\dots,2k-2\}$. Knowing this, we can prove the following:

\begin{prop} \label{SSPDistSimples}
Let $\mathcal{A}_{2k}(\delta)$ be semisimple. Let $(\lambda,l),(\rho,r) \in \hat{\mathcal{A}}_{2k}$ such that $(\lambda,l) \neq (\rho,r)$. Then there exists a polynomial $p \in SS_{2k}[x]$ such that $p(\lambda)\neq p(\rho)$.
\end{prop}

\begin{proof}
We will prove this by showing the contrapositive, that is if $p(\lambda)=p(\rho)$ for all $p \in SS_{2k}[x]$, then $(\lambda,l)=(\rho,r)$. As mentioned above, having $p(\lambda)=p(\rho)$ for all $p \in SS_{2k}[x]$ is equivalent to saying that
\[ \frac{\prod_{i=0}^{k-l-1}(1+(i-\frac{\delta}{2})t)}{\prod_{j \in D(\lambda)}(1+(\frac{\delta}{2}-j)t)^{m_{\lambda}(j)}} = \frac{\prod_{i=0}^{k-r-1}(1+(i-\frac{\delta}{2})t)}{\prod_{j \in D(\rho)}(1+(\frac{\delta}{2}-j)t)^{m_{\rho}(j)}}. \]
Using \emph{\Cref{EleSSPRedFrac}} we will break the above equality into four cases, and for each we will show that either $(\lambda,l)=(\rho,r)$, or the case is impossible. The four cases to consider are the following:
\vspace{2mm}

\begin{itemize}
\item[(Case 1)] $\delta \not\in \{-h(\lambda),\dots,2k-2\}\cup\{-h(\rho),\dots,2k-2\}$.
\item[(Case 2)] $\delta \not\in \{-h(\rho),\dots,2k-2\}$ but $\delta \in \{-h(\lambda),\dots,-1\}$ with $h(\lambda)\geq 1$.
\item[(Case 3)] $\delta \not\in \{-h(\lambda),\dots,2k-2\}$ but $\delta \in \{-h(\rho),\dots,-1\}$ with $h(\rho)\geq 1$.
\item[(Case 4)] $\delta \in \{-h(\rho),\dots,-1\}\cap\{-h(\lambda),\dots,-1\}$.
\end{itemize}
\vspace{2mm}

(Case 1): Since $\delta \not\in \{-h(\lambda),\dots,2k-2\}\cup\{-h(\rho),\dots,2k-2\}$, \emph{\Cref{EleSSPRedFrac}} (1) implies that 
\[ \frac{\prod_{i=0}^{k-l-1}(1+(i-\frac{\delta}{2})t)}{\prod_{j \in D(\lambda)}(1+(\frac{\delta}{2}-j)t)^{m_{\lambda}(j)}} = \frac{\prod_{i=0}^{k-r-1}(1+(i-\frac{\delta}{2})t)}{\prod_{j \in D(\rho)}(1+(\frac{\delta}{2}-j)t)^{m_{\rho}(j)}}. \]
where both sides are reduced. Since they are reduced, we may equate the numerators and denominators. Equating the numerators gives
\[ \prod_{i=0}^{k-l-1}\left(1+\left(i-\frac{\delta}{2}\right)t\right) = \prod_{i=0}^{k-r-1}\left(1+\left(i-\frac{\delta}{2}\right)t\right). \tag{Eq4} \]
Assume one of the factors on the left hand side is trivial, that is $i=\delta/2$ for some $0\leq i \leq k-l-1$. This would imply that $0 \leq \delta \leq 2(k-l-1)$, which contradicts the assumption $\delta \not\in \{-h(\lambda),\dots,2k-2\}$. As such no factor on the left hand side of (Eq4) is trivial, similarly no factor on the right is trivial. Therefore (Eq4) implies that $l=r$ and hence $|\lambda|=|\rho|$. Now equating the denominators gives
\[ \prod_{j \in D(\lambda)}\left(1+\left(\frac{\delta}{2}-j\right)t\right)^{m_{\lambda}(j)} = \prod_{j \in D(\rho)}\left(1+\left(\frac{\delta}{2}-j\right)t\right)^{m_{\rho}(j)}.
\]
This implies $D(\lambda)\backslash\{\delta/2\} = D(\rho)\backslash\{\delta/2\}$ and that $m_{\lambda}(j)=m_{\rho}(j)$ for all $ j \in D(\lambda)\backslash\{\delta/2\}$. This means that the Young diagrams $[\lambda]$ and $[\rho]$ can only differ in the diagonal indexed by $\delta/2$. However since $|\lambda|=|\rho|$, no such difference is present, hence $(\lambda,l)=(\rho,r)$.
\vspace{2mm}

(Case 2): Since $\delta \not\in \{-h(\rho),\dots,2k-2\}$ but $\delta \in \{-h(\lambda),\dots,-1\}$ with $h(\lambda)\geq 1$, \emph{\Cref{EleSSPRedFrac}} (1) and (2) tell us that
\[ \frac{\prod_{i=\delta+h(\lambda)+1}^{k-l-1}(1+(i-\frac{\delta}{2})t)}{\prod_{j \in D(\lambda)_{\leq \delta}}(1+(\frac{\delta}{2}-j)t)^{m_{\lambda}(j)-1}\prod_{j \in D(\lambda)_{> \delta}}(1+(\frac{\delta}{2}-j)t)^{m_{\lambda}(j)}} = \frac{\prod_{i=0}^{k-r-1}(1+(i-\frac{\delta}{2})t)}{\prod_{j \in D(\rho)}(1+(\frac{\delta}{2}-j)t)^{m_{\rho}(j)}}. \]
As these are reduced we may equate the numerators and denominators. Equating the numerators gives
\[ \prod_{i=\delta+h(\lambda)+1}^{k-l-1}\left(1+\left(i-\frac{\delta}{2}\right)t\right) = \prod_{i=0}^{k-r-1}\left(1+\left(i-\frac{\delta}{2}\right)t\right). \tag{Eq5} \]
From the previous case we know that the right hand side of (Eq5) has no trivial factors. Assume the left hand side has a trivial factor, that is $i=\delta/2$ for some $\delta+h(\lambda)+1\leq i \leq k-l-1$. This implies that $2(\delta+h(\lambda)+1) \leq \delta$, which gives the inequality $\delta \leq -2h(\lambda)-2$. However this contradicts the assumption $\delta \in \{-h(\lambda),\dots,-1\}$. Hence none of the factors on the left hand side of (Eq5) are trivial. Therefore (Eq5) implies that $\delta+h(\lambda)+1=0$, and so $\delta=-h(\lambda)-1$, but this contradicts the assumption $\delta \in \{-h(\lambda),\dots,-1\}$. Thus this equality can never hold, i.e. this case is impossible. By symmetry, the same can be said for (Case 3).
\vspace{2mm}

(Case 4): Since $\delta \in \{-h(\rho),\dots,-1\}\cap\{-h(\lambda),\dots,-1\}$ with $h(\lambda),h(\rho)\geq 1$, \emph{\Cref{EleSSPRedFrac}} (2) implies
\[ \frac{\prod_{i=\delta+h(\lambda)+1}^{k-l-1}(1+(i-\frac{\delta}{2})t)}{\prod_{j \in D(\lambda)_{\leq \delta}}(1+(\frac{\delta}{2}-j)t)^{m_{\lambda}(j)-1}\prod_{j \in D(\lambda)_{> \delta}}(1+(\frac{\delta}{2}-j)t)^{m_{\lambda}(j)}} \hspace{50mm} \]
\[ \hspace{50mm} = \frac{\prod_{i=\delta+h(\rho)+1}^{k-r-1}(1+(i-\frac{\delta}{2})t)}{\prod_{j \in D(\rho)_{\leq \delta}}(1+(\frac{\delta}{2}-j)t)^{m_{\rho}(j)-1}\prod_{j \in D(\rho)_{> \delta}}(1+(\frac{\delta}{2}-j)t)^{m_{\rho}(j)}}. \]
Since both sides are reduced, we may equate the numerators and denominators. Equating numerators gives
\[ \prod_{i=\delta+h(\lambda)+1}^{k-l-1}\left(1+\left(i-\frac{\delta}{2}\right)t\right) = \prod_{i=\delta+h(\rho)+1}^{k-r-1}\left(1+\left(i-\frac{\delta}{2}\right)t\right). \] 
Arguing as in case (2), none of the factors in the above equality are trivial. As such we must have that both $\delta+h(\lambda)+1 = \delta+h(\rho)+1$ and $k-l-1 = k-r-1$, hence $h(\lambda)=h(\rho)$, $l=r$, and $|\lambda|=|\rho|$. By assumption $-h(\lambda)=-h(\rho) \leq \delta \leq -1$, hence we have that $D(\lambda)_{\leq \delta} = D(\rho)_{\leq \delta}$. Now equating the denominators gives
\[ \prod_{j \in D(\lambda)_{\leq \delta}}\left(1+\left(\frac{\delta}{2}-j\right)t\right)^{m_{\lambda}(j)-1}\prod_{j \in D(\lambda)_{> \delta}}\left(1+\left(\frac{\delta}{2}-j\right)t\right)^{m_{\lambda}(j)} \hspace{40mm} \]
\[ \hspace{40mm} = \prod_{j \in D(\rho)_{\leq \delta}}\left(1+\left(\frac{\delta}{2}-j\right)t\right)^{m_{\rho}(j)-1}\prod_{j \in D(\rho)_{> \delta}}\left(1+\left(\frac{\delta}{2}-j\right)t\right)^{m_{\rho}(j)}. \]
Since $D(\mu_{1})_{\leq \delta}\cap D(\mu_{2})_{> \delta} = \emptyset$ for any $\mu_{1},\mu_{2} \in \{\lambda,\rho\}$, we must have 
\[ \prod_{j \in D(\lambda)_{\leq \delta}}\left(1+\left(\frac{\delta}{2}-j\right)t\right)^{m_{\lambda}(j)-1} = \prod_{j \in D(\rho)_{\leq \delta}}\left(1+\left(\frac{\delta}{2}-j\right)t\right)^{m_{\rho}(j)-1}, \tag{Eq6} \]
and
\[ \prod_{j \in D(\lambda)_{> \delta}}\left(1+\left(\frac{\delta}{2}-j\right)t\right)^{m_{\lambda}(j)} = \prod_{j \in D(\rho)_{> \delta}}\left(1+\left(\frac{\delta}{2}-j\right)t\right)^{m_{\rho}(j)}. \tag{Eq7} \]
Since $\delta/2 \not\in D(\lambda)_{\leq \delta} = D(\rho)_{\leq \delta}$ by definition, there must be no trivial factors in (Eq6). As such the multiplicities in (Eq6) must agree, that is $m_{\lambda}(j)=m_{\rho}(j)$ for all $j \in D(\lambda)_{\leq \delta}$. Now (Eq7) tells us that $D(\lambda)_{> \delta}\backslash\{\delta/2\} = D(\rho)_{> \delta}\backslash\{\delta/2\}$ and that the multiplicity functions $m_{\lambda}$ and $m_{\rho}$ agree on this set. Hence together, (Eq6) and (Eq7) tell us that $D(\lambda)\backslash\{\delta/2\} = D(\rho)\backslash\{\delta/2\}$ and their multiplicity functions $m_{\lambda}$ and $m_{\rho}$ agree on this set. As was the situation in case (1), this implies that the Young diagrams $[\lambda]$ and $[\rho]$ can only differ in the diagonal indexed by $\delta/2$, but since $|\lambda|=|\rho|$, no difference is present, showing that $(\lambda,l)=(\rho,r)$.

\end{proof}

\begin{rmk} \label{SpecCaseBlockResult}
It is worth mentioning that the above proposition follows from the semisimple case of \emph{\Cref{AltBlockCri}}, proven in the next section, and a description of the blocks of the partition algebra given by Martin (see \emph{\Cref{MartinBlocks}}). We have included the proof above for completeness, and to show that the result on the semisimple center (\emph{\Cref{SSP=Center}} below) can be proven without any knowledge of the block theory of $\mathcal{A}_{2k}(\delta)$, but just from knowing that $l_{n}(N_{1},\dots,N_{2k})$ is central.
\end{rmk}

\begin{thm} \label{SSP=Center}
Let $\mathcal{A}_{2k}(\delta)$ be semisimple. Then the supersymmetric polynomials in $N_{1},\dots,N_{2k}$ generate the center of $\mathcal{A}_{2k}(\delta)$. That is $SS_{2k}[N_{1},\dots,N_{2k}] = Z(\mathcal{A}_{2k}(\delta))$.
\end{thm}

\begin{proof}
By \emph{\Cref{SSPDistSimples}}, we can apply \emph{\Cref{JKLinIndPolys}} to the case $A = SS_{2k}[x]$ and where
\[ \{ (c_{11},\dots,c_{1n}),\dots,(c_{m1},\dots,c_{mn}) \} = \left\{ \left(c_{\lambda}(1),\dots,c_{\lambda}(2k)\right) \hspace{1mm} : \hspace{1mm} (\lambda,l) \in \hat{\mathcal{A}}_{2k} \right\}. \]
So we have that $n = 2k$ and $m = |\hat{\mathcal{A}}_{2k}|$. Hence \emph{\Cref{JKLinIndPolys}} tells us that there exists a family of supersymmetric polynomials $\{ p_{\lambda} : (\lambda,l) \in \hat{\mathcal{A}}_{2k}\} \subset SS_{2k}[x]$ such that the matrix $(p_{\lambda}(\mu))_{(\lambda,l),(\mu,m) \in \hat{\mathcal{A}}_{2k}}$ in $\mathbb{C}^{m\times m}$ is invertible, recalling that $p_{\lambda}(\mu) := p_{\lambda}(c_{\mu}(1),\dots,c_{\mu}(2k))$. We will now show that the corresponding polynomials $p_{\lambda}(N_{1},\dots,N_{2k})$ in $SS_{2k}[N_{1},\dots,N_{2}]$ are also linearly independent. Assume that
\[ P = \sum_{(\lambda,l) \in \hat{\mathcal{A}}_{2k}}c_{\lambda}p_{\lambda}(N_{1},\dots,N_{2k}) = 0, \]
for some $c_{\lambda} \in \mathbb{C}$. Hence $P$ acts on any simple $\mathcal{A}_{2k}(\delta)$ module $\Delta_{2k}^{(\mu,m)}$ by 0. By \emph{\Cref{HRJMAction}} and \emph{\Cref{EvalPathEquiv}} this means
\[ \sum_{(\lambda,l) \in \hat{\mathcal{A}}_{2k}}c_{\lambda}p_{\lambda}(\mu) = 0, \] 
for any $(\mu,m) \in \hat{\mathcal{A}}_{2k}$. However, since the column vectors of $((p_{\lambda}(\mu))_{(\lambda,l),(\mu,m) \in \hat{\mathcal{A}}_{2k}}$ are linearly independent, we must have that $c_{\lambda} = 0$ for all $(\lambda,l) \in \hat{\mathcal{A}}_{2k}$. Therefore the set 
\[ \left\{ p_{\lambda}(N_{1},\dots,N_{2k}) : (\lambda,l) \in \hat{\mathcal{A}}_{2k} \right\} \]
is linearly independent in $SS_{2k}[N_{1},\dots,N_{2k}]$. Since $\mathcal{A}_{2k}(\delta)$ is semisimple, we know that the dimension of the center $Z(\mathcal{A}_{2k}(\delta))$ equals $|\hat{\mathcal{A}}_{2k}|$, which equals the size of the above linearly independent set. Hence this set is a basis, which shows that $SS_{2k}[N_{1},\dots,N_{2k}] = Z(\mathcal{A}_{2k}(\delta))$.
 
\end{proof}


\section{Non-semisimple Case}

In this section we will recall some of the block theory of $\mathcal{A}_{2k}(\delta)$ for arbitrary $\delta \in \mathbb{C}$ which was developed by P. Martin in \cite{Martin96}, and later by D. Wales and W. Doran in \cite{DW00}. We will conclude by giving an equivalent condition for when two partitions belong to the same block, in terms of the generating function of $l_{n}(\lambda)$, which played a pivotal role in the previous section. 

Let $A$ be any finite dimensional $\mathbb{C}$-algebra, and let $\Lambda$ be an indexing set for the isomorphism classes of simple $A$-modules. The algebra $A$ has a unique decomposition as a direct sum of indecomposable 2-sided ideals 
\[ A=e_{1}A\oplus e_{2}A\oplus \dots \oplus e_{n}A, \]
where $1=e_{1}+e_{2}+\dots+e_{n}$ is a decomposition of unity as a sum of primitive central idempotents $e_{i}\in A$. The direct summands in the above decomposition are called the \emph{blocks} of $A$. We say that an $A$-module $M$ belongs to the block $e_{i}A$ if $e_{i}M=M$ and $e_{j}M=0$ for all $j \neq i$. Any simple module of $A$ belongs to a particular block. Also one can deduce that $M$ belongs to the block $e_{i}A$ if and only if all the composition factors of $M$ belong to $e_{i}A$. We can equip the indexing set $\Lambda$ with the equivalence relation $\lambda \sim \mu$ if and only if the simple modules indexed by $\lambda$ and $\mu$ belong to the same block. Let $\mathcal{B}_{A}(\lambda)$ be the equivalence class of $\lambda$ in $\Lambda$ with respect to this equivalence relation. We will refer to $\mathcal{B}_{A}(\lambda)$ as a \emph{block} of $\Lambda$. Whenever $A$ is semisimple, we have that $\mathcal{B}_{A}(\lambda)=\{\lambda\}$ for all $\lambda \in \Lambda$. 

For any $\lambda \in \Lambda$ let $A^{\lambda}$ denote the simple $A$-module corresponding to $\lambda$. Let $z$ belong to the center $Z(A)$ of $A$. Then by Schur's lemma the element $z$ acts by a scalar on $A^{\lambda}$. Let $\chi_{\lambda}(z) \in \mathbb{C}$ denote this scalar. Then we obtain a $\mathbb{C}$-algebra homomorphism
\[ \chi_{\lambda}:Z(A) \rightarrow \mathbb{C}, \]
which we call the \emph{central character induced by} $\lambda$. It is well known that $\lambda$ and $\mu$ belong to the same block of $\Lambda$ if and only if the central characters $\chi_{\lambda}$ and $\chi_{\mu}$ equal one another. In this sense the center $Z(A)$ can distinguish between the blocks of $A$. 

We return now to the partition algebra $\mathcal{A}_{2k}(\delta)$. Whenever $\delta\neq 0$, the indexing set for the simple $\mathcal{A}_{2k}(\delta)$-modules is $\hat{\mathcal{A}}_{2k}$, that is the set of partitions $\lambda \vdash k-l$ where $0\leq l \leq k$. When $\delta=0$, the indexing set is $\hat{\mathcal{A}}_{2k}\backslash\{\emptyset\}$ (see \cite[Corollary 2.3]{DW00}). We set
\[ \Lambda_{k}(\delta) := \begin{cases}
                  \hat{\mathcal{A}}_{2k}, & \delta \neq 0 \\
                  \hat{\mathcal{A}}_{2k}\backslash\{\emptyset\}, & \delta = 0
                  \end{cases} \]
For any $\lambda \in \Lambda_{k}(\delta)$ set $\mathcal{B}_{k}(\lambda):=\mathcal{B}_{\mathcal{A}_{2k}(\delta)}(\lambda)$. The blocks of $\Lambda_{k}(\delta)$ were first described by P. Martin in \cite[Proposition 9]{Martin96} for $\delta\neq 0$, where he gave an elegant combinatorial condition for when two partitions belong to the same block. These results were later extended to the case $\delta=0$ by D. Wales and W. Doran in \cite{DW00}. We recall here the combinatorial description for the blocks.

Let $\lambda,\mu$ be partitions. We will write $\mu \subset \lambda$ to mean that as sets we have the inclusion of Young diagrams $[\mu] \subset [\lambda]$. Similarly we write $\lambda\cap \mu$ for the partition given by the Young diagram $[\mu]\cap[\lambda]$. Also, whenever $\mu \subset \lambda$, we call the collection of boxes one obtains from removing the boxes of $\mu$ from $\lambda$, i.e. the set $[\lambda]\backslash [\mu]$, a \emph{skew diagram}, and we denote it by $\lambda/\mu$ and set $[\lambda/\mu] := [\lambda]\backslash [\mu]$.

\begin{ex} \label{SkewDiaEx}
Let $\mu_{1} = (2,2,1,1) \vdash 6$, $\mu_{2} = (2,2,2) \vdash 6$, and $\lambda = (4,3,3,2) \vdash 12$ be partitions with Young diagrams
\[ [\mu_{1}] = \Yvcentermath1 \yng(2,2,1,1) \hspace{8mm} [\mu_{2}] = \Yvcentermath1 \yng(2,2,2) \hspace{8mm} [\lambda] = \Yvcentermath1 \yng(4,3,3,2) \]
We have $\mu_{1},\mu_{2} \subset \lambda$, and the skew diagrams, with the content of each box inscribed within it, are
\[ \lambda/\mu_{1} = \Yvcentermath1 \young(::23,::1,:\mone0,:\mtwo) \hspace{8mm} \lambda/\mu_{2} = \Yvcentermath1 \young(::23,::1,::0,\mthree\mtwo) \]
\end{ex}

As with partitions, we will denote by $MC(\lambda/\mu)$ the multi-set of contents for the skew diagram $\lambda/\mu$. However, given a skew diagram $\lambda/\mu$, there are infinitely many pairs $\mu' \subset \lambda'$ of partitions such that $MC(\lambda/\mu)=MC(\lambda'/\mu')$. Hence the multi-set of contents of a skew diagram does not in general determine the skew diagram. However it is clear that if we also know one of the partitions $\lambda$ or $\mu$, then one can determine the set $[\lambda/\mu]$. 

\begin{defn} \label{SkewHookDef}
Let $\mu$ and $\lambda$ be partitions with $\mu \subset \lambda$. The skew diagram $\lambda/\mu$ is called a \emph{skew hook} if $MC(\lambda/\mu) = \{x,x+1,\dots,y\}$ for some $x,y \in \mathbb{Z}$ with $x\leq y$. 
\end{defn}

In \emph{\Cref{SkewDiaEx}}, the skew diagram $\lambda/\mu_{1}$ is a skew hook, while $\lambda/\mu_{2}$ is not a skew hook since there is no box with content $-1$ present. If $\lambda/\mu$ is a skew hook, then as a diagram it is one connected piece with one box in each of its diagonals. Suppose every box in a skew diagram $\lambda/\mu$ has the same row index, i.e. all of its boxes lie in the same row, then we call such a skew diagram a \emph{horizontal strip}. Noticeably any horizontal strip is a skew hook.

\begin{defn} \label{DeltaPairDef}
Let $\lambda$ and $\mu$ be elements of $\Lambda_{k}(\delta)$ with $\mu \subset \lambda$. We say that the ordered pair $(\mu,\lambda)$ is a $\delta$-pair, written $\mu \hookrightarrow_{\delta} \lambda$, if $\lambda/\mu$ is a horizontal strip with the last (right-most) box having content $\delta - |\mu|$. We say a chain of partitions
\[ \tau^{(x)} \subset \tau^{(x+1)} \subset \dots \subset \tau^{(y)} \]
is a $\delta$-chain if, for each $i$, $(\tau^{(i)},\tau^{(i+1)})$ is a $\delta$-pair, differing in the $(i+1)$-th row, for all $x \leq i \leq y$. 
\end{defn}

\begin{ex} \label{DeltaPairEx}
Let $\delta = 1$, then the chain
\[ \emptyset \subset \Yvcentermath1 \young(01) \subset \Yvcentermath1 \young(01,\mone) \]
of elements in $\Lambda_{3}(1)$ is a $1$-chain. We see that $\emptyset$ and $\PART{2}$ differ in the first row, and the content of the last box in this row is $1 = \delta - |\emptyset|$. Also $\PART{2}$ and $\PART{2,1}$ differ in the second row, and the content of the last box is $-1 = \delta - |\PART{2}|$. 
\end{ex}

Note if $\delta \in \mathbb{C}$ is not an integer, then no $\delta$-pairs exist. Let $\tau$ be a partition with $n\geq 1$ rows. Let $1 \leq i \neq j \leq n+1$, and let $R_{i}$ and $R_{j}$ denote two horizontal strips which could be added to $\tau$ in the $i$-th and $j$-th row respectively to give new partitions. Set $c(R_{i})$ to be the set of contents of the horizontal strip $R_{i}$, and similarly for $R_{j}$. Then one can observe that $c(R_{i})\cap c(R_{j})=\emptyset$. This tells us that if there exists a partition $\lambda$ such that $\tau \subset \lambda$ and $(\tau,\lambda)$ is a $\delta$-pair, then $\lambda$ is the unique partition to do so. Similarly, if $\mu$ exists such that $\mu \subset \tau$ and $(\mu,\tau)$ is a $\delta$-pair, then $\mu$ is unique.    

\begin{lem} \label{AllChainsDeltaChains}
Let $\tau$ be a partition and $\delta \in \mathbb{Z}$. Suppose there exists a partition $\lambda$ such that $\tau \subset \lambda$, $(\tau,\lambda)$ is a $\delta$-pair, and $\lambda$ differs to $\tau$ in the row indexed by $i$ for $i\neq 1$. Then there exists a partition $\mu \subset \tau$ such that $(\mu,\tau)$ is a $\delta$-pair, and $\tau$ differs to $\mu$ in the row indexed by $i-1$. Furthermore $\lambda/\mu$ is a skew hook.
\end{lem}

\begin{proof}
Let $\lambda/\tau = R_{i}$ be the horizontal strip in the $i$-th row of $\lambda$, and let the right-most box of $R_{i}$ be $(i,j)$. Since $(\tau,\lambda)$ is a $\delta$-pair, we have that $j-i=\delta-|\tau|$. Consider the box just above $(i,j)$ in $\lambda$, that is the box $(i-1,j)$. Note that $(i-1,j)$ also belongs to $\tau$. Let $(i-1,k) \in \tau$ be the last (right-most) box in the $(i-1)$-th row of $\tau$. Then consider the horizontal strip of boxes $R_{i-1}=\{(i-1,j),(i-1,j+1),\dots,(i-1,k)\}$ with size $|R_{i-1}|=k-j+1$. Let $\mu$ be the partition such that $\mu\cup R_{i-1} = \tau$, in particular $\mu \subset \tau$, $\tau/\mu = R_{i-1}$ is a horizontal strip, and $R_{i-1}$ has been chosen such that $\lambda/\mu$ is a skew hook. What remains now is to show that the last box of $R_{i-1}$, that is the box $(i-1,k)$, has content equal to $\delta -|\mu|$. Well,
\begin{align*}
\delta - |\mu| &= (j-i+|\tau|) - (|\tau|-|R_{i-1}|) \\
&= j-i +|R_{i-1}| \\
&= j-i + (k-j+1) \\
&= k-i+1 \\
&= c((i-1,k)).
\end{align*}

\end{proof}

The above lemma, along with the fact that $\delta$-pairings are unique, implies that if we have a chain of partitions $\tau^{(x)} \subset \tau^{(x+1)} \subset \dots \subset \tau^{(y)}$ such that $(\tau^{(i)},\tau^{(i+1)})$ is a $\delta$-pair for each $i$, then it must be the case that this chain is a $\delta$-chain, i.e. the horizontal strips in which consecutive pairs differ by occur in consecutive rows. Moreover, given such a $\delta$-chain, the skew diagram $\tau^{(y)}/\tau^{(x)}$ is in fact a skew hook. We say that a $\delta$-chain $\tau^{(x)} \subset \tau^{(x+1)} \subset \dots \subset \tau^{(y)}$ is maximal if no partition $\mu$ can be added to form a new $\delta$-chain, that is the chain is maximal in size. The above lemma tells us that if $\tau^{(x)} \subset \tau^{(x+1)} \subset \dots \subset \tau^{(y)}$ is maximal, then $x=0$.

Due to the uniqueness of $\delta$-pairings, the maximal $\delta$-chains partition the indexing set $\Lambda_{k}(\delta)$ (also see \cite[Proposition 8]{Martin96}). Let $\sim_{\mathcal{C}_{k}}$ denote the corresponding equivalence relation, and let $\mathcal{C}_{k}(\lambda)$ denote the equivalence class of $\lambda$ with respect to this relation. For the following result see \cite[Proposition 9]{Martin96} and \cite[Theorem 1.1]{DW00}.

\begin{prop} \label{MartinBlocks}
Let $\lambda \in \Lambda_{k}(\delta)$, then $\mathcal{B}_{k}(\lambda) = \mathcal{C}_{k}(\lambda)$.
\end{prop}   

We now wish to give an alternative criterion to describe the block structure of $\Lambda_{k}(\delta)$. The partition algebras $\mathcal{A}_{2k}(\delta)$ were shown in \cite{Xi99} to be \emph{cellular algebras} in the sense of Graham and Lehrer in \cite{GL96}. For each $(\lambda,l) \in \hat{\mathcal{A}}_{2k}$ we have a \emph{cell module} $\Delta_{2k}^{(\lambda,l)}$, and for each $(\lambda,l) \in \Lambda_{k}(\delta)$ the cell module $\Delta_{2k}^{(\lambda,l)}$ has a simple head $D^{(\lambda,l)}_{2k}$. The set $\{D_{2k}^{(\lambda,l)} \hspace{1mm} | \hspace{1mm} (\lambda,l) \in \Lambda_{k}(\delta)\}$ gives a complete set of pairwise non-isomorphic simple $\mathcal{A}_{2k}(\delta)$-modules. Let
\[ \chi_{\lambda}:Z(\mathcal{A}_{2k}(\delta))\rightarrow \mathbb{C} \] 
be the central character induced by $(\lambda,l) \in \Lambda_{k}(\delta)$. When $\mathcal{A}_{2k}(\delta)$ is semisimple then $\Delta_{2k}^{(\lambda,l)} \cong D_{2k}^{(\lambda,l)}$ are precisely the simple modules discussed in the previous section. Also we know by \emph{\Cref{SSP=Center}} that $SS_{2k}[N_{1},\dots,N_{2k}]=Z(\mathcal{A}_{2k}(\delta))$, and by \emph{\Cref{HRJMAction}} and \emph{\Cref{EvalPathEquiv}} the characters $\chi_{\lambda}$ act on polynomials $p \in SS_{2k}[N_{1},\dots,N_{2k}]$ by evaluating them over the contents $(c_{\lambda}(1),\dots,c_{\lambda}(2k))$, that is $\chi_{\lambda}(p) = p(c_{\lambda}(1),\dots,c_{\lambda}(2k)) = p(\lambda)$ using the notation established in the previous section. From the discussion at the start of this section, in general we know that $\lambda,\mu \in \Lambda_{k}(\delta)$ are in the same block if and only if $\chi_{\lambda}=\chi_{\mu}$. In the semisimple case $\mathcal{B}_{k}(\lambda) = \{\lambda\}$, and we proved in \emph{\Cref{SSPDistSimples}} that
\[ \lambda=\mu \iff \chi_{\lambda}(l_{n})=\chi_{\mu}(l_{n}) \hspace{2mm} \forall n\geq 0 \iff \chi_{\lambda}=\chi_{\mu}. \]
When $\mathcal{A}_{2k}(\delta)$ is non-semisimple, we know that $SS_{2k}[N_{1},\dots,N_{2k}] \subseteq Z(\mathcal{A}_{2k}(\delta))$. Let $\chi_{\lambda}|_{SS}$ denote the restriction of $\chi_{\lambda}$ to $SS_{2k}[N_{1},\dots,N_{2k}]$. It turns out that $\lambda$ and $\mu$ belong to the same block if and only if $\chi_{\lambda}|_{SS}=\chi_{\mu}|_{SS}$. To prove this we first need to show that $\chi_{\lambda}|_{SS}$ acts by evaluation by contents, just like the semisimple case.

Let $R=\mathbb{C}[x]$ with indeterminate $x$. We can define the partition algebras over $R$ by letting $x$ play the role of $\delta$. Denote these $R$-algebras by $\mathcal{A}_{2k}^{R}(x)$. The $R$-algebra $\mathcal{A}_{2k}^{R}(x)$ is cellular with cell modules $\Delta_{2k,R}^{(\lambda,l)}$ for each $(\lambda,l) \in \hat{\mathcal{A}}_{2k}$. In \cite[Section 3]{Eny13} it was shown that each cell module $\Delta_{2k,R}^{(\lambda,l)}$ has a Murphy type basis $\{m_{T} \hspace{1mm} | \hspace{1mm} T \in \mathsf{Path}(\lambda,l)\}$ on which the normalised Jucys-Murphy elements act upper-triangularly (see Proposition 3.15 of \cite{Eny13}). As an operator on $\Delta_{2k,R}^{(\lambda,l)}$, the eigenvalues of $N_{i}$ are the contents cont$^{R}(T,i)$ for all $T\in \mathsf{Path}(\lambda,l)$. Here cont$^{R}(T,i)$  is the same as \emph{\Cref{PathContDef}} except $\delta$ is replaced by $x$. Now let $\delta \in \mathbb{C}$ and $I(\delta)$ be the ideal of $R$ generated by the polynomial $x-\delta$. Set $\mathbb{C}(\delta) := R/I(\delta) \cong \mathbb{C}$ and let
\[ \mathcal{A}_{2k}^{\mathbb{C}(\delta)}(x) = \mathbb{C}(\delta)\otimes_{R}\mathcal{A}_{2k}^{R}(x), \hspace{5mm} \Delta_{2k,\mathbb{C}(\delta)}^{(\lambda,l)} = \mathbb{C}(\delta)\otimes_{R}\Delta_{2k,R}^{(\lambda,l)}. \]
As $\mathbb{C}$-algebras $\mathcal{A}_{2k}^{\mathbb{C}(\delta)}(x) \cong \mathcal{A}_{2k}(\delta)$ and the cell modules agree, that is $ \Delta_{2k,\mathbb{C}(\delta)}^{(\lambda,l)} =  \Delta_{2k}^{(\lambda,l)}$. The set $\{1\otimes m_{T} \hspace{1mm} | \hspace{1mm} T \in \mathsf{Path}(\lambda,l)\}$ is a Murphy type basis for $\Delta_{2k}^{(\lambda,l)}$. Also as operators on $\Delta_{2k}^{(\lambda,l)}$, the eigenvalues of $N_{i}$ are the images of cont$^{R}(T,i)$ under the projection $R\rightarrow R/I(\delta)$, as such they are precisely cont$(T,i)$ as given in \emph{\Cref{PathContDef}} for all $T \in \mathsf{Path}(\lambda,l)$. For each $(\lambda,l) \in \Lambda_{k}(\delta)$, there exist some submodule $M \subset \Delta_{2k}^{(\lambda,l)}$ such that $D_{2k}^{(\lambda,l)} \cong \Delta_{2k}^{(\lambda,l)}/M$. By Schur's Lemma End$_{\mathcal{A}_{2k}(\delta)}(D_{2k}^{(\lambda,l)}) \cong \mathbb{C}$, hence each element of the center acts on $D_{2k}^{(\lambda,l)}$ by a constant. Since each $p \in SS_{2k}[N_{1},\dots,N_{2k}]$ is central in $\mathcal{A}_{2k}(\delta)$, and since each $N_{i}$ acts upper-triangularly with eigenvalues given by contents, one can deduce that
\[ p(N_{1},\dots,N_{2k})(1\otimes m_{T} + M) = p(c_{\lambda}(1),\dots,c_{\lambda}(2k))(1\otimes m_{T} + M), \]
for all $T \in \mathsf{Path}(\lambda,l)$. As such $\chi_{\lambda}|_{SS}$ acts by evaluation by contents even in the non-semisimple case. To show that $\chi_{\lambda}|_{SS}=\chi_{\mu}|_{SS}$ if and only if $\lambda$ and $\mu$ belong to the same block, we first define a generating function to express the information of the character $\chi_{\lambda}|_{SS}$.  

\begin{defn} \label{SSBalDef}
Let $\delta \in \mathbb{C}$ and $(\lambda,l) \in \Lambda_{k}(\delta)$. Let $t$ be a formal variable, then we set 
\[ \lambda(t) := \frac{\prod_{i=0}^{k-l-1}(1+(i-\frac{\delta}{2})t)}{\prod_{a \in \lambda}(1+(\frac{\delta}{2}-c(a))t)}. \] 
For convention we set $\emptyset(t)=1$. 
\end{defn}

Recall from the previous section that $l_{n}(\lambda)$ denoted the evaluation of the $n$-th elementary supersymmetric polynomial with respect to the content vector $(c_{\lambda}(1),\dots,c_{\lambda}(2k))$, where $c_{\lambda}(i):=\text{cont}(T^{(\lambda,l)},i)$. From above we know that $\chi_{\lambda}(l_{n})=l_{n}(\lambda)$. Then $\lambda(t)$ in the above definition is the generating function whose $t^{n}$ coefficient is $\chi_{\lambda}(l_{n})$ (see \emph{\Cref{EleSSPDef}} and the discussion preceding (Eq3)). Since $l_{n}$ generates all of $SS_{2k}[x]$, the generating function $\lambda(t)$ is the same data as $\chi_{\lambda}|_{SS}$, in particular $\lambda(t)=\mu(t)$ if and only if $\chi_{\lambda}|_{SS}=\chi_{\mu}|_{SS}$.

\begin{ex} \label{SSBalEx}
Let $\delta = 1$ and consider the partitions $(\emptyset,3),(\PART{2},1),(\PART{2,1},0) \in \Lambda_{3}(\delta)$ from \emph{\Cref{DeltaPairEx}}. Then we have $\emptyset(t) = 1$ and
\begin{align*}
\PART{2}(t) &= \frac{(1+(0-\frac{1}{2})t)(1+(1-\frac{1}{2})t)}{(1+(\frac{1}{2}-0)t)(1+(\frac{1}{2}-1)t)} = 1, \\
\PART{2,1}(t) &= \frac{(1+(0-\frac{1}{2})t)(1+(1-\frac{1}{2})t)(1+(2-\frac{1}{2})t)}{(1+(\frac{1}{2}-0)t)(1+(\frac{1}{2}-1)t)(1+(\frac{1}{2}+1)t)} = 1.
\end{align*}
Hence $\emptyset(t)=\PART{2}(t)=\PART{2,1}(t)$.
\end{ex}

The above example tells us that, for any $(\lambda,l) \in \{(\emptyset,3),(\PART{2},1),(\PART{2,1},0)\}$, we have $l_{n}(\lambda)=0$ for all $n\geq1$, and as such all non-constant elements of $SS_{6}[N_{1},\dots,N_{6}]$ act on $D_{6}^{\lambda}$ by 0. \emph{\Cref{DeltaPairEx}} showed that the partitions in the above example belong to the same maximal $1$-chain, hence belong to the same block. We seek to show that two partitions $\lambda,\mu \in \Lambda_{k}(\delta)$ belong to the same block of $\mathcal{A}_{2k}(\delta)$ if and only if $\lambda(t)=\mu(t)$. This will tell us that $SS_{2k}[N_{1},\dots,N_{2k}]$ can distinguish between the blocks of $\mathcal{A}_{2k}(\delta)$. We will prove this by showing that $\lambda(t)=\mu(t)$ is equivalent to $\mu$ and $\lambda$ belonging to the same maximal $\delta$-chain.

\begin{lem} \label{DeltaPairAreSSBal}
Let $(\lambda,l),(\mu,m) \in \Lambda_{k}(\delta)$. If $(\mu,\lambda)$ is a $\delta$-pair, then $\mu(t) = \lambda(t)$.
\end{lem}

\begin{proof}
We have that $\lambda/\mu = R$ where $R$ is a horizontal strip. Let $R = \{b_{1},\dots,b_{n}\}$ where the boxes $b_{i}$ run from left to right as $i$ runs from 1 to $n$. Since $(\mu,\lambda)$ is a $\delta$-pair we have that $c(b_{n}) = \delta-|\mu|  =\delta-k+m$. As such,
\[ \lambda(t) = \frac{\prod_{i=0}^{k-l-1}(1+(i-\frac{\delta}{2})t)}{\prod_{a \in \lambda}(1+(\frac{\delta}{2}-c(a))t)} = \frac{\prod_{i=0}^{k-m-1}(1+(i-\frac{\delta}{2})t)\prod_{i=k-m}^{k-l-1}(1+(i-\frac{\delta}{2})t)}{\prod_{a \in \lambda\backslash R}(1+(\frac{\delta}{2}-c(a))t)\prod_{1\leq i \leq n}(1+(\frac{\delta}{2}-c(b_{i}))t)}. \]
We have that $|\lambda| = |\mu| + n$ and so $m=l+n$. Then reindexing gives
\[ \prod_{i=k-m}^{k-l-1}(1+(i-\delta/2)t) = \prod_{i=k-l-n}^{k-l-1}(1+(i-\delta/2)t) = \prod_{i=1}^{n}(1+(k-l-i-\delta/2)t). \tag{Eq8} \]
Since $R=\{b_{1},\dots,b_{n}\}$ consists of consecutive boxes in the same row, and $c(b_{n}) = \delta-k+m$, we see that
\begin{align*}
c(b_{i}) &= c(b_{n})-(n-i) \\
&= \delta-k+m-n+i \\
&= \delta-k+l+i
\end{align*}
where the last equality follows since $m-n=l$. This gives
\[ \prod_{1\leq i \leq n}(1+(\delta/2-c(b_{i}))t) = \prod_{1\leq i \leq n}(1+(\delta/2-(\delta-k+l+i))t) = \prod_{1\leq i \leq n}(1+(k-l-i-\delta/2)t). \tag{Eq9} \]
Thus (Eq8) and (Eq9) agree, and so these factors cancel in $\lambda(t)$. Hence, since $\lambda\backslash R=\mu$, we see that
\[ \lambda(t) = \frac{\prod_{i=0}^{k-m-1}(1+(i-\frac{\delta}{2})t)}{\prod_{a \in \lambda\backslash R}(1+(\frac{\delta}{2}-c(a))t)} = \mu(t). \]

\end{proof}

From this lemma, one can immediately see that if $\tau^{(0)} \subset \tau^{(1)} \subset \dots \subset \tau^{(y)}$ is a maximal $\delta$-chain, then $\tau^{(i)}(t)=\tau^{(j)}(t)$ for any $0\leq i,j\leq y$. This has given us the forward direction, that is
\[ \mu \in \mathcal{C}_{k}(\lambda) \implies \mu(t)=\lambda(t). \]
The other direction will following from the next two lemmas.

\begin{lem}
Let $(\lambda,l),(\mu,m) \in \Lambda_{k}(\delta)$ such that $\mu \subset \lambda$. If $\lambda(t)=\mu(t)$, then there exists a $\delta$-chain $\tau^{(x)} \subset \dots \subset \tau^{(y)}$ for some $x<y$ such that $\mu = \tau^{(x)}$ and $\lambda = \tau^{(y)}$. 
\end{lem}

\begin{proof}
Since $\mu \subset \lambda$, we may write $\lambda = \mu \cup_{i}R^{(i)}$ where $R^{(i)} := \{r_{1}^{(i)},r_{2}^{(i)},\dots,r_{n_{i}}^{(i)}\}$ is a horizontal strip of boxes of $\lambda$ in the $i$-th row, for each $1 \leq i \leq h(\lambda)+1$ and $n_{i} \in \mathbb{Z}_{\geq 0}$. Note some $R^{(i)}$ may be empty. We will prove the result by induction on the number of such non-empty horizontal strips $R^{(i)}$. For the base case, assume that $\mu$ and $\lambda$ differ by a single row, that is $\lambda\backslash \mu = R := \{r_{1},\dots,r_{n}\}$ for some $n \in \mathbb{N}$. Since $\mu\cup R = \lambda$, we have
\[ \lambda(t) = \frac{\prod_{i=0}^{k-l-1}(1+(i-\frac{\delta}{2})t)}{\prod_{a \in \lambda}(1+(\frac{\delta}{2}-c(a))t)} = \frac{\prod_{i=0}^{k-l-1}(1+(i-\frac{\delta}{2})t)}{\prod_{a \in \mu}(1+(\frac{\delta}{2}-c(a))t)\prod_{1 \leq i \leq n}(1+(\frac{\delta}{2}-c(r_{i}))t)}, \] 
and by definition we have
\[ \mu(t) = \frac{\prod_{i=0}^{k-m-1}(1+(i-\frac{\delta}{2})t)}{\prod_{a \in \mu}(1+(\frac{\delta}{2}-c(a))t)}. \]
By assumption we have that $\lambda(t)=\mu(t)$, and so we can deduce that
\[ \frac{\prod_{i=k-m}^{k-l-1}(1+(i-\frac{\delta}{2})t)}{\prod_{1 \leq i \leq n}(1+(\frac{\delta}{2}-c(r_{i}))t)} = 1. \]
Note that there is $n=m-l$ irreducible factors in the numerator and denominator of above. Since the fraction equals 1, the factors in the numerator must match up one-to-one with the factors in the denominator. The box $r_{n}$ has the largest content of all the boxes in $R$, and so the factor $(1+(\frac{\delta}{2}-c(r_{n}))t)$ has the smallest coefficient of $t$ out of all the factors in the denominator. As such, this factor must cancel out the the factor $(1+(k-m-\frac{\delta}{2})t)$ in the numerator, since this has the smallest coefficient of $t$ among the factors in the numerator. Equating these coefficients yields
\[ c(r_{n}) = \delta-k+m = \delta - |\mu|. \]
Hence $\mu \hookrightarrow_{\delta} \lambda$ proving the base case. Now assume the result holds if $\lambda$ differs from $\mu$ by $s>1$ strips. We seek to prove the $s+1$ case. So suppose that $\mu \cup_{i \in I}R^{(i)} = \lambda$ where $R^{(i)} = \{r_{1}^{(i)},\dots,r_{n_{i}}^{(i)}\}$ is a non-empty strip of boxes in the $i$-th row of $\lambda$ and $I \subset \{1,2,\dots,h(\lambda)+1\}$ with $|I|=s+1$. We have that
\[ \lambda(t) =  \frac{\prod_{i=0}^{k-l-1}(1+(i-\frac{\delta}{2})t)}{\prod_{a \in \mu}(1+(\frac{\delta}{2}-c(a))t)\prod_{i \in I}\prod_{1 \leq j \leq n_{i}}(1+(\frac{\delta}{2}-c(r_{j}^{(i)}))t)}.\] 
By assumption we have that $\lambda(t)=\mu(t)$, and so one can deduce that
\[ \frac{\prod_{i=k-m}^{k-l-1}(1+(i-\frac{\delta}{2})t)}{\prod_{i \in I}\prod_{1 \leq j \leq n_{i}}(1+(\frac{\delta}{2}-c(r_{j}^{(i)}))t)} = 1. \]
Note, as was the case previously, the number of irreducible factors in the numerator agrees with that of the denominator. Also the coefficients of $t$ in the irreducible factors in the numerator are all distinct, and differ to their neighbours by $\pm 1$. Since the fraction equals 1, and the number of factors on top and bottom agree, the same must be true for the coefficients of $t$ in the denominator.
This tells us that the skew diagram $\lambda\backslash \mu = \cup_{i}R^{(i)}$ is a skew hook, and in particular the strips $R^{(i)}$ must occur among consecutive rows, that is $I = \{x,x+1,\dots,y\}$ for some $1\leq x<y\leq h(\lambda)+1$. The strip $R^{(y)}$ is within the lowest row of $[\lambda]$ among the rows indexed by $I$, and so the box $r_{1}^{(y)}$ has the smallest content among the boxes $r \in \cup_{i}R^{(i)}$. As such the factor $(1+(\delta/2 - c(r_{1}^{(y)})t)$ has the largest coefficient of $t$ among the irreducible factors in the denominator, hence this factor must cancel out with $(1+(k-l-1-\frac{\delta}{2})t)$, since this is the irreducible factor with the largest coefficient of $t$ in the numerator. Therefore we must have that
\[ c(r_{1}^{(y)}) = \delta - (k-l-1) = \delta-|\lambda|+1. \]
The last box $r_{n_{y}}^{(y)}$ in $R^{(y)}$ has content $c(r_{n_{y}}^{(y)}) = c(r_{1}^{(y)})+n_{y}-1$, and so we have that
\[ c(r_{n_{y}}^{(y)}) = \delta-|\lambda|+1 +n_{y}-1 = \delta -|\lambda|+n_{y} = \delta - |\lambda\backslash R^{(y)}|. \]
Therefore $\lambda\backslash R^{(y)} \hookrightarrow_{\delta} \lambda$, and thus by \emph{\Cref{DeltaPairAreSSBal}} we have that $\lambda(t) = (\lambda\backslash R^{(y)})(t)$. This means also that $\mu(t) = (\lambda\backslash R^{(y)})(t)$ with $\mu \subset \lambda\backslash R^{(y)}$, and so by the inductive hypothesis there exists a $\delta$-chain $\tau^{(x)} \subset \dots \subset \tau^{(y-1)}$ with $1\leq x<y-1\leq h(\lambda\backslash R^{(y)})+1$ such that $\tau^{(x)} = \mu$ and $\tau^{(y-1)}=\lambda\backslash R^{(y)}$. Since  $\lambda\backslash R^{(y)} \hookrightarrow_{\delta} \lambda$, we can extend this chain by adding $\tau^{(y)} = \lambda$, which completes the proof.

\end{proof}

\begin{lem}
Let $(\lambda,l),(\mu,m) \in \Lambda_{k}(\delta)$. If $\lambda(t)=\mu(t)$, then either $\mu \subset \lambda$ or $\lambda \subset \mu$.
\end{lem}

\begin{proof}
Assume for contradiction that $\lambda(t)=\mu(t)$ but $\lambda \not\subset \mu$ and $\mu \not\subset \lambda$. Let $\tau = \lambda\cap\mu$. We have
\begin{align*}
\lambda(t) = \frac{\prod_{i=0}^{k-l-1}(1+(i-\frac{\delta}{2})t)}{\prod_{a \in \lambda}(1+(\frac{\delta}{2}-c(a))t)} = \frac{\prod_{i=0}^{k-l-1}(1+(i-\frac{\delta}{2})t)}{\prod_{a \in \tau}(1+(\frac{\delta}{2}-c(a))t)\prod_{a \in \lambda/ \tau}(1+(\frac{\delta}{2}-c(a))t)},
\end{align*}
and similarly
\begin{align*}
\mu(t) = \frac{\prod_{i=0}^{k-m-1}(1+(i-\frac{\delta}{2})t)}{\prod_{a \in \mu}(1+(\frac{\delta}{2}-c(a))t)} = \frac{\prod_{i=0}^{k-m-1}(1+(i-\frac{\delta}{2})t)}{\prod_{a \in \tau}(1+(\frac{\delta}{2}-c(a))t)\prod_{a \in \mu/ \tau}(1+(\frac{\delta}{2}-c(a))t)}.
\end{align*}
Without loss of generality assume that $|\mu| \leq |\lambda|$, and so $m \geq l$. Then since $\lambda(t)=\mu(t)$, we have that
\[ \frac{\prod_{i=k-m}^{k-l-1}(1+(i-\frac{\delta}{2})t)}{\prod_{a \in \lambda/ \tau}(1+(\frac{\delta}{2}-c(a))t)} = \frac{1}{\prod_{a \in \mu/ \tau}(1+(\frac{\delta}{2}-c(a))t)}, \tag{Eq10} \]
which implies
\[ \prod_{i=k-m}^{k-l-1}(1+(i-\delta/2)t) = \frac{\prod_{a \in \lambda/ \tau}(1+(\frac{\delta}{2}-c(a))t)}{\prod_{a \in \mu/ \tau}(1+(\frac{\delta}{2}-c(a))t)}. \tag{Eq11} \]
This tells us that all the irreducible factors in the denominator of the right hand side of (Eq11) must cancel out with factors in the numerator, or be trivial. This means that the multiset of contents which do not equal $\delta/2$ of the skew diagram $\mu/ \tau$ is contained in the multiset of contents which do not equal $\delta/2$ of the skew diagram $\lambda/ \tau$. Since $(\lambda/ \tau) \cap (\mu/ \tau) = \emptyset$, these multisets are distinct, and so $\mu\backslash \tau$ must consists only of boxes with content $\delta/2$ and $\lambda/\tau$ has no boxes with content $\delta/2$. The only way $\mu/ \tau$ can consist solely of boxes with content $\delta/2$ is if $|\mu/ \tau| = 1$. So let $\mu/ \tau = \{a\}$ where $c(a)=\delta/2$. Hence, from (Eq10) above,
\[ \frac{\prod_{i=k-m}^{k-l-1}(1+(i-\frac{\delta}{2})t)}{\prod_{a \in \lambda/ \tau}(1+(\frac{\delta}{2}-c(a))t)} = 1. \tag{Eq12} \]
Note, the number of irreducible factors in the numerator of (Eq12) equals $m-l=|\lambda| - |\mu|$, while there is $|\lambda/ \tau| = |\lambda|-|\mu|+1$ (since $|\mu|=|\tau|+1$) irreducible factors in the denominator. Thus for the equality of (Eq12) to hold, one of the factors in the denominator must equal 1, i.e. there must exist some $a \in \lambda/ \tau$ such that $c(a)=\delta/2$. However, as mentioned above, this cannot occur, giving the desired contradiction.

\end{proof}

\begin{cor} \label{AltBlockCri}
The partitions $\lambda$ and $\mu$ are within the same block of $\Lambda_{k}(\delta)$ if and only if $\lambda(t)=\mu(t)$.
\end{cor}


\vspace*{\fill}
\noindent
School of Mathematics, Computer Science and Engineering, City, University of London \\
\textit{Email address}: Samuel.Creedon@city.ac.uk


\end{document}